\documentclass{amsart}
\usepackage{graphicx}
\usepackage[all]{xy}
\usepackage{amscd}
\usepackage{amssymb}
\usepackage{amsmath}
\usepackage{amsthm}
\usepackage{amsfonts}
\usepackage{graphics}
\usepackage{epsfig,psfrag}
\usepackage[english]{babel}
\usepackage{latexsym}
\usepackage{mathrsfs}

\newtheorem{theorem}{Theorem}[section]

\newtheorem{proposition}[theorem]{Proposition}

\theoremstyle{definition}

\newtheorem{definition}[theorem]{Definition}

\theoremstyle{remark}
\newtheorem{remark}[theorem]{Remark}

\begin{document}

\title[A lower bound of the crossing number of composite knots]
{A lower bound of the crossing number of composite knots}

\author{Ruifeng Qiu}
\address{School of Mathematical Sciences, Key Laboratory of MEA(Ministry of Education) \& Shanghai Key Laboratory of PMMP, East China Normal University, Shanghai 200241, China}
\email{rfqiu@math.ecnu.edu.cn}

\author{Chao Wang}
\address{School of Mathematical Sciences, Key Laboratory of MEA(Ministry of Education) \& Shanghai Key Laboratory of PMMP, East China Normal University, Shanghai 200241, China}
\email{chao\_{}wang\_{}1987@126.com}

\subjclass[2020]{Primary 57K10; Secondary 57K30}

\keywords{crossing number, composite knot, handle structure, normal surface}

\thanks{The authors were supported by National Natural Science Foundation of China (NSFC), grant Nos. 12131009 and 12371067, and Science and Technology Commission of Shanghai Municipality (STCSM), grant No. 22DZ2229014.}

\begin{abstract}
Let $c(K)$ denote the crossing number of a knot $K$, and let $K_1\#K_2$ denote the connected sum of two oriented knots $K_1$ and $K_2$. As a famous old question in knot theory, whether $c(K_1\#K_2)=c(K_1)+c(K_2)$ is still unsolved. In this paper, we show that for any nontrivial knot $K=K_1\#\cdots\#K_n$, where $K_1,\ldots,K_n$ are oriented knots, the following inequality holds
\[c(K)>\frac{1}{16}(c(K_1)+\cdots+c(K_n)).\]
The result improves a well-known lower bound of $c(K)$ given by Lackenby.
\end{abstract}

\date{}
\maketitle

\tableofcontents


\section{Introduction}\label{sec:Intro}
Let $c(K)$ denote the crossing number of a knot $K$, and let $K_1\#K_2$ denote the connected sum of two oriented knots $K_1$ and $K_2$. A famous old question in knot theory asks whether $c(K_1\#K_2)=c(K_1)+c(K_2)$ always holds. By the definition of $K_1\#K_2$, it is easy to see that $c(K_1\#K_2)\leq c(K_1)+c(K_2)$ always holds. However, it seems that the inequality $c(K_1\#K_2)\geq c(K_1)+c(K_2)$ is quite hard to establish. At present, we know that it is true for some special classes of knots, including the alternating knots \cite{Ka, Mu, Th1}, its generalization the adequate knots and certain related ones \cite{LT, Th2, KL, KM}, and the zero deficiency knots \cite{Di}, which contains the torus knots. For the general knots, two kinds of lower bounds of $c(K_1\#K_2)$ in terms of $c(K_1)$ and $c(K_2)$ have been obtained in \cite{La} and \cite{It}, respectively, where the bound in \cite{La} is linear, while the bound in \cite{It} depends on the braid indices of $K_1$ and $K_2$. The latter is better than the former if the braid indices of $K_1$ and $K_2$ are small. On the other hand, in \cite{La}, it is actually proved that
\[c(K_1\#\cdots\#K_n)\geq\frac{1}{152}(c(K_1)+\cdots+c(K_n)),\]
where $K_1,\ldots,K_n$ are oriented knots. The aim of the present paper is to obtain a better lower bound of this kind, and this bound also improves the results in \cite{It} in certain cases. Our main result is the following.

\begin{theorem}\label{thm:main}
For a nontrivial knot $K=K_1\#\cdots\#K_n$, where $K_1,\ldots,K_n$ are oriented knots, the following inequality holds
\[c(K)>\frac{1}{16}(c(K_1)+\cdots+c(K_n)).\]
\end{theorem}

The proof of Theorem~\ref{thm:main} uses a strategy similar to the one in \cite{La}. However, to obtain the number $16$, which is much smaller than $152$, we need to change the basic settings and develop various new techniques. The ideas and results in \cite{La} are also used in \cite{La2} to show that for a satellite knot $K$ with companion knot $L$, the inequality $c(K)\geq c(L)/10^{13}$ always holds. Now, by using the methods in this paper, the results in \cite{La2} can also be improved, but the number that can replace $10^{13}$ will still be very large, while a main goal in this case is to prove $c(K)\geq c(L)$; see Problem~1.67 in Kirby's problem list \cite{Ki}. Also, note that the knot $K_1\#K_2$ is a satellite knot with companion knot $K_1$, and whether $c(K_1\#K_2)\geq c(K_1)$ always holds is still unknown. We believe that the methods in this paper will have further applications to such questions related to knot diagrams.

Below we provide the outline of the proof of Theorem~\ref{thm:main} and the organization of the paper. One can compare the outline with the subsequent sections, where we will give the detailed explanation. We work in the piecewise-linear category. Some elementary facts on 3-manifolds and knots are needed; see \cite{He,Li,Mo,Ro}.

\vspace{8pt}

\noindent {\bf Outline of the proof.} We can assume that all $K_i$ are nontrivial and prime, and $n>1$. We regard $K$ as a knot in the 3-dimensional sphere $S^3=\mathbb{R}^3\cup\{\infty\}$.

According to the nontrivial prime decomposition of $K$, there are $n-1$ pairwise disjoint decomposition spheres in $S^3$. By cutting $S^3$ along these spheres, we obtain 3-manifolds $M_1,\ldots,M_n$, which correspond to $K_1,\ldots,K_n$ as follows.

\begin{figure}[h]
\includegraphics{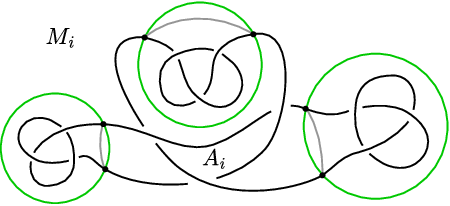}
\caption{Green circles indicate 2-spheres. The chosen arcs are grey.}\label{fig:cs4}
\end{figure}

For $1\leq i\leq n$, let $A_i=K\cap M_i$, which consists of some proper arcs in $M_i$. Each component of $\partial M_i$, which is a 2-sphere, contains exactly two points in $\partial A_i$. Then, to obtain a knot which is equivalent to some $K_j$, we only need to choose an arc in each component of $\partial M_i$ so that the arc connects the two points; see Figure~\ref{fig:cs4}. The knot is the union of the arcs and $A_i$, and it lies in $M_i$. We let $\widehat{A}_i$ denote the knot. We can assume that $\widehat{A}_i$ is equivalent to $K_i$ for $1\leq i\leq n$.

Let $D$ be a diagram of $K$ which has $c(K)$ crossings. We can assume that $D$ lies in a 2-sphere $S_0\subset S^3$, and $K$ agrees with $D$ except near the crossings, where near each crossing $K$ is a union of two arcs lying on different sides of $S_0$. By identifying a regular neighborhood of $S_0$ with $[-1,1]\times S_0$, we assume that $K\subset(-1,1)\times S_0$, and the projection $p:[-1,1]\times S_0\rightarrow S_0$ induces the projection of $K$ to $D$.

Let $\Sigma$ denote the union of the above $n-1$ decomposition spheres, and let $\mu_i$ be the union of the chosen arcs in $\partial M_i\subseteq\Sigma$ for $1\leq i\leq n$. Up to isotopy, we can also require that $\Sigma\subset(-1,1)\times S_0$, and $p(\widehat{A}_i)=p(\mu_i)\cup p(A_i)$ gives a diagram of $\widehat{A}_i$ for $1\leq i\leq n$, with respect to the projection map. We denote the diagram by $D_i$, and let $c(D_i)$ be the number of crossings in $D_i$. Then $c(K_i)\leq c(D_i)$ for $1\leq i\leq n$.

To prove the theorem, we want to choose suitable $\Sigma$ and $\mu_i$, $1\leq i\leq n$, so that
\[c(D_1)+\cdots+c(D_n)<16c(D),\]
where $c(D)$ is the number of crossings in $D$, which is equal to $c(K)$. We want that all the $D_i$ can be compared with $D$, and all the $c(D_i)$ are as small as possible. To show the above inequality and finish the proof, there are two steps.

\vspace{5pt}

\noindent {\bf Step~1.} Simplify a given $\Sigma$ so that it has a normal form with respect to $D$.

We will define a {\it handle structure} for the pair $(S^3,K)$, where the two 3-balls in $S^3\setminus(-1,1)\times S_0$ are the 3-handles, and $[-1,1]\times S_0$ is the union of the remaining handles. Each 2-handle, 1-handle, or 0-handle is a product of $[-1,1]$ with a region in $S_0$, and it corresponds to a face, an edge, or a crossing of $D$, respectively. Each 1-handle meets $K$ in an arc, and each 0-handle meets $K$ in two arcs.

Then we will define a kind of {\it normal surfaces} with respect to the above handle structure. Such a surface does not meet 3-handles, and it meets the other handles in some proper disks, which have certain shapes and are called {\it normal disks}. Each normal disk is a union of several {\it horizontal disks} and {\it vertical disks}. So the surface can be divided into the horizontal parts and the vertical parts.

\begin{figure}[h]
\includegraphics{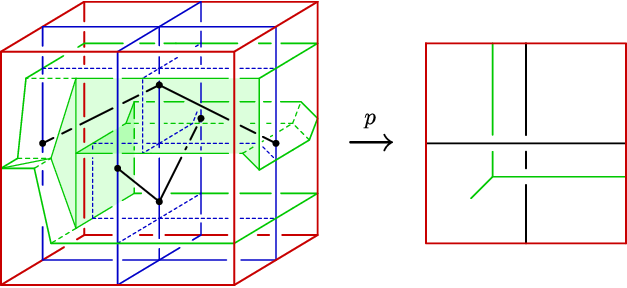}
\caption{Green disks are vertical. The other disks are horizontal.}\label{fig:example}
\end{figure}

A useful property of the normal surface is that locally if a normal disk contains vertical disks or meets $K$, then we can define a {\it diagram} of the disk by considering the projections of vertical disks and intersections with $K$, and the disk can always be reconstructed from the diagram. We also have a diagram for the surface. As an example, Figure~\ref{fig:example} shows a normal disk in a 0-handle. The disk is a union of three horizontal disks, two of which meet in an arc, and three vertical disks, which meet each other and give a figure ``$\textsf{Y}$'' in the projection diagram.

The main result in this step is that up to isotopy and a certain kind of {\it surgery}, any collection of prime decomposition spheres of $K$ can be converted into another such collection so that the union of the decomposition spheres is a normal surface. So we can assume that $\Sigma$ is a normal surface. Here a certain kind of {\it complexity} of surfaces is needed. We will also use it in Step~2 when modifying $(M_i,A_i)$.

\vspace{5pt}

\noindent {\bf Step~2.} Construct a handle structure for $(M_i,A_i)$ and use it to simplify $\mu_i$.

Let $\mathcal{H}_D$ denote the above handle structure for $(S^3,K)$. Since $\partial M_i$ is a union of normal disks, $M_i$ intersects each handle of $\mathcal{H}_D$ in some 3-balls. This will induce a {\it handle structure} $\mathcal{H}_{D,i}$ for $(M_i,A_i)$, where $j$-handles of $\mathcal{H}_{D,i}$ lie in $j$-handles of $\mathcal{H}_D$ for $0\leq j\leq 3$. Then we will modify $\mathcal{H}_{D,i}$ and construct a new handle structure for (a modified) $(M_i,A_i)$, which will be good enough to help us choose $\mu_i$.

Here the main idea is as follows. Since $\partial M_i$ meets the 2-handles of $\mathcal{H}_D$ in some disjoint normal disks, which do not meet $K$, we can require that $\mu_i$ does not meet the 2-handles of $\mathcal{H}_D$. A 1-handle or a 0-handle of $\mathcal{H}_D$ may contain a lot of normal disks in $\partial M_i$, and a normal disk may consist of many horizontal disks and vertical disks. Because we want that $c(D_i)$ is small, $\mu_i$ should avoid most of the horizontal disks and vertical disks in $\partial M_i$. For this purpose, and for the estimation of $c(D_i)$, we will cut 1-handles and 0-handles of $\mathcal{H}_{D,i}$ along $[-1,1]\times G$, where $G$ consists of certain arcs in $S_0$. So these handles will be divided into lots of {\it pieces}. Most of the pieces will be disjoint from $K$, and will be (naturally) homeomorphic to $I$-bundles over disks, where their intersections with the set $\partial M_i\cup\{\pm1\}\times S_0$ will correspond to the $\partial I$-bundles. We call them {\it parallel pieces}. The others are {\it nonparallel pieces}. Then, as in \cite{La}, each component of the union of all parallel pieces and 2-handles of $\mathcal{H}_{D,i}$ will be a trivial $I$-bundle over a sphere with holes, and it can be extended to a {\it generalized parallelity 2-handle} in $M_i$. On the other hand, the components of the union of all nonparallel pieces in a 1-handle or a 0-handle of $\mathcal{H}_{D,i}$ will be some 3-balls with peculiar shapes, and they have certain {\it models}. A 3-ball with model $X$ will be called an {\it $X$-piece}. All 1-handles and 0-handles of the new handle structure will be certain unions of these $X$-pieces. Then we will pick $\mu_i$ so that it is disjoint from the generalized parallelity 2-handles and certain new 1-handles.

\begin{figure}[h]
\includegraphics{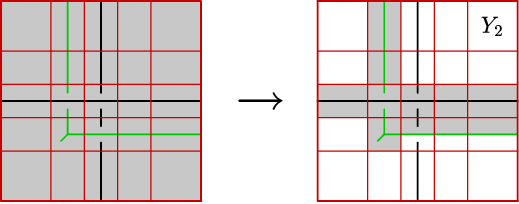}
\caption{Green arcs indicate vertical disks. Red arcs lie in $G$.}\label{fig:example2}
\end{figure}

As an example, assume that the normal disk shown in Figure~\ref{fig:example} lies in $\partial M_i$ and $M_i$ contains the upper arc. Let $H$ be the intersection of $M_i$ with the 0-handle. So $H$ is a 0-handle of $\mathcal{H}_{D,i}$. Up to isotopy, we can assume that the normal disk has a diagram as shown in Figure~\ref{fig:example2}, where the part of $G$ is also given. The shade in the left picture indicates the projection of $H$. Clearly, $[-1,1]\times G$ cuts the 0-handle of $\mathcal{H}_D$ into $25$ blocks. One can check that $H$ is divided into $37$ pieces. Among them, there are $26$ parallel pieces and $11$ nonparallel pieces. By removing all the parallel pieces from $H$, we obtain a 3-ball which meets $\partial M_i\cup\{\pm1\}\times S_0$ in two disks. We define the model of the 3-ball to be $Y_2$. The right picture in Figure~\ref{fig:example2} indicates the projection of the $Y_2$-piece. On the other hand, the parallel pieces and the adjacent 2-handles of $\mathcal{H}_D$ will be glued together. Also, when there are many parallel copies of the normal disk, all the pieces lying between the copies will be parallel.

To estimate $c(D_i)$, we will modify $\mu_i$ so that the intersections of it with all new 1-handles have {\it standard forms}, and all the possible crossings in $D_i$ come from the intersections of $\widehat{A}_i$ with those $X$-pieces in 0-handles of $\mathcal{H}_D$. Then, since we have a list of all models and we know exactly which $X$-pieces can appear in a 0-handle of $\mathcal{H}_D$ simultaneously, we will actually choose $\mu_i$, $1\leq i\leq n$, so that in each 0-handle of $\mathcal{H}_D$, $D_1\cup\cdots\cup D_n$ has at most $16$ crossings, and the number cannot be reached for all 0-handles. So $c(D_1)+\cdots+c(D_n)<16c(D)$.

Finally, we must mention the following points. Actually, when constructing the new handle structure, we will first classify the possible models. In this process, we will push $\partial M_i$ inward to eliminate some cases, and $(M_i,A_i)$ will become a smaller $(M_i',A_i')$. Then when constructing the generalized parallelity 2-handles, which will be the new 2-handles, we will need to remove some parts of $M_i'$, and $(M_i',A_i')$ will become an even smaller $(M_i'',A_i'')$ while $\widehat{A}_i''$ will still be equivalent to $\widehat{A}_i$. Also, we will need to modify $A_i''$ when estimating $c(D_i)$.

All technical details will be carefully explained in the subsequent sections.

\vspace{8pt}

\noindent {\bf Organization.} In Section~\ref{sec:MSS}, we first define {\it maximal sphere systems} of $K$. Such a system is just a collection of prime decomposition spheres of $K$, but from different points of view. Then we define an operation called {\it $D^2$-surgery}. It can convert one maximal sphere system into another and will be used throughout the proof.

In Section~\ref{sec:HS}, we first define the {\it handle structure} $\mathcal{H}_D$ for $(S^3,K)$. Then we give the definitions of {\it horizontal} and {\it vertical}, and the descriptions of the {\it normal disks}. Then we define the {\it normal surfaces} with respect to $\mathcal{H}_D$ and the {\it diagram} of such a surface. Finally, we define the {\it complexity} of certain surfaces in $[-1,1]\times S_0$. Then, in Section~\ref{sec:SF}, we prove the main result in Step~1, where we will simplify $\Sigma$ as much as possible. There will be {\it eight cases} for the part of $\Sigma$ in a 0-handle of $\mathcal{H}_D$.

In Section~\ref{sec:FDH}, we first define the {\it induced handle structure} $\mathcal{H}_{D,i}$ for $(M_i,A_i)$, and then we define {\it parallel pieces} and {\it nonparallel pieces}. Then, by examining the part of $\Sigma$ in a 1-handle of $\mathcal{H}_D$ and the above eight cases obtained in Section~4, we give the definitions and classifications of the {\it models} and {\it $X$-pieces}. Finally, we simplify the models and replace $(M_i,A_i)$ by $(M_i',A_i')$. Then, in Section~\ref{sec:GP}, we construct the {\it generalized parallelity 2-handles} and replace $(M_i',A_i')$ by $(M_i'',A_i'')$. We then define the new handle structure for $(M_i'',A_i'')$, where the generalized parallelity 2-handles are the new 2-handles, the new 1-handles consist of certain $X$-pieces, and the new 0-handles are certain $X$-pieces in the 0-handles of $\mathcal{H}_D$.

In Section~\ref{sec:EN}, we show how to choose the $\mu_i$, $1\leq i\leq n$, and finish the proof.

\newpage


\section{Maximal sphere systems and $D^2$-surgery}\label{sec:MSS}
In this section, we define maximal sphere systems of $K$ and an operation called $D^2$-surgery. We mainly consider the manifold pair having the form $(M,A)$, where $M$ is a connected compact 3-manifold embedded in $S^3$ whose boundary is a union of 2-spheres, and $A$ is a compact 1-manifold properly embedded in $M$. So we have $A\cap\partial M=\partial A$. We require that $A$ intersects any component of $\partial M$ in exactly two points, and if we add an arc in each component of $\partial M$ connecting the two points, then $A$ becomes a knot. This knot does not depend on the choices of those arcs in $\partial M$, since in a 2-sphere, any two arcs with common boundary are isotopic relative to the boundary. Below we use $\widehat{A}$ to denote this knot.

\begin{definition}\label{def:simple pair}
We call the pair $(M,A)$ {\it simple} if $\widehat{A}$ is a trivial knot. We call the pair $(M,A)$ {\it trivial} if it is simple and $\partial M$ is a 2-sphere.
\end{definition}

It is a basic fact that a 2-sphere embedded in $S^3$ is 2-sided and bounds a 3-ball on each side. Hence the complement of the interior of $M$ in $S^3$ is a union of some 3-balls. Up to isotopy, we can assume that the 3-balls are small tetrahedra, and $A$ consists of polygonal arcs whose edges are relatively much longer than those edges of the tetrahedra. Then it is easy to see that a simple pair in $S^3$ is isotopic to one of the manifold pairs shown in Figure~\ref{fig:simple pair}, where the number of small spheres can be any positive integer and $A$ lies in a straight line.

\begin{figure}[h]
\includegraphics{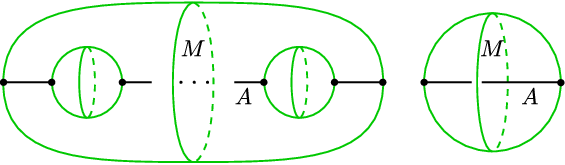}
\caption{A nontrivial simple pair and a trivial pair.}\label{fig:simple pair}
\end{figure}

Let $K=K_1\#\cdots\#K_n$ be the connected sum of oriented nontrivial prime knots $K_1,\ldots,K_n$ in $S^3$, where $n>1$. By the definition of connected sum, we have $n-1$ pairwise disjoint 2-spheres in $S^3$, where each of the spheres intersects $K$ in exactly two points. So by cutting $S^3$ along the spheres, we have pairs $(M_i,A_i)$, $1\leq i\leq n$, where $A_i=K\cap M_i$. We can assume that $\widehat{A}_i$ is equivalent to $K_i$. Then since $K_i$ is nontrivial and prime, $(M_i,A_i)$ is not simple, and if there is another 2-sphere lying in the interior of $M_i$ and intersecting $A_i$ transversely in two points, it must cut off a simple pair. So we give the following definition.

\begin{definition}\label{def:maxSS}
Let $\mathcal{S}=\{S_1,\ldots,S_m\}$ be a collection of pairwise disjoint spheres in $S^3$ so that each $S_j$ intersects $K$ transversely in two points. By cutting $S^3$ along the spheres there are pairs $(M_i,A_i)$, $1\leq i\leq m+1$. If each $(M_i,A_i)$ is nonsimple, and any properly embedded sphere in $M_i$ that meets $A_i$ transversely in two points cuts off a simple pair, then we call $\mathcal{S}$ a {\it maximal sphere system} of $K$.
\end{definition}

In general, from a prime decomposition of $K$ we can construct various maximal sphere systems. Conversely, let $\mathcal{S}$ be a maximal sphere system of $K$, then we have a decomposition $K=\widehat{A}_1\#\cdots\#\widehat{A}_{m+1}$, and each $\widehat{A}_i$ is nontrivial and prime. So by the prime decomposition theorem of knots we always have $m=n-1$, and we can assume that $\widehat{A}_i$ is equivalent to $K_i$ for $1\leq i\leq n$. Hence any $\mathcal{S}$ gives us a concrete correspondence between the pairs $(M_i,A_i)$ and the knots $K_i$.

\vspace{8pt}

Below we define an operation which can convert one system into another.

Let $\mathcal{S}=\{S_1,\ldots,S_{n-1}\}$ be a maximal sphere system of $K$. Let $(M_i,A_i)$ be the pair corresponding to $K_i$, $1\leq i\leq n$. Assume that we have a disk $D^2$ embedded in $M_i$ so that $D^2\cap\partial M_i=\partial D^2$. Then $\partial D^2$ lies in some $S_j$, and it splits $S_j$ into two disks $D^2_1$ and $D^2_2$, which give two spheres $D^2_1\cup D^2$ and $D^2_2\cup D^2$. We require that $D^2\cap A_i=\emptyset$ or $D^2$ meets $A_i$ transversely in one point in the interior of $M_i$.

\begin{figure}[h]
\includegraphics{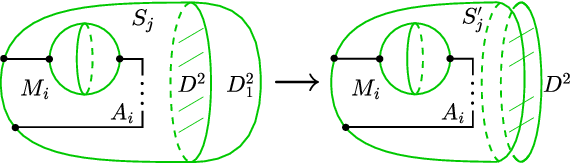}
\caption{$S_j$ can be isotoped to $S_j'$ via the $B^3$ bounded by $D^2_1\cup D^2$.}\label{fig:surgery1}
\end{figure}

\noindent {\bf Case 1.} $D^2\cap A_i=\emptyset$.

Since $S_j$ meets $K$ in two points, one of $D^2_1\cup D^2$ and $D^2_2\cup D^2$ does not meet $K$. Assume that $D^2_1\cup D^2$ does not meet $K$, then it bounds a 3-ball $B^3$ that also does not meet $K$. Hence $B^3$ does not contain $D^2_2$, and it lies in $M_i$. By doing a surgery to $S_j$ along $D^2$, we get a new sphere $S_j'$ disjoint from $D^2_1\cup D^2$. Because $S_j$ can be isotoped to $S_j'$ via $B^3$, $\mathcal{S}'=\mathcal{S}\cup\{S_j'\}\setminus\{S_j\}$ gives a new system; see Figure~\ref{fig:surgery1}, where there can be any number of small spheres and $A_i$ can be complicated.

\begin{figure}[h]
\includegraphics{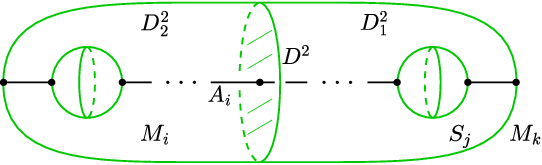}
\caption{$D^2$ splits $(M_i,A_i)$ into a nonsimple pair and a simple pair.}\label{fig:surgery2}
\end{figure}

\noindent {\bf Case 2.} $D^2\cap A_i\neq\emptyset$.

Since $S_j$ meets $K$ in two points, each of $D^2_1\cup D^2$ and $D^2_2\cup D^2$ meets $K$ in two points. Since $M_i$ lies in a 3-ball bounded by $S_j$, $D^2$ splits $(M_i,A_i)$ into two pairs. Let $(M_{i,r},A_{i,r})$ be the pair so that $M_{i,r}$ contains $D^2_r$, where $r=1,2$. By choosing suitable arcs in $\partial M_i\cup D^2$ that connect the points in $K$, it is not hard to see that $\widehat{A}_i=\widehat{A}_{i,1}\#\widehat{A}_{i,2}$. Then since $\widehat{A}_i$ is nontrivial and prime, exactly one of $(M_{i,1},A_{i,1})$ and $(M_{i,2},A_{i,2})$ is simple. Assume that $(M_{i,1},A_{i,1})$ is simple. By doing a surgery to $S_j$ along $D^2$, we get a new sphere $S_j'$ disjoint from $D^2_1\cup D^2$. Consider the pair $(M_k,A_k)$ so that $M_k\cap M_i=S_j$. It meets $(M_{i,1},A_{i,1})$ in $D^2_1$. By choosing suitable arcs as before, we see that $(M_k\cup M_{i,1},A_k\cup A_{i,1})$ is a nonsimple pair. Because $S_j'$ is essentially $D^2_2\cup D^2$, we obtain a new system $\mathcal{S}'=\mathcal{S}\cup\{S_j'\}\setminus\{S_j\}$; see Figure~\ref{fig:surgery2}, where there can be any number of small spheres and $A_i$ can be complicated.

\begin{definition}\label{def:Dsurgery}
The maximal sphere system $\mathcal{S}'$ obtained from $\mathcal{S}$ is determined by the disk $D^2$. We call this operation converting $\mathcal{S}$ into $\mathcal{S}'$ a {\it $D^2$-surgery}.
\end{definition}

\begin{remark}\label{rem:corr}
In above Case 1 and Case 2, only the pairs that meet $S_j$ will change, however, their corresponding knots do not change. Hence the $D^2$-surgery does not affect the correspondence between the pairs $(M_i,A_i)$ and the knots $K_i$.
\end{remark}

In later sections, we will usually consider the intersections between spheres in $\mathcal{S}$ and another surface. If in that surface we find a disk $D^2$ satisfying the conditions, then we will use the $D^2$-surgery to reduce the intersections.


\section{Handle structure and normal surfaces}\label{sec:HS}
In this section, we first construct a handle structure for $(S^3,K)$. Then, we give the definitions of normal surfaces and related concepts. Let $S_0$ be a fixed 2-sphere in $S^3$, and let $D\subset S_0$ be a diagram of $K$ which has $c(K)$ crossings. Here we view $D$ as a 4-valent graph in $S_0$ with over/under marks at each vertex. So the vertices represent the crossings. We fix a regular neighborhood of $S_0$ in $S^3$, and identify it with $[-1,1]\times S_0$. Then, we can assume that $K\subset(-1,1)\times S_0$, and the projection $p:[-1,1]\times S_0\rightarrow S_0$ induces the projection of $K$ to $D$.

\vspace{8pt}

Let $B_-$ and $B_+$ be the two 3-balls in $S^3\setminus(-1,1)\times S_0$ where $\partial B_-=\{-1\}\times S_0$ and $\partial B_+=\{1\}\times S_0$. The handle structure for $(S^3,K)$ will have these two 3-balls as the 3-handles, and $[-1,1]\times S_0$ will be the union of all the 2-handles, 1-handles, and 0-handles. Before defining these handles, we adjust $D$ and $K$ as follows.

\begin{figure}[h]
\includegraphics{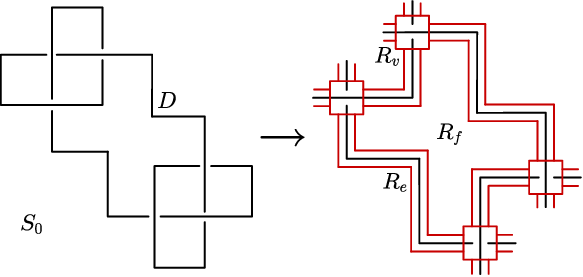}
\caption{The diagram $D\subset S_0$ and the regions $R_v$, $R_e$, and $R_f$.}\label{fig:diagram}
\end{figure}

We can identify $S_0$ with $\mathbb{R}^2\cup\{\infty\}$. Then, we can require that $D\subset\mathbb{R}^2$ and it is a union of straight arcs parallel to the $x$-axis or $y$-axis, where two of the arcs meet if and only if they meet in their endpoints. Consider the distance on $\mathbb{R}^2$ given by
\[\mathrm{dist}((x_1,y_1),(x_2,y_2))=\max\{|x_1-x_2|,|y_1-y_2|\}.\]
We can require that two disjoint straight arcs in $D$ have distance at least $3$. Then for each crossing $v$ in $D$, we have a square region in $S_0$ defined by
\[R_v=\{r\in\mathbb{R}^2\mid\mathrm{dist}(r,v)\leq 1\}.\]
For each polygonal arc $e$ in $D$ that connects two adjacent crossings $v_1$ and $v_2$, let
\[R_e=\{r\in\mathbb{R}^2\mid\mathrm{dist}(r,e)\leq \frac{1}{2},\, \mathrm{dist}(r,v_1)\geq 1,\, \mathrm{dist}(r,v_2)\geq 1\}.\]
Let $N$ be the union of all the $R_v$ and $R_e$. For each component $f$ of $S_0\setminus D$, let $R_f$ be the closure of $f\setminus N$. Then $S_0$ is decomposed into all these regions $R_v$, $R_e$, and $R_f$; see Figure~\ref{fig:diagram}. Then we can require that $K\cap [-1,1]\times R_e\subset D$, and each of the two arcs in $K\cap [-1,1]\times R_v$ consists of two straight arcs which meet at $(-1/2,v)$ or $(1/2,v)$. Let $A_e=K\cap [-1,1]\times R_e$ and $A_v=K\cap [-1,1]\times R_v$.

\begin{definition}\label{def:handlestr}
Let $\mathcal{H}^3$, $\mathcal{H}^2$, $\mathcal{H}^1$, and $\mathcal{H}^0$ denote the following four collections
\[\{B_-,B_+\},\{[-1,1]\times R_f\},\{([-1,1]\times R_e,A_e)\},\{([-1,1]\times R_v,A_v)\},\]
respectively, whose elements are defined to be 3-handles, 2-handles, 1-handles, and 0-handles. These handles together give a handle structure for $(S^3,K)$, and we call it a {\it $D$-structure} for $(S^3,K)$. We denote this $D$-structure by $\mathcal{H}_D$.

We call the handlebody $[-1,1]\times N$ the {\it mainbody} associated to $\mathcal{H}_D$ and denote it by $H_D$. We call $[-1,1]\times D$ the {\it skeleton} of $H_D$ and denote it by $S_D$.
\end{definition}

Figure~\ref{fig:handle} gives pictures of a 1-handle and a 0-handle, together with the parts of the skeleton in them. Note that in general $A_e$ is not a straight arc and $R_e$ is not a rectangle. However, it is not hard to construct a specific homeomorphism from $R_e$ to $[-1/2,1/2]\times[-1,1]$, mapping $A_e$ to $\{0\}\times[-1,1]$. So we will usually consider a 1-handle as the product $([-1,1]\times[-1/2,1/2],\{(0,0)\})\times[-1,1]$.

\begin{figure}[h]
\includegraphics{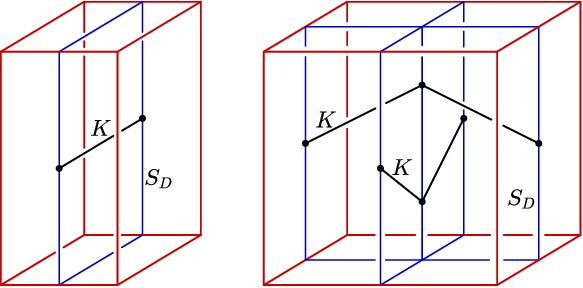}
\caption{A 1-handle, a 0-handle, and the parts of $S_D$ in them.}\label{fig:handle}
\end{figure}

By the construction of $\mathcal{H}_D$, it is easy to see that each 1-handle meets 0-handles in two disks, each 2-handle meets $H_D$ in an annulus, and two handles in the same $\mathcal{H}^i$ are disjoint. Since $D$ realizes $c(K)$, we also have the following.

\begin{proposition}\label{prop:goodhandle}
If two handles of $\mathcal{H}_D$ meet, then they meet in exactly one disk.
\end{proposition}

\begin{proof}
Clearly a 3-handle meets any other handle in at most one disk. Handles in $\mathcal{H}^2$, $\mathcal{H}^1$, and $\mathcal{H}^0$ correspond to regions in $S_0$. Clearly $R_f$ cannot meet $R_e$ twice. If $R_f$ meets $R_v$ twice, then in $R_f\cup R_v\subset S_0$ there exists a circle which meets $D$ only in $v$. So the crossing $v$ can be removed. Since $D$ is minimal, this does not happen. Similarly, since $D$ is minimal, $R_e$ cannot meet $R_v$ twice.
\end{proof}

Below we define the normal surfaces with respect to $\mathcal{H}_D$. Such a surface lies in $(-1,1)\times S_0$ and consists of normal disks, which are unions of horizontal disks and vertical disks. We first define these different kinds of disks.

\begin{definition}\label{def:Hor and Ver}
An arc in $(-1,1)\times S_0$ is called {\it horizontal} if under the projection $p:[-1,1]\times S_0\rightarrow S_0$ it is locally homeomorphic to its image. It is called {\it vertical} if it lies in $[-1,1]\times\{s\}$ for some point $s$ in $S_0$.

Similarly, a surface in $(-1,1)\times S_0$ is called {\it horizontal} if under the projection $p$ it is locally homeomorphic to its image. It is called {\it vertical} if it lies in $[-1,1]\times G$, where $G$ is some embedded graph in $S_0$.
\end{definition}

Consider an embedded disk $D^2$ in an $i$-handle $H^i$, where $0\leq i\leq 2$. We require that $D^2\cap\partial H^i=\partial D^2$, $\partial D^2\cap\{\pm1\}\times S_0=\emptyset$, and either $D^2\cap K=\emptyset$ or $D^2$ meets $K$ transversely in one point in the interior of $H^i$. In the following cases, we define certain classes of such disks. They will give the normal disks.

\vspace{5pt}

\noindent {\bf Case 1.} $i=2$ and $H^2$ is $[-1,1]\times R_f$.

There is one class, called the {\it flat disks}. The $D^2$ is horizontal, where $\partial D^2$ meets each $[-1,1]\times R_e$ in $\{t_e\}\times (R_f\cap R_e)$ for some $t_e$, and $\partial D^2\cap[-1,1]\times R_v$ is either empty or a union of two straight arcs. A typical example is $\{t\}\times R_f$.

\vspace{5pt}

\noindent {\bf Case 2.} $i=1$ and $H^1$ is identified with $([-1,1]\times[-1/2,1/2],\{(0,0)\})\times[-1,1]$.

There are two classes, called the {\it flat disks} and the {\it curved disks}. These disks do not meet $K$, and have the form $\alpha\times[-1,1]$ for some arc $\alpha$ in $[-1,1]\times[-1/2,1/2]$. The arc $\alpha$ is a union of three straight arcs. The $D^2$ is flat if $\alpha\times\{0\}$ is horizontal. The $D^2$ is curved if $\{\pm1\}\times[-1/2,1/2]$ and $\{(0,0)\}$ lie on different sides of $\alpha$, and $\alpha\times\{0\}$ is a union of one vertical arc and two horizontal arcs. See Figure~\ref{fig:surin1} for the examples of flat disks and curved disks, where the green disks are vertical. We can divide the flat disks into two subclasses according to whether $D^2\cap S_D$ is above or below $K$. The curved disks can also be divided into two smaller classes.

\begin{figure}[h]
\includegraphics{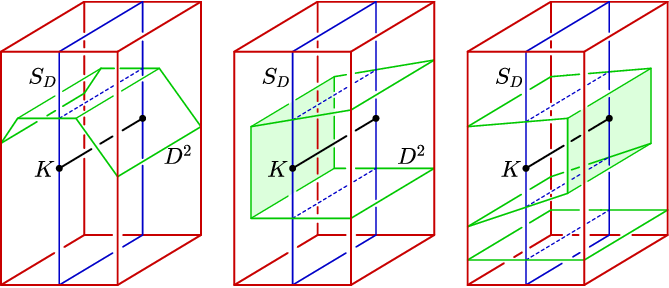}
\caption{A flat disk, a curved disk, and another two possibilities.}\label{fig:surin1}
\end{figure}

Note that in general $[-1,1]\times R_e$, the $D^2$ is not a product, but those horizontal disks and vertical disks in $D^2$ are preserved by the identification.

\vspace{5pt}

\noindent {\bf Case 3.} $i=0$ and $H^0$ is $([-1,1]\times R_v,A_v)$. Additionally, $D^2\cap K=\emptyset$.

There are three classes, called the {\it flat disks}, {\it curved disks}, and {\it twisted disks}. On certain pairs of these disks, there is also a {\it band sum} operation, which can produce infinitely many subclasses of these disks. The disks consist of horizontal disks and vertical disks, and their boundaries consist of the polygonal arcs which come from those flat disks and curved disks defined in Cases 1 and 2.

The $D^2$ is flat if it is horizontal. According to whether it meets $[-1,1]\times\{v\}$ in a point in $[-1,-1/2]\times\{v\}$, $[-1/2,1/2]\times\{v\}$, or $[1/2,1]\times\{v\}$, we can divide the flat disks into three subclasses which are denoted by $F_-$, $F_0$, and $F_+$, respectively. See Figure~\ref{fig:surF} for the examples of flat disks in these subclasses.

\begin{figure}[h]
\includegraphics{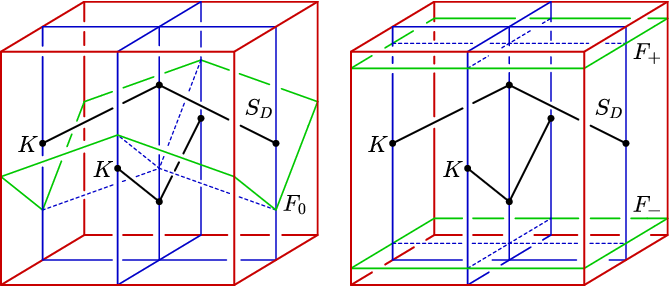}
\caption{Examples of flat disks in subclasses $F_-$, $F_0$, and $F_+$.}\label{fig:surF}
\end{figure}

The $D^2$ is curved if up to Euclidean isometries of $([-1,1]\times R_v,A_v)$, $D^2$ has the form
shown in one of the pictures in Figure~\ref{fig:surC}. The disk is a union of one vertical disk and two horizontal disks, where the vertical one is an $I$-bundle over a straight arc in $R_v$, and $\{\pm1\}\times R_v$ and some arc of $A_v$ lie on different sides of $D^2$. In fact, these properties can be used to define a curved disk. According to the positions of the intersections of $D^2$ with $[-1,1]\times\{v\}$, we can also divide the curved disks into three subclasses which are denoted by $C_-$, $C_0$, and $C_+$, respectively. In Figure~\ref{fig:surC}, an isometry interchanging the two arcs of $A_v$ will send the disk in $C_+$ to a disk in $C_-$. As in Case 2, we can further divide $C_0$ into four even smaller classes, and the isometries permute the disks in them. Similarly, we can divide $C_+$ (resp. $C_-$) into two smaller classes, which are differ by a $\pi$-rotation about $[-1,1]\times\{v\}$.

\begin{figure}[h]
\includegraphics{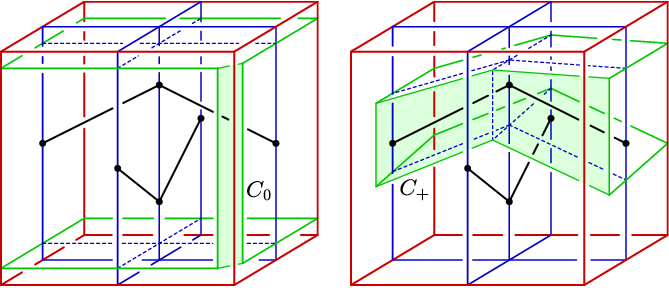}
\caption{Examples of curved disks in subclasses $C_0$ and $C_+$.}\label{fig:surC}
\end{figure}

Note that actually $\partial D^2$ in the left picture of Figure~\ref{fig:surC} cannot come from those disks defined in Cases 1 and 2. To ensure that $\partial D^2$ consists of required polygonal arcs, we should draw the vertical rectangle wider so that the vertical edges can lie in the adjacent 1-handles. The picture indicates that we can obtain the disk in $C_0$ from a disk in $F_-$ and a disk in $F_+$ by adding a band. Below there will be similar situations. The pictures will mainly show the structures of the disks.

The $D^2$ is twisted if up to Euclidean isometries of $([-1,1]\times R_v,A_v)$, it has the form shown in the left picture of Figure~\ref{fig:surT}. The disk is a union of three horizontal disks and three vertical disks. Two of the vertical ones are $I$-bundles over straight arcs in $R_v$, while the other vertical one is a triangle in the interior of $D^2$. Each of the $I$-bundles meets the vertical edge of the triangle in a fiber, and the three disks correspond to the three branches of a figure ``$\textsf{Y}$'', under $p:[-1,1]\times S_0\rightarrow S_0$. Two of the horizontal ones meet in a straight arc. It connects the vertex of the triangle not in the vertical edge to a point in $\partial D^2$. Also, $\{1\}\times R_v$ and $(-1/2,v)$ lie on the same side of $D^2$, while $\{-1\}\times R_v$ and $(1/2,v)$ lie on the other side. In fact, these properties can be used to define a twisted disk. Note that the union of those three vertical disks in $D^2$ is also a vertical disk. Then, according to its shape, we divide the twisted disks into two subclasses which are denoted by $T_-$ and $T_+$. The $D^2$ in the left picture of Figure~\ref{fig:surT} lies in $T_-$, while any orientation-reversing isometry of $([-1,1]\times R_v,A_v)$ will send it to a disk in $T_+$.

\begin{figure}[h]
\includegraphics{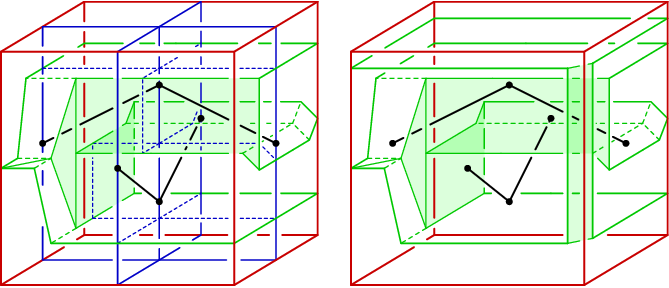}
\caption{A twisted disk in $T_-$ and its band sum with a disk in $F_+$.}\label{fig:surT}
\end{figure}

As in the case of $C_0$, we can add a band between a disk in $T_-$ and a disk in $F_+$ as shown in
the right picture of Figure~\ref{fig:surT}. Here a band is a vertical $I$-bundle over a straight arc in $R_v$ near a corner. Similarly, we can add a band between a disk in $T_-$ and a disk in $F_-$. Then we can add a disk in $F_-$ or $F_+$ further. In general, we can add $r$ disjoint disks in $F_+$ and $s$ disjoint disks in $F_-$ to the disk in $T_-$ by such band sum operations, provided that $|r-s|\leq 1$. We also regard this resulting disk as a twisted disk. Then such disks provide another subclass of twisted disks which is denoted by $T_-^{r,s}$. In a similar way, we also have a subclass $T_+^{r,s}$.

Note that each of the subclasses $T_-$, $T_+$, $T_-^{r,s}$, and $T_+^{r,s}$ can be divided into two smaller classes which are differ by a $\pi$-rotation about $[-1,1]\times\{v\}$, as in the cases of $C_-$ and $C_+$. Also, since boundaries of the twisted disks should be unions of the polygonal arcs that come from those disks defined in Cases 1 and 2, the horizontal disks in Figure~\ref{fig:surT} should be suitably isotoped and the band should be wider.

\vspace{5pt}

\noindent {\bf Case 4.} $i=0$ and $H^0$ is $([-1,1]\times R_v,A_v)$. Additionally, $D^2\cap K\neq\emptyset$.

There are two classes, called the {\it flat disks} and the {\it curved disks}. The disks have shapes similar to those in $F_{\pm}$, $C_0$, or $C_{\pm}$, and they do not meet $[-1/2,1/2]\times\{v\}$. As in Case 3, their boundaries consist of the polygonal arcs coming from the disks defined in Cases 1 and 2. The $D^2$ is flat if it is horizontal and the two intersection points of $D^2$ with $[-1,1]\times\{v\}$ and $K$ lie on different sides of $S_0$. Then according to whether it meets $[-1,-1/2]\times\{v\}$ or $[1/2,1]\times\{v\}$, we can divide the flat disks into two subclasses, denoted by $F_-^+$ and $F_+^-$, respectively. The $D^2$ is curved if it is a union of one vertical disk and two horizontal disks, where these disks satisfy the properties of the curved disks in Case 3, and each of the horizontal disks can only meet $[-1,1]\times\{v\}$ and $K$ on different sides of $S_0$. Then, according to the shape of $D^2$ and the position of the point $D^2\cap K$, we can divide the curved disks into four subclasses, denoted by $C_0^+$, $C_0^-$, $C_-^+$, and $C_+^-$, respectively.

\begin{figure}[h]
\includegraphics{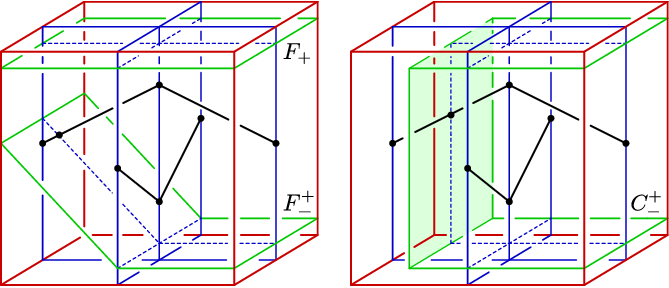}
\caption{A flat disk in $F_-^+$ and a curved disk in $C_-^+$.}\label{fig:surP}
\end{figure}

The left picture of Figure~\ref{fig:surP} shows a disk in $F_-^+$ and a disk in $F_+$. By adding a band between them as in the left picture of Figure~\ref{fig:surC}, we can obtain a disk in $C_0^+$. The right picture of Figure~\ref{fig:surP} shows a disk in $C_-^+$. An isometry interchanging the two arcs of $A_v$ will send a disk in $F_-^+$, $C_0^+$, or $C_-^+$ to some disk in $F_+^-$, $C_0^-$, or $C_+^-$, respectively. Each of the subclasses $F_-^+$, $F_+^-$, $C_-^+$, and $C_+^-$ can be divided into two smaller classes which are differ by a $\pi$-rotation about $[-1,1]\times\{v\}$, as in the cases of $C_-$ and $C_+$, and according to the position of the band, each of $C_0^+$ and $C_0^-$ can be divided into four smaller classes, as in the case of $C_0$. Finally, note that for $D^2$ in $C_-^+$ (resp. $C_+^-$), the point $D^2\cap K$ may lie in the lower (resp. upper) horizontal disk. For simplicity, we delete such $D^2$ from $C_-^+$ (resp. $C_+^-$).

\begin{definition}\label{def:normalS}
We call those flat, curved, and twisted disks defined in Cases 1-4 the {\it normal disks} with respect to $\mathcal{H}_D$.

Let $S$ be a closed surface in $S^3$, which meets $K$ transversely. Then, we call $S$ a {\it normal surface} with respect to $\mathcal{H}_D$ if it does not meet the 3-handles, and it meets any other handle in a union of disjoint normal disks.
\end{definition}

Below we first provide some properties of the normal disks. Then we define the diagrams of the normal disks and normal surfaces. Finally, we define a complexity of certain surfaces. To avoid tedious case-by-case arguments, we omit some proofs and only give the ideas. This does not affect the proof of Theorem~\ref{thm:main}.

\begin{proposition}\label{pro:distin}
Let $D^2_1$ and $D^2_2$ be normal disks in a handle $[-1,1]\times R$, where $R$ is some $R_e$ or $R_v$. If $D^2_1$ and $D^2_2$ do not lie in the same class (resp. subclass or smaller class), then $\partial D^2_1$ and $\partial D^2_2$ are not isotopic in $([-1,1]\times\partial R)\setminus K$.
\end{proposition}

To prove the proposition, we can use two elements in the fundamental group of $([-1,1]\times\partial R)\setminus K$ to represent $\partial D^2_1$ and $\partial D^2_2$, respectively, and show that the two elements are not conjugate in the free group. We can also try to choose a circle in $([-1,1]\times\partial R)\setminus K$ so that it has different geometric intersection numbers with the two circles $\partial D^2_1$ and $\partial D^2_2$. As an example, to distinguish two twisted disks $D^2_1$ and $D^2_2$, we can compute the geometric intersection numbers of $\partial D^2_1$ and $\partial D^2_2$ with the boundary circle of some disk in $F_-^+$ or $F_+^-$.

\begin{proposition}\label{pro:isotop}
Let $D^2_1$ and $D^2_2$ be normal disks in a handle $[-1,1]\times R$, where $R$ is some $R_e$ or $R_v$. If $D^2_1$ and $D^2_2$ lie in the same subclass and the same smaller class, then they are isotopic via a family of normal disks in the same subclass and the same smaller class, where the isotopy is ambient and keeps $K$ invariant.
\end{proposition}

To get a required isotopy of $[-1,1]\times R$, we can first adjust the vertical disks of $D^2_1$, so that their projections in $R$ coincide with the projections of vertical disks of $D^2_2$. Then, we can adjust the horizontal disks of $D^2_1$, keeping the projections fixed, so that $D^2_1$ coincides with $D^2_2$. We say that such $D^2_1$ and $D^2_2$ are {\it equivalent}. Then, by above propositions, there is a bijective correspondence between the equivalence classes of the normal disks in $[-1,1]\times R$ and the isotopy classes of the boundaries of normal disks in $([-1,1]\times\partial R)\setminus K$.

There is a similar correspondence for unions of normal disks. Let $\Psi_1$ and $\Psi_2$ be unions of disjoint normal disks in a handle $[-1,1]\times R$, where $R$ is some $R_e$ or $R_v$. For $i=1,2$, each component of $\partial\Psi_i$ in $([-1,1]\times\partial R)\setminus K$ defines an isotopy class. Let $\Phi_i$ be the set (with multiplicity) of isotopy classes given by the components of $\partial\Psi_i$. It is a basic fact that $\Phi_i$ determines $\partial\Psi_i$ up to isotopy of $([-1,1]\times\partial R)\setminus K$. Similar to above propositions, we have the following.

\begin{proposition}\label{pro:multiD}
The two sets (with multiplicity) $\Phi_1$ and $\Phi_2$ are the same if and only if the $\Psi_1$ and $\Psi_2$ are isotopic via a family of unions of disjoint normal disks, where the isotopy is ambient and keeps $K$ invariant.
\end{proposition}

The ``if'' part is clear. To show the ``only if'' part, we can get a required isotopy of $[-1,1]\times R$ as in the case of Proposition~\ref{pro:isotop}. Here when $R$ is some $R_e$, the case is easy, since the projections of the vertical disks are some parallel arcs; when $R$ is some $R_v$, we can show that the projections of the vertical disks in $\Psi_1$ and $\Psi_2$ give the same diagram in $R_v$; see Proposition~\ref{pro:diag}. Actually, we can first determine all the possible $\Phi_1=\Phi_2$ or, equivalently, the subclasses and smaller classes of normal disks in $\Psi_i$. For simplicity, we only provide the relations between the normal disks in $[-1,1]\times R_v$ which do not contain bands, as follows.

\begin{proposition}\label{prop:NdiskR}
We regard $F_-$, $F_0$, $F_+$, $C_-$, $C_+$, $T_-$, $T_+$, $F_-^+$, $F_+^-$, $C_-^+$, $C_+^-$ as $11$ vertices. If in a 0-handle, there are two disjoint disks belonging to two different vertices, respectively, then we add an edge between the two vertices. Then

(1) $F_-$ is adjacent to any other vertex, and so is $F_+$;

(2) the remaining edges are given in the graph in Figure~\ref{fig:graph1};

(3) for each triangle in the graph in Figure~\ref{fig:graph1}, there are three mutually disjoint disks in a 0-handle, which belong to the three vertices, respectively.
\end{proposition}

\begin{remark}
We can replace the subclasses other than $F_-$, $F_0$, and $F_+$ by smaller classes in them. Then we get $19$ vertices and the results still hold, except that the graph in Figure~\ref{fig:graph1} should be replaced by the graph in Figure~\ref{fig:graph2}. Here we use the symbol and its image under a $\pi$-rotation to denote the two smaller classes lying in the same subclass, since such smaller classes differ by a $\pi$-rotation.
\end{remark}

\begin{figure}[h]
\includegraphics{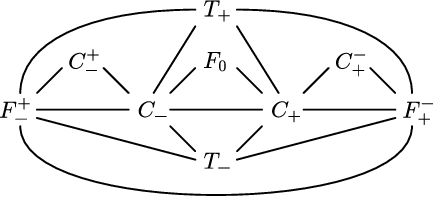}
\caption{Edges between the subclasses other than $F_-$ and $F_+$.}\label{fig:graph1}
\end{figure}

\begin{figure}[h]
\includegraphics{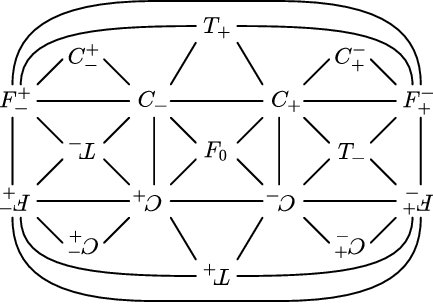}
\caption{Edges between the smaller classes together with $F_0$.}\label{fig:graph2}
\end{figure}

The existence of the edges can be checked case-by-case. The nonexistence of an edge can be proved by computing the geometric intersection numbers between the boundary circles of normal disks, as in the case of Proposition~\ref{pro:distin}. So, we see that if $\Psi_i$ contains no bands, then the possible $\Phi_i$ can be classified. Note that when we add a band between two normal disks, one of the disks belongs to $F_+$ or $F_-$. This means that, starting from such an uppermost disk in $F_+$ or lowermost disk in $F_-$, we can remove the possible bands in $\Psi_i$ inductively until it contains no bands and we meet some previous case of $\Phi_i$. Then, by reversing this process, the possible $\Phi_i$ when $\Psi_i$ contains bands can also be classified. We will provide parts of the results in the last step of the proof of Theorem~\ref{thm:normalForm}.

Now consider the vertical disks in $\Psi_i$ and $K\cap\Psi_i$. Let $G_i$ be the image of their union in $R$ under the projection $p:[-1,1]\times S_0\rightarrow S_0$, and assume that $G_i\neq\emptyset$. If $R$ is some $R_e$, then $G_i$ can only be a union of parallel arcs lying on one side of $A_e$. So below we mainly consider the case when $R$ is some $R_v$. If $K\cap\Psi_i\neq\emptyset$, then we mark the points $p(K\cap\Psi_i)$ in $G_i$. Then $G_i$ has the following form.

If $\Psi_i$ contains no bands, then by Proposition~\ref{prop:NdiskR}, $G_i$ can only be certain unions of the following graphs: the arcs coming from disks in $C_{\pm}$, the figures ``$\textsf{Y}$'' coming from disks in $T_{\pm}$, the points coming from disks in $F_-^+$ or $F_+^-$, and the arcs coming from disks in $C_-^+$ or $C_+^-$, where the arc contains a marked point. See Figure~\ref{fig:example3} for two examples of $G_i$ where we also give the over/under marks between the straight arcs indicating curved disks, the figure ``$\textsf{Y}$'', and the diagram $p(A_v)\subset D$.

\begin{figure}[h]
\includegraphics{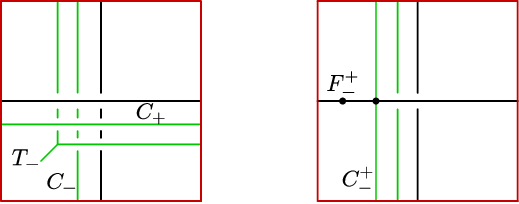}
\caption{Two examples of $G_i$. Black dots indicate $K\cap\Psi_i$.}\label{fig:example3}
\end{figure}

Note that the points in $G_i\cap\partial R_v$ correspond to the vertical arcs in $\partial\Psi_i$. These arcs also lie in adjacent 1-handles, so they have different images in $G_i\cap \partial R_v$. The intersection patterns between the above possible graphs in $G_i$ together with $p(A_v)$ can be determined by the positions of these vertical arcs and $K\cap\Psi_i$.

If $\Psi_i$ contains bands, then we can obtain $\Psi_i$ from some previous case by adding the bands inductively. Each time we add a band, a new straight arc will appear in $R_v$ near a corner. So $G_i$ can only be certain unions of above graphs together with some parallel straight arcs near the corners of $R_v$. See Figure~\ref{fig:example4} for two examples of $G_i$ in this case, where over/under marks are given as above.

\begin{figure}[h]
\includegraphics{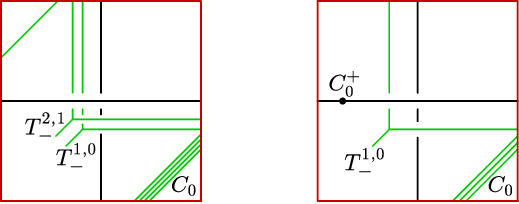}
\caption{Two examples of $G_i$ when $\Psi_i$ contains bands.}\label{fig:example4}
\end{figure}

Note that in each of the pictures in Figure~\ref{fig:example4}, there are more than one straight arcs near the bottom right corner of $R_v$. Only the outermost one indicates a band in some disk in $C_0$, while the innermost one indicates a band in some disk in $T_-^{1,0}$. The other ones indicate the bands in disks in $T_-^{2,1}$ and $C_0^+$, respectively.

\begin{proposition}\label{pro:diag}
Let $\Psi_i$, $\Phi_i$, $G_i$ be as above, $i=1,2$, and $R$ is some $R_e$ or $R_v$. If $\Phi_1=\Phi_2$, then there is an isotopy of $R$, which keeps $A_e$ or $p(A_v)$ and four sides of $R$ invariant, sending $G_1$ to $G_2$. Moreover, during this isotopy, the graphs in $G_1$ can always represent normal disks, and the over/under marks can be preserved.
\end{proposition}

The case when $R$ is some $R_e$ is clear. To show the case when $R$ is some $R_v$, we can use the classification of $\Phi_i$ described above. So $G_i$ can only be a certain union of the straight arcs, figures ``$\textsf{Y}$'', and marked points, which represent certain kinds of normal disks in $\Psi_i$. When there are no bands, the intersection pattern between these graphs in $G_i$, together with $p(A_v)$, can be determined by the vertical arcs in $\partial\Psi_i$ and the points in $K\cap\Psi_i$. Then, since $\Phi_i$ can determine the relative positions of these vertical arcs and points, $\Phi_i$ determines $G_i$, up to isotopy. When there are bands, the case is similar. Bands give the ``outer'' vertical arcs in $\partial\Psi_i$.

With Proposition~\ref{pro:diag}, we can finish the proof of Proposition~\ref{pro:multiD} as in the case of Proposition~\ref{pro:isotop}. We say that such $\Psi_1$ and $\Psi_2$ are {\it equivalent}. So we can classify all equivalence classes of unions of normal disks in $[-1,1]\times R$, as described above. By Proposition~\ref{pro:diag}, we see that if $\Psi_1$ and $\Psi_2$ are equivalent, then $G_1$ and $G_2$ are essentially the same. In many cases, the converse is also true.

According to the above discussion about $G_i$, we can regard $G_i$ as a graph in $R$, whose vertices have degree at most $4$. Those endpoints of straight arcs and figures ``$\textsf{Y}$'' are the vertices of degree $1$. The branch points of figures ``$\textsf{Y}$'' give the vertices of degree $3$. The marked points $p(K\cap\Psi_i)$ are those vertices of degree $0$ or $2$. The intersection points are the vertices of degree $4$. Then we have the following.

\begin{proposition}\label{pro:recon}
If the two graphs $G_1$ and $G_2$ coincide, then after removing all the flat disks in $\Psi_i$ which do not meet $K$, $\Psi_1$ and $\Psi_2$ are equivalent. Moreover, $\Psi_i$ can contain disks in $F_0$ only if $G_i=\emptyset$ or $G_i$ only represents disks in $C_{\pm}$.
\end{proposition}

Similar to Proposition~\ref{prop:NdiskR}, we can show that disks in $F_0$ and normal disks with bands must intersect. So, by Proposition~\ref{prop:NdiskR}, the ``moreover'' part holds. From $G_i$ we can first recover the straight arcs, figures ``$\textsf{Y}$'', and marked points, where those arcs indicating bands can be identified. If there are no bands, then we can recover all the normal disks in $\Psi_i$ except the flat disks which do not meet $K$. When there exist bands, the possible flat disks in $\Psi_i$ belong to $F_{\pm}$. We can recover the normal disks obtained from $\Psi_i$ by removing the bands, up to disks in $F_{\pm}$. Then according to the arcs indicating bands, we can add the bands inductively starting from some ``innermost'' one. The way to add each band is essential unique.

Note that when $R$ is some $R_v$, from the graph $G_i$ together with the over/under mark at $v$, we can also recover all the over/under marks between the straight arcs indicating disks in $C_{\pm}$, the figures ``$\textsf{Y}$'', and the diagram $p(A_v)$. If the straight arc (or a branch of ``$\textsf{Y}$'') intersects an arc in $p(A_v)$ with over (resp. under) mark, then it only has under (resp. over) marks; see Figure~\ref{fig:example3} and Figure~\ref{fig:example4}.

\begin{definition}\label{def:diagram}
Let $\Psi$ be a union of disjoint normal disks in $[-1,1]\times R$, where $R$ is some $R_e$ or $R_v$, and let $G$ be the graph induced by the union of vertical disks in $\Psi$ and $K\cap\Psi$.
We call $G\subset R$ with the over/under marks a {\it diagram} of $\Psi$.

Let $S$ be a normal surface with respect to $\mathcal{H}_D$. Then $S$ intersects each possible $[-1,1]\times R$ in a union of disjoint normal disks. Let $G_S$ denote the union of all the diagrams of $S\cap[-1,1]\times R$. We call $G_S\subset S_0$ a {\it diagram} of $S$.
\end{definition}

Now let $\mathcal{C}$ denote the class of closed surfaces $S\subset S^3$ which satisfy the following conditions: $S$ does not intersect 3-handles of $\mathcal{H}_D$; $S$ intersects each $[-1,1]\times R_f$ in a union of disjoint normal disks; $S$ intersects any $[-1,1]\times A_e$ transversely in some horizontal arcs; and $S$ intersects any $[-1,1]\times\{v\}$ transversely in some points. We use $I_f(S)$, $I_e(S)$, and $I_v(S)$ to denote the numbers of the normal disks, horizontal arcs, and points, respectively. Then we give the following definition.

\begin{definition}\label{def:complex}
For a closed surface $S\subset S^3$ which lies in $\mathcal{C}$, let
\[I_2(S)=\sum I_f(S),\quad I_1(S)=\sum I_e(S),\quad I_0(S)=\sum I_v(S),\]
where the three sums are over all $f$, $e$, and $v$, respectively. Then we call the triple $(I_2(S),I_1(S),I_0(S))$ in $\mathbb{Z}^3$ the {\it complexity} of $S$ and denote it by $\mathcal{I}(S)$. Here we use the dictionary order on $\mathbb{Z}^3$, which is denoted by ``$\preceq$''.
\end{definition}

Clearly all normal surfaces lie in $\mathcal{C}$. In later sections, we will use the complexity many times when we simplify a maximal sphere system of $K$. We will convert the union of spheres in the system into a normal surface with good properties.


\section{Simplifications of maximal sphere systems of $K$}\label{sec:SF}
Let $K$, $D$, and $\mathcal{H}_D$ be as in previous sections. Let $\Sigma$ be the union of spheres in a maximal sphere system $\mathcal{S}=\{S_1,\ldots,S_{n-1}\}$ of $K$ where $n>1$. The main goal of this section is to prove the following theorem.

\begin{theorem}\label{thm:normalForm}
Up to isotopy and $D^2$-surgery, any $\mathcal{S}$ can be converted into a new maximal sphere system of $K$ where the corresponding $\Sigma$ becomes a normal surface with respect to $\mathcal{H}_D$. Also, $\mathcal{I}(\Sigma)$ does not increase in the process when $\Sigma$ lies in $\mathcal{C}$.
\end{theorem}

The proof consists of three parts, which are shown in Sections~\ref{subsec:eleSimp}, \ref{subsec:Moves}, and \ref{subsec:ParaofD}, respectively. In Section~\ref{subsec:eleSimp}, we modify $\Sigma$ so that it lies in $\mathcal{C}$, it meets 1-handles of $\mathcal{H}_D$ in normal disks, and it meets the blocks obtained by cutting each $[-1,1]\times R_v$ along $S_D$ in certain types of disks. In Sections~\ref{subsec:Moves} and \ref{subsec:ParaofD}, we simplify these disks and glue them together to give normal disks in the 0-handles.


\subsection{Preliminary simplifications}\label{subsec:eleSimp}
We first modify $\Sigma$ so that it has the required form except in 0-handles of $\mathcal{H}_D$, and in the 0-handles it is relatively simple.

\vspace{5pt}

\noindent {\bf Step 1.} We modify $\Sigma$ so that it lies in $\mathcal{C}$.

In each of $B_-$ and $B_+$, we can find a small 3-ball disjoint from $\Sigma$. Then we can have an isotopy of $S^3$, so that it sends the two 3-balls to $B_-$ and $B_+$, respectively, and it keeps the points outside some small neighborhood of $B_-\cup B_+$ fixed. Then, by restricting this isotopy to $\Sigma$, we can push $\Sigma$ off the 3-handles of $\mathcal{H}_D$.

For each $[-1,1]\times R_f$, we choose a point $r$ in the interior of $R_f$. We can require that $\Sigma$ meets $[-1,1]\times\{r\}$ transversely, and near each intersection point, $\Sigma$ lies in some $\{t\}\times R_f$. Then we can find a small disk $D^2$ centered at $r$ such that $\Sigma$ meets $[-1,1]\times D^2$ in the disks having the form $\{t\}\times D^2$. We can have an isotopy of $S_0$, so that it sends $D^2$ to $R_f$, and it keeps the points outside a small neighborhood of $R_f$ fixed. This gives an isotopy of $[-1,1]\times S_0$, and we can extend it to an isotopy of $S^3$. Then by restricting the isotopy to $\Sigma$, we can isotope $\Sigma$ so that it meets the 2-handles of $\mathcal{H}_D$ in unions of disjoint normal disks.

We identify $([-1,1]\times R_e, A_e)$ with $([-1,1]\times[-1/2,1/2],\{(0,0)\})\times[-1,1]$. For each 1-handle of $\mathcal{H}_D$, we can require that $\Sigma$ meets $[-1,1]\times\{0\}\times\{0\}$ transversely in some points other than $(0,0,0)$. The complement of a small neighborhood of
\[(\partial([-1,1]\times[-1/2,1/2])\times[-1,1])\cup([-1,1]\times[-1/2,1/2]\times\{0\})\]
in the $1$-handle is a union of two 3-balls. So we can have an isotopy of $S^3$, so that it sends the two 3-balls into the adjacent 0-handles, respectively, it fixes the other handles, and it keeps $K$ invariant. By restricting this isotopy to $\Sigma$, we can isotope $\Sigma$ so that it meets $[-1,1]\times A_e$ in some horizontal arcs.

We can further isotope the part of $\Sigma$ in each $([-1,1]\times R_v,A_v)$ so that it meets $[-1,1]\times\{v\}$ transversely in some points other than $(\pm1/2,v)$. So, the resulting $\Sigma$ lies in $\mathcal{C}$, and the complexity $\mathcal{I}(\Sigma)$ can be defined.

Note that if $\Sigma$ already lies in $\mathcal{C}$, then in this process $\mathcal{I}(\Sigma)$ does not change, and the intersections $\Sigma\cap[-1,1]\times A_e$ and $\Sigma\cap[-1,1]\times\{v\}$ will avoid the points in $K$. Also, we will use $D^2$-surgeries below, where $\mathcal{I}(\Sigma)$ does not increase.

\vspace{5pt}

\noindent {\bf Step 2.} We modify $\Sigma$ so that it meets each 1-handle in disjoint normal disks.

Under the above identification, for each 1-handle of $\mathcal{H}_D$, we can further require that $\Sigma$ meets $[-1,1]\times[-1/2,1/2]\times\{0\}$ transversely and has the form $\alpha\times(-\epsilon,\epsilon)$ near the rectangle, where $\epsilon>0$ and $\alpha$ is a union of disjoint circles and proper arcs in $[-1,1]\times[-1/2,1/2]$. If there are circles, then in $\alpha\times\{0\}$, we have an innermost circle which bounds a disk $D^2$. By applying the $D^2$-surgery, we see that all circles can be eliminated. So by repeating Step~1, we have a new $\Sigma$ lying in $\mathcal{C}$, where the corresponding $\alpha$ can only contain disjoint proper arcs. See Figure~\ref{fig:arcin1}.

\begin{figure}[h]
\includegraphics{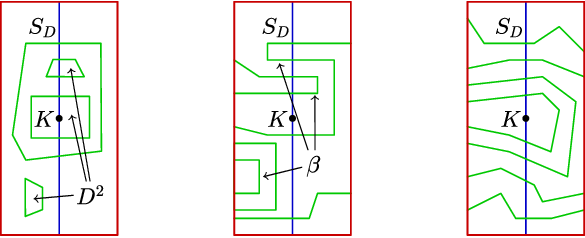}
\caption{Examples of $\alpha\times\{0\}$ in $[-1,1]\times[-1/2,1/2]\times\{0\}$.}\label{fig:arcin1}
\end{figure}

Note that $S_D$ splits $[-1,1]\times[-1/2,1/2]\times\{0\}$ into two rectangles, and it splits an arc in $\alpha\times\{0\}$ into several arcs. Let $\beta=\gamma\times\{0\}$ be an outermost arc in one of the rectangles, where $\partial\beta$ lies in one edge. If $\partial\beta$ lies in the boundary of the handle, then there are two parallel normal disks adjacent to $\beta$. We can isotope $\gamma\times(-\epsilon,\epsilon)$ across the region between the two disks. This reduces $I_2(\Sigma)$. If $\partial\beta$ lies in $S_D$, and the two points lie on the same side of $(0,0,0)$, then we can isotope $\beta$ across $S_D$ so that it lies in the other rectangle. By repeating Step~1, this reduces $I_1(\Sigma)$. Hence, we can require that the two points in $\partial\beta$ lie on different sides of $(0,0,0)$.

Then, since there are no circles in $\alpha$, only one of the two rectangles can contain arcs whose two endpoints lie in the same edge. So, $\alpha\times\{0\}$ contains at most three types of arcs, as shown in the right picture of Figure~\ref{fig:arcin1} up to mirror image. Then we can isotope $\alpha\times\{0\}$ so that each proper arc in it consists of three straight arcs as described before Figure~\ref{fig:surin1}. By suitably modifying normal disks in 2-handles, we can also require that each horizontal arc lies in $\{t\}\times[-1/2,1/2]\times\{0\}$ for some $t$. Then, by repeating Step~1, we can get the required normal disks.

Finally, we note that while the isotopy reducing $I_1(\Sigma)$ does not affect the parts of $\Sigma$ in other 1-handles, the isotopy reducing $I_2(\Sigma)$ does. If $\Sigma$ meets a 1-handle in disjoint normal disks, then these disks can be destroyed when we simplify the part of $\Sigma$ in another 1-handle. So for the first 1-handle, we need to repeat the previous process. Until $\mathcal{I}(\Sigma)$ cannot be reduced, we can finish Step~2.

\vspace{5pt}

\noindent {\bf Step 3.} We simplify the intersection $\Sigma\cap S_D\cap[-1,1]\times R_v$ for each $v$.

The arc $[-1,1]\times\{v\}$ splits $S_D\cap[-1,1]\times R_v$ into four rectangles, and $K$ splits each of these rectangles into two quadrilaterals. Since $\Sigma$ intersects $K$ transversely, we can require that $\Sigma$ meets all these quadrilaterals transversely in disjoint circles and proper arcs. Then by applying the $D^2$-surgery as above, we can eliminate the circles. Because with points in $K$ marked, the four rectangles in $S_D$ are isometric, below it suffices to consider one of them. For example, consider the left one in the right picture of Figure~\ref{fig:handle}. Let $\alpha$ be the intersection of this rectangle with $\Sigma$. Then $K$ splits the components of $\alpha$ into several arcs. See Figure~\ref{fig:arcinSD}.

\begin{figure}[h]
\includegraphics{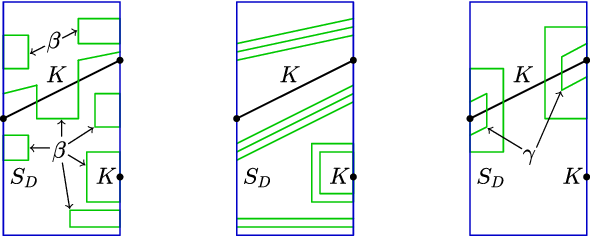}
\caption{Some examples of $\alpha$, $\beta$, and $\gamma$ in a rectangle in $S_D$.}\label{fig:arcinSD}
\end{figure}

Let $\beta$ be an outermost arc in one of the two quadrilaterals where $\partial\beta$ lies in one edge. See the left picture of Figure~\ref{fig:arcinSD} for some possible $\beta$. If $\partial\beta$ lies in $K$, then $\beta$ cuts off a disk whose boundary is a connected summand of $K$. Since it contradicts the assumption that $\mathcal{S}$ is a maximal sphere system of $K$, this case cannot happen. If $\partial\beta$ lies in the boundary of the handle, then $\beta$ meets two normal disks, which lie in the adjacent 1-handle and may be flat or curved. There is an $I$-bundle over the corresponding $A_e$ lying between the two disks. Along the $I$-bundle, we can isotope a small neighborhood of $\beta$ in $\Sigma$ across the 1-handle. This reduces $I_1(\Sigma)$. Then, by repeating Step~2, we get a new $\Sigma$ that meets 1-handles in disjoint normal disks. If $\partial\beta$ lies in one of the three components of $[-1,1]\times\{v\}\setminus K$, then we can isotope $\beta$ across $[-1,1]\times\{v\}$ so that it lies in another rectangle. This reduces $I_0(\Sigma)$.

After reducing $\mathcal{I}(\Sigma)$, we eliminate the circles in the quadrilaterals as before. So we can require that $\partial\beta\subset[-1,1]\times\{v\}$, and the two points lie on different sides of some point in $K$. Now if $\alpha\cap K=\emptyset$, then $\alpha$ contains at most four types of arcs, as shown in the middle picture of Figure~\ref{fig:arcinSD}. Since we have seen that a component of $\alpha$ can intersect $K$ at most once, $\alpha$ consists of proper arcs.

Let $\gamma$ be an outermost arc in the rectangle, where $\gamma\cap K\neq\emptyset$ and $\partial\gamma$ lies in one edge. See the right picture of Figure~\ref{fig:arcinSD} for two possible $\gamma$. Then, as in the case of $\beta$, if $\partial\gamma$ lies in the boundary of the handle, then $\gamma$ meets either one curved disk or two flat disks. There is an $I$-bundle in the adjacent 1-handle, and along it, we can isotope a small neighborhood of $\gamma$ in $\Sigma$ across the 1-handle to reduce $I_1(\Sigma)$. Then by repeating Step~2, we get a new $\Sigma$ as before. If $\gamma$ meets $(-1/2,1/2)\times\{v\}$, then we can isotope $\gamma$ across $[-1,1]\times\{v\}$ so that it lies in another rectangle, as before. This reduces $I_0(\Sigma)$. After reducing $\mathcal{I}(\Sigma)$, we eliminate the circles and certain arcs in the quadrilaterals as in previous processes. So we can require that the arcs in $\alpha$ that meet $K$ have the four types shown in Figure~\ref{fig:arcT1234}.

Note that only the arc of type-2 has endpoints lying in the same edge; only the arc of type-3 meets $(-1/2,1/2)\times\{v\}$; and the arc of type-1 and the arc of type-4 can be distinguished according to whether the intersections of the arc with $K$ and $[-1,1]\times\{v\}$ lie on different sides of $S_0$. Also, if two different types of arcs appear simultaneously, then the case must be one of the three shown in Figure~\ref{fig:arcT1234}. Hence by combining the possible arcs in $\alpha$ which do not meet $K$, $\alpha$ contains at most five different types of arcs. The possible cases can be listed easily.

\begin{figure}[h]
\includegraphics{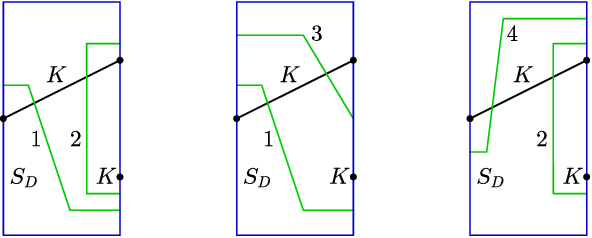}
\caption{The four types of the arcs in $\alpha$ that meet $K$.}\label{fig:arcT1234}
\end{figure}

Finally, we note that as in Step~2, the isotopies reducing $I_1(\Sigma)$ and $I_0(\Sigma)$ affect the parts of $\Sigma$ in other handles and the intersections of $\Sigma$ with other rectangles in $S_D$. Until $\mathcal{I}(\Sigma)$ cannot be reduced, we can finish Step~3.

\vspace{5pt}

\noindent {\bf Step 4.} We simplify the parts of $\Sigma$ in the components of $[-1,1]\times R_v\setminus S_D$.

The skeleton $S_D$ splits $[-1,1]\times R_v$ into four blocks. With points in $K$ marked, the four blocks differ by $\pi$-rotations about $[-1,1]\times\{v\}$ or reflections across those rectangles in $S_D$. So it suffices to consider one of them. For example, consider the one in front of the left rectangle in Figure~\ref{fig:handle}. Let $S^2$ be its boundary sphere. Since $\Sigma$ intersects $S^2$ transversely, we can find a neighborhood of $S^2$ in the block, which can be identified with $S^2\times[0,1)$, so that it meets $\Sigma$ in $(\Sigma\cap S^2)\times[0,1)$. Hence, $\Sigma$ meets $S^2\times\{1/2\}$ transversely in disjoint circles, and by applying $D^2$-surgeries, we can eliminate them. In this process, the circles in $\Sigma\cap S^2$ will either be eliminated or bound some disks in $S^2\times[0,1/2)$. Since the resulting $\Sigma$ cannot meet the 3-ball bounded by $S^2\times\{1/2\}$ in the block, it meets the block only in the disks.

Consider an arc lying in some edge of a quadrilateral in $S^2$ so that either it lies in $(-1/2,1/2)\times\{v\}$ or it does not meet $(-1/2,1/2)\times\{v\}$. If there is such an arc intersecting a component of $\Sigma\cap S^2$ at least twice, then it must contain a subarc $\delta$ so that $\delta\cap(\Sigma\cap S^2)=\partial\delta$ and $\delta$ meets only one component of $\Sigma\cap S^2$.

In the block, the component meeting $\delta$ bounds a disk in $\Sigma$. Let $D^2$ be the disk, let $\delta_1$ be a close parallel copy of $\delta$ in the quadrilateral so that $\delta_1\cap\partial D^2=\partial\delta_1$, and let $\delta_2$ be a proper arc in $D^2$ so that $\partial\delta_2=\partial\delta_1$. Note that $D^2$ splits the block into two 3-balls. In one of the 3-balls, $\delta_1\cup\delta_2$ bounds a properly embedded disk $D^2_1$, so that $D^2_1\cap\Sigma=\delta_2$. So we can isotope $\delta_2$ across $D^2_1$ to change the pattern of $\Sigma\cap S^2$ in the quadrilateral. See the left picture of Figure~\ref{fig:pushdisk}, where the pattern of $\Sigma\cap S^2$ in the larger quadrilateral has been changed.

This isotopy provides an outermost arc $\beta$ in the quadrilateral, whose endpoints lie in the same edge. By repeating Step~3, we see that $\partial\beta$ (and $\delta$) cannot lie in $K$, and $\mathcal{I}(\Sigma)$ is reduced. Then by repeating the previous process, we can get a new $\Sigma$ which meets the block in disjoint disks. So, we can require that there is no such $\delta$. It is convenient to view the longer edge in $[-1,1]\times\{v\}$ as two edges meeting at a point in $K$. Then a component of $\Sigma\cap S^2$ can intersect any edge of a quadrilateral at most once. In particular this means that each component of $\Sigma\cap S^2$ meets each smaller quadrilateral in at most one arc.

\begin{figure}[h]
\includegraphics{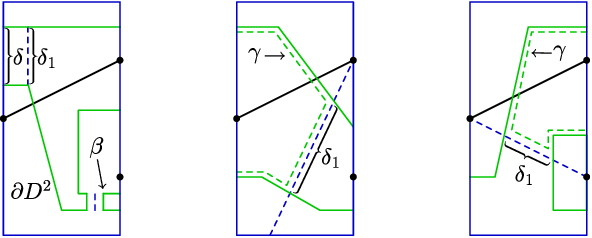}
\caption{The pattern of $\Sigma\cap S^2$ can be changed by an isotopy.}\label{fig:pushdisk}
\end{figure}

If a component of $\Sigma\cap S^2$ meets a larger quadrilateral in at least two arcs, then it must be one of the two cases shown in the right two pictures of Figure~\ref{fig:pushdisk}. This is because the endpoints of those arcs must lie in different edges, where the longer edge in $[-1,1]\times\{v\}$ is viewed as two edges. Then, the arc meeting $K$ must give a type-3 or type-4 arc in the rectangle. Now, consider a straight arc as shown in the pictures. It divides the four edges into two pairs of edges. So it must intersect the component. Up to isotopy, we can require that the arcs from $\Sigma\cap S^2$ have minimal intersections with it. Since it intersects the component twice, it has a subarc $\delta_1$ so that $\delta_1\cap(\Sigma\cap S^2)=\partial\delta_1$ and $\delta_1$ intersects only one component of $\Sigma\cap S^2$. Then $\delta_1$ meets two arcs from the component with the same types as before. As in previous processes, we can change the pattern of $\Sigma\cap S^2$ in the quadrilateral, and $\mathcal{I}(\Sigma)$ can be reduced. The only difference is that now the isotopy provides an outermost arc $\gamma$ in the rectangle, where $\gamma\cap K\neq\emptyset$ and $\partial\gamma$ lies in one edge.

Then, by repeating the previous processes, we can require that each component of $\Sigma\cap S^2$ meets each quadrilateral in at most one arc. Finally, note that as in the previous steps, the isotopies reducing $\mathcal{I}(\Sigma)$ affect those parts of $\Sigma$ in other handles or blocks. Until $\mathcal{I}(\Sigma)$ cannot be reduced, we can finish Step~4.

\vspace{5pt}

\noindent {\bf Step 5.} We classify the possible types of the components of $\Sigma\cap S^2$.

Note that $S^2$ consists of two squares meeting 3-handles, two rectangles meeting 2-handles, and two rectangles in $S_D$. For simplicity, we view the block from above so that in a picture, the square meeting $B_+$ is larger than the square meeting $B_-$; see the figures below. Let $\Delta$ be a component of $\Sigma\cap S^2$. Then, by Step~4, $\Delta$ meets each quadrilateral in at most one arc; by Step~2, in a rectangle meeting 2-handles, the arcs in $\Delta$ come from adjacent normal disks. So, we can list the possible $\Delta$, up to isotopy. According to the number of points in $\Delta\cap K$, we have three cases.

Case 1. $|\Delta\cap K|=2$.

Since $\Delta$ cannot meet an arc in $K\cap S^2$ twice, it meets each of the two arcs, and in the rectangles in $S_D$, it must have one of the four types shown in Figure~\ref{fig:arcT1234}. So according to the types, there are four possibilities, as shown in Figure~\ref{fig:cirT2}.

In the block, $\Delta$ bounds a disk in $\Sigma$. Let $D^2$ be the disk, and assume that it lies in the sphere $S_j\subseteq\Sigma$. Then $S_j$ meets $K$ in $\Delta\cap K$ and it splits $K$ into two arcs $E_1$ and $E_2$. Let $\mu$ be a proper arc in $D^2$ which connects the two points in $\Delta\cap K$. We have two knots $\widehat{E}_i=E_i\cup\mu$ where $i=1,2$, and $K=\widehat{E}_1\#\widehat{E}_2$; see the right picture of Figure~\ref{fig:cirT2} for an example. Up to isotopy we can require that $\mu$ is a straight arc. So for $i=1,2$, by projecting $\widehat{E}_i$ to $S_0$, we can obtain a diagram $Q_i$. If $\Delta$ does not contain arcs of type-3, then $Q_1$ or $Q_2$ is not minimal. If $\Delta$ contains arcs of type-3, then the crossing in the 0-handle is no longer a crossing in $Q_1$ and $Q_2$. So in each case we can construct a diagram of $K$, with fewer crossings than $D$, from those of $\widehat{E}_1$ and $\widehat{E}_2$. This contradiction means that Case~1 cannot happen.

\begin{figure}[h]
\includegraphics{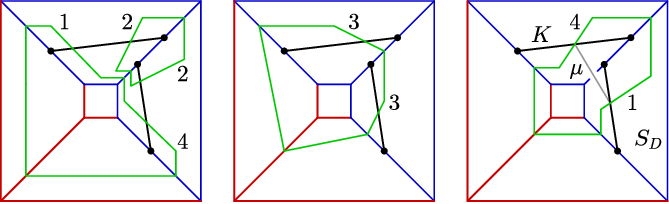}
\caption{The four possible $\Delta$ with $|\Delta\cap K|=2$.}\label{fig:cirT2}
\end{figure}

Case 2. $|\Delta\cap K|=1$.

First assume that $\Delta$ meets the upper arc of $K\cap S^2$. Then corresponding to the four types of arcs shown in Figure~\ref{fig:arcT1234}, there are also four possibilities, as shown in Figure~\ref{fig:cirT1}, and we use $1^+$, $2^+$, $3^+$, and $4^+$ to denote their types, respectively. The pictures in Figure~\ref{fig:cirT1} also give all those cases when two components with different types appear simultaneously. Similarly, there are types $1^-$, $2^-$, $3^-$, and $4^-$ for the components which meet the lower arc of $K\cap S^2$. Also, note that there can be two components which meet different arcs of $K\cap S^2$. All the possible sets of the types are $\{1^+,1^-\}$, $\{3^+,3^-\}$, $\{3^+,4^-\}$, $\{4^+,3^-\}$, and $\{4^+,4^-\}$.

\begin{figure}[h]
\includegraphics{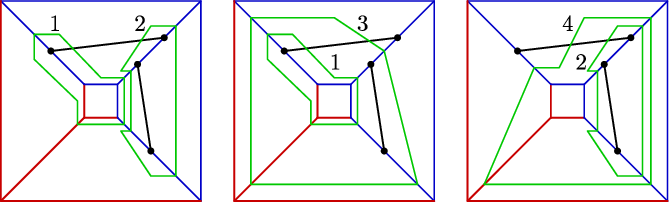}
\caption{The four possible $\Delta$ with $|\Delta\cap K|=1$.}\label{fig:cirT1}
\end{figure}

Case 3. $|\Delta\cap K|=0$.

Note that $\Delta\cap[-1,1]\times\{v\}$ must be nonempty and it has at most three points. According to the number of points in it we have three cases, which will be related to the flat, curved, and twisted disks, respectively.

Let $(t,v)$ denote a point in $[-1,1]\times\{v\}\setminus K$. It has one of the types $-$, $0$, or $+$ if $t<-1/2$, $|t|<1/2$, or $t>1/2$, respectively. Since $\Delta$ meets those quadrilaterals in at most four types of arcs as shown in the middle picture of Figure~\ref{fig:arcinSD}, it is not hard to give all possible patterns of $\Delta$. If $\Delta$ meets $[-1,1]\times\{v\}$ in one point, then there are three types $\mathrm{I}_-$, $\mathrm{I}$, $\mathrm{I}_+$, which correspond to those three types of the point, respectively. If $\Delta$ meets $[-1,1]\times\{v\}$ in two points, then we also have three types $\mathrm{II}_-$, $\mathrm{II}$, $\mathrm{II}_+$, which correspond to the three sets of types, $\{-,0\}$, $\{-,+\}$, $\{0,+\}$, of the two points, respectively. If $\Delta$ meets $[-1,1]\times\{v\}$ in three points, then there is only one type $\mathrm{III}$. See Figure~\ref{fig:cirT0} for the seven types of $\Delta$.

\begin{figure}[h]
\includegraphics{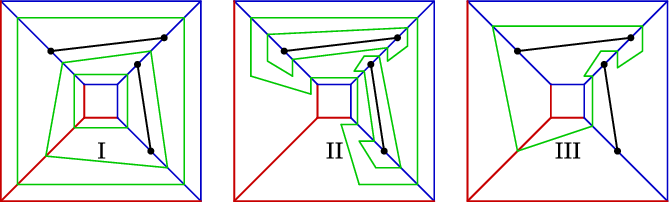}
\caption{The seven possible $\Delta$ with $|\Delta\cap K|=0$.}\label{fig:cirT0}
\end{figure}

Note that any two of the types $\mathrm{I}$, $\mathrm{II}$, and $\mathrm{III}$ cannot appear simultaneously, and each of the types $\mathrm{I}_-$, $\mathrm{I}_+$, $\mathrm{II}_-$, and $\mathrm{II}_+$ can appear simultaneously with other types. The relations between the types in Case~2 and Case~3 can also be listed easily.

\vspace{5pt}

\noindent {\bf Summary.} By Steps~1-5 we can require that $\Sigma$ does not meet 3-handles, it meets 2-handles and 1-handles in disjoint normal disks, and it meets the blocks obtained by cutting 0-handles along $S_D$ in some disjoint disks, where the boundaries of the disks have types $1^{\pm}$, $2^{\pm}$, $3^{\pm}$, $4^{\pm}$, $\mathrm{I}_{\pm}$, $\mathrm{II}_{\pm}$, $\mathrm{I}$, $\mathrm{II}$, and $\mathrm{III}$.

\begin{definition}\label{def:tauP}
We call a disk component of the intersection of $\Sigma$ with a block a {\it $\tau$-disk} if its boundary circle has type $\tau$. If such a $\tau$-disk intersects $K$, then we call the intersection point in $K$ a {\it $\tau$-point}.
\end{definition}


\subsection{Moves of the $\tau$-points}\label{subsec:Moves}
Now we isotope $\Sigma$ so that the $\tau$-points move along $K$. Then the types will change. There are two ways to define such an isotopy, and we can use it to eliminate $1^-$, $2^-$, $3^{\pm}$, $4^{\pm}$, or $1^+$, $2^+$, $3^{\pm}$, $4^{\pm}$, respectively.

\vspace{5pt}

First note that if a $\tau$-disk $D^2_1$ intersects $K$, then it intersects the rectangle that contains the $\tau$-point in an arc whose type determines $\tau$. Then, since the rectangle is shared by two blocks, the arc must also lie in another $\tau$-disk $D^2_2$. So, we can get a larger disk $D^2_1\cup D^2_2$, which intersects $K$ transversely in the $\tau$-point. We call this disk an {\it enlarged $\tau$-disk}. We will isotope it to move the $\tau$-point.

To define the isotopy, we first choose a vertical direction of $[-1,1]\times S_0$, namely ``up'' or ``down''. Then we can locally push $\Sigma$ along the direction. The two choices will give parallel consequences, so we just choose ``up''.

There will be two kinds of isotopies, called the {\it 0-moves} and {\it 1-moves}. A 0-move will happen in a 0-handle, while a 1-move will pass through a 1-handle. Note that by Proposition~\ref{prop:goodhandle}, each 1-handle meets two 0-handles. The 1-move will affect the parts of $\Sigma$ in the 1-handle together with the adjacent two 0-handles.

To give the moves, it suffices to focus on the intersection $\Sigma\cap S_D$. Let $D^2$ be an enlarged $\tau$-disk in some $[-1,1]\times R_v$, and let $P=D^2\cap K$. Let $\Pi$ be the rectangle in $S_D$ that contains $P$, and let $\Lambda=D^2\cap\Pi$. Then a 0-move (resp. 1-move) can be applied if $\tau$ is one of $3^+$, $2^-$, or $4^-$ (resp. $4^+$, $1^-$, or $3^-$), and $\Lambda$ cuts off a disk $D^2_0$ from the lower quadrilateral in $\Pi$ so that $D^2_0\cap B_-=\emptyset$ and $D^2_0\cap\Sigma\subset\Lambda$.

\vspace{5pt}

\noindent {\bf Case 1.} The conditions for the 0-move hold.

In $[-1,1]\times R_v$, there exists a rectangle $\Pi'$ which differs from $\Pi$ by a $\pi$-rotation about $[-1,1]\times\{v\}$. To apply the 0-move, we isotope the part of $\Lambda$ that lies in the lower quadrilateral across $[-1,1]\times\{v\}$ via the disk $D^2_0$. Then $P$ is moved into $\Pi'$, and one intersection point in $\Sigma\cap[-1,1]\times\{v\}$ is ``moved up''; see Figure~\ref{fig:0move} for an example, where $\tau$ is $2^-$. Note that if $\tau$ is $3^+$, $2^-$, or $4^-$, then the $\tau$-disk in each of the blocks containing $\Pi$ becomes a $\mathrm{I}_+$-disk, $\mathrm{II}_+$-disk, or $\mathrm{I}$-disk, respectively. In the blocks containing $\Pi'$, the two disks from $\Sigma$ which meet $D^2_0$ become two new disks, and each of the new disks intersects $K$.

\begin{figure}[h]
\includegraphics{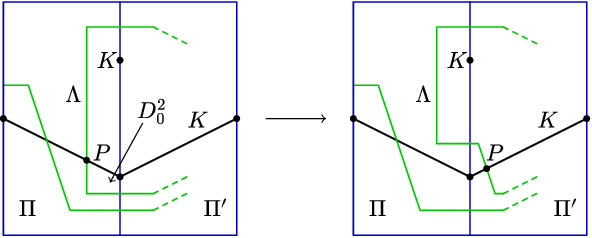}
\caption{An example of a 0-move where the type $\tau$ is $2^-$.}\label{fig:0move}
\end{figure}

Now consider a block containing $\Pi'$. Let $D^2_1$ be the disk from $\Sigma$ that meets $D^2_0$, and let $D^2_2$ be the new disk. So, by applying the 0-move, $D^2_1$ becomes $D^2_2$. Denote by $\Pi_0$ the other rectangle in $S_D$ that lies in the block. Then $D^2\cap\Pi_0$ gives an arc in $\partial D^2_1$ which does not meet $K$. So $\partial D^2_1\cap\Pi_0\cap K=\emptyset$. If $\partial D^2_1\cap\Pi'\cap K\neq\emptyset$, then $\partial D^2_1\cap\Pi'$ cuts off a disk from the low quadrilateral in $\Pi'$ where the disk meets $D^2_0$ in an arc. So their union is a disk, which gives a trivial connected summand of $K$. The contradiction means that $\partial D^2_1\cap K=\emptyset$. So, the type of $\partial D^2_1$ is one of $\mathrm{I}_{\pm}$, $\mathrm{II}_{\pm}$, $\mathrm{I}$, $\mathrm{II}$, and $\mathrm{III}$. Then we can find the type of $\partial D^2_2$ case-by-case.

If $\tau$ is $3^+$, then $\partial D^2_1$ has type $\mathrm{II}_-$ or $\mathrm{I}$. So, corresponding to the two types, $\partial D^2_2$ has type $2^+$ or $4^+$, respectively. If $\tau$ is $2^-$, then $\partial D^2_1$ has type $\mathrm{I}_-$ or $\mathrm{II}$. Then $\partial D^2_2$ has type $3^-$ in the case of $\mathrm{I}_-$. The case of $\mathrm{II}$ cannot happen, because there are no circles in the intersection of $\Sigma$ with a square in $[-1,1]\times\partial R_v$. Similarly, if $\tau$ is $4^-$, then $\partial D^2_1$ has type $\mathrm{I}_-$ or $\mathrm{II}$, and $\partial D^2_2$ has type $3^-$ in the case of $\mathrm{I}_-$. In the case of $\mathrm{II}$, $\partial D^2_2$ meets the larger quadrilateral in $\Pi'$ in two arcs. By Step~4, $\mathcal{I}(\Sigma)$ can then be reduced. So, by repeating previous steps, we can require that the case does not happen. Hence $P$ becomes some $\tau'$-point, and we have the following.

\begin{proposition}\label{prop:0move}
A 0-move converts a $\tau$-point into some $\tau'$-point, so that

(1) if $\tau$ is $3^+$, then $\tau'$ is $2^+$ or $4^+$;

(2) if $\tau$ is $2^-$, then $\tau'$ is $3^-$;

(3) if $\tau$ is $4^-$, then $\tau'$ is $3^-$.
\end{proposition}

\noindent {\bf Case 2.} The conditions for the 1-move hold.

Let $([-1,1]\times R_e, A_e)$ be the 1-handle meeting $\Pi$, and let $([-1,1]\times R_{v'}, A_{v'})$ be the other 0-handle adjacent to this 1-handle. In $[-1,1]\times R_{v'}$, there exists another rectangle $\Pi'\subset S_D$ meeting $[-1,1]\times R_e$. In $[-1,1]\times A_e$, there exists an arc $\omega\subset\Sigma$ meeting $\Lambda$, and between $\omega$ and $A_e$, we have a disk $D^2_e$ with $D^2_e\cap\Sigma=\omega$. To apply the 1-move, we isotope the part of $\Lambda$ that lies in the lower quadrilateral across the 1-handle via the disk $D^2_0\cup D^2_e$. Then $P$ is moved into $\Pi'$, and the arc $\omega$ is ``moved up''; see Figure~\ref{fig:1move} for two examples when $\tau$ is $1^-$. Note that if $\tau$ is $4^+$, $1^-$, or $3^-$, then the $\tau$-disk in each of the two blocks containing $\Pi$ becomes a $\mathrm{I}_+$-disk, $\mathrm{I}_+$-disk, or $\mathrm{I}$-disk, respectively. In the blocks containing $\Pi'$, the disks from $\Sigma$ that meet $D^2_e$ become new disks, and each of the new disks intersects $K$.

\begin{figure}[h]
\includegraphics{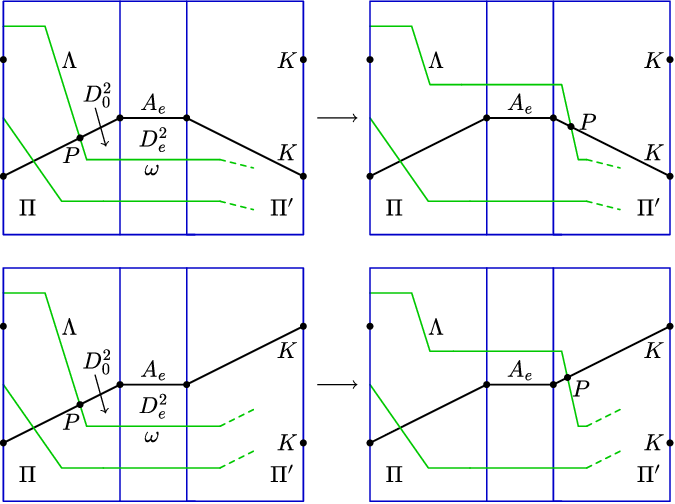}
\caption{Two examples of 1-moves when the type $\tau$ is $1^-$.}\label{fig:1move}
\end{figure}

Now consider a block containing $\Pi'$. Let $D^2_1$ be the disk from $\Sigma$ that meets $D^2_e$, and let $D^2_2$ be the new disk. So, by applying the 1-move, $D^2_1$ becomes $D^2_2$. Denote by $\Pi_0$ the other rectangle in $S_D$ that lies in the block. If $\partial D^2_1\cap\Pi_0$ meets $K$, then $\partial D^2_1$ has type $1^+$, $3^-$, or $4^-$, and $\partial D^2_2$ gives some $\Delta$ discussed in Case~1 of Step~5. Since the case does not happen, we have $\partial D^2_1\cap\Pi_0\cap K=\emptyset$. Similar to the case of the 0-move, we also have $\partial D^2_1\cap\Pi'\cap K=\emptyset$. So $\partial D^2_1\cap K=\emptyset$ and the type of $\partial D^2_1$ is one of $\mathrm{I}_{\pm}$, $\mathrm{II}_{\pm}$, $\mathrm{I}$, $\mathrm{II}$, and $\mathrm{III}$. Then we can get the type of $\partial D^2_2$ case-by-case.

Note that the two arcs from $K$ that lie in $\Pi$ and $\Pi'$ may be ``upper'' or ``lower''. So, according to whether the types of the two arcs are the same or different, there are two kinds of 1-moves, called the {\it 1s-moves} and {\it 1d-moves}, respectively.

First consider the 1s-moves. If $\tau$ is $4^+$, then $\partial D^2_1$ has type $\mathrm{I}_-$ or $\mathrm{I}$. So, $\partial D^2_2$ has type $1^+$ or $3^+$, respectively. If $\tau$ is $1^-$, then $\partial D^2_1$ has type $\mathrm{I}_-$ or $\mathrm{III}$. In the case of $\mathrm{I}_-$, $\partial D^2_2$ has type $4^-$. For the case of $\mathrm{III}$, $\partial D^2_2$ meets the larger quadrilateral in $\Pi'$ in two arcs, so $\mathcal{I}(\Sigma)$ can be reduced by Step~4. Hence by repeating previous steps, we can require that the case does not happen. Similarly, when $\tau$ is $3^-$, the type of $\partial D^2_1$ is $\mathrm{I}_-$ or $\mathrm{III}$, and we can require that the case of $\mathrm{III}$ does not happen. So $\partial D^2_2$ has type $4^-$. The discussion for the 1d-moves is similar. So we omit it. Hence, we see that $P$ becomes some $\tau'$-point, and we have the following.

\begin{proposition}\label{prop:1move}
A 1s-move converts a $\tau$-point into some $\tau'$-point, so that

(1) if $\tau$ is $4^+$, then $\tau'$ is $1^+$ or $3^+$;

(2) if $\tau$ is $1^-$, then $\tau'$ is $4^-$;

(3) if $\tau$ is $3^-$, then $\tau'$ is $4^-$.

In contrast, a 1d-move converts a $\tau$-point into some $\tau'$-point, so that

(1) if $\tau$ is $4^+$, then $\tau'$ is $4^-$;

(2) if $\tau$ is $1^-$, then $\tau'$ is $1^+$ or $3^+$;

(3) if $\tau$ is $3^-$, then $\tau'$ is $1^+$ or $3^+$.
\end{proposition}

\vspace{3pt}

\noindent {\bf Step 6.} We remove the $\tau$-points with the types $1^-$, $2^-$, $3^{\pm}$, and $4^{\pm}$.

By the above discussion, if a 0-move or a 1-move is applied and $\mathcal{I}(\Sigma)$ cannot be reduced, then $\Sigma$ still meets those 2-handles and 1-handles in disjoint normal disks, and it meets those blocks in disjoint $\tau$-disks, while $\mathcal{I}(\Sigma)$ stays the same. However, since in this process, either some point in $\Sigma\cap[-1,1]\times\{v\}$ is ``moved up'' or some arc in $\Sigma\cap[-1,1]\times A_e$ is ``moved up'', the move can only be applied finitely many times. For example, we can use the number of arcs in all $\Sigma\cap[-1,0]\times A_e$ and the number of points in all $\Sigma\cap[-1,1/2]\times\{v\}$ and $\Sigma\cap[-1,-1/2]\times\{v\}$ to bound the number of 1-moves and the number of 0-moves, respectively.

If there exist $3^+$-points or $2^-$-points, then some 0-move can be applied. If there exist $4^+$-points or $1^-$-points, then some 1-move can be applied. So when there are no 0-moves and 1-moves that can be applied, there are no $\tau$-points with types $1^-$, $2^-$, $3^+$, and $4^+$. Then, there are also no $3^-$-points and $4^-$-points. So, by applying the moves as much as possible, we only have $1^+$-points and $2^+$-points.

As a consequence, in each $[-1,1]\times R_v$ only the two rectangles that contain the upper arcs from $K$ can contain $\tau$-points. Below we show that actually only one of the two rectangles can contain $\tau$-points.

\begin{figure}[h]
\includegraphics{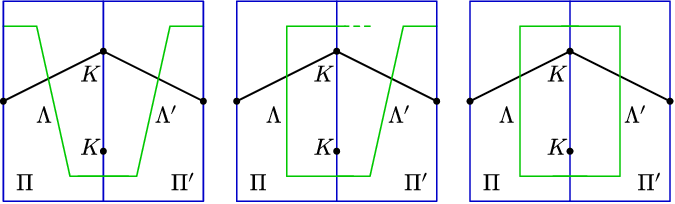}
\caption{Possible cases according to the types of $\Lambda$ and $\Lambda'$.}\label{fig:Tpoint}
\end{figure}

Assume that $\Pi$ and $\Pi'$ are the two rectangles in a $[-1,1]\times R_v$ that contain the upper arcs from $K$, and both of them contain $\tau$-points. Then, there are some arcs in $\Sigma\cap\Pi$ and $\Sigma\cap\Pi'$ passing through those $\tau$-points, respectively. Since $\tau$ is $1^+$ or $2^+$, all these arcs intersect $[-1,-1/2]\times\{v\}$. Let $\Lambda\subset\Pi$ (resp. $\Lambda'\subset\Pi'$) be the arc whose intersection with $[-1,-1/2]\times\{v\}$ is the uppermost, among those arcs in $\Pi$ (resp. $\Pi'$) that contain $\tau$-points. Then, in $[-1,-1/2]\times\{v\}$, the intersections of $\Lambda$ and $\Lambda'$ coincide, since an arc in $\Sigma\cap\Pi$ or $\Sigma\cap\Pi'$ that meets $[-1/2,1/2]\times\{v\}$ also meets $[-1,-1/2]\times\{v\}$. Then according to those possible types of $\Lambda$ and $\Lambda'$, there are essentially three cases, as shown in Figure~\ref{fig:Tpoint}.

Note that the arcs in $\Sigma\cap\Pi$ and $\Sigma\cap\Pi'$ which meet $[-1/2,1/2]\times\{v\}$ give some circles in $\Pi\cup\Pi'$ surrounding $(-1/2,v)$. So, if there exist such arcs, then $\mathcal{I}(\Sigma)$ can be reduced by applying the $D^2$-surgery. Then by repeating previous steps, we can require that there are no such arcs. So, in each case shown in Figure~\ref{fig:Tpoint}, there is a disk $D^2\subset\Pi\cup\Pi'$, such that $\partial D^2$ consists of an arc $A\subset K$ and an arc $\mu\subset\Lambda\cup\Lambda'$, $D^2\cap\Sigma=\mu$, and $(-1/2,v)\in D^2$. Then $K$ has a connected summand which has at least two components, $A\cup\mu$ and the one passing $(-1/2,v)$. It is impossible.

\vspace{5pt}

\noindent {\bf Summary.} By Step~6 we can further require that in each block the boundaries of $\tau$-disks have types $1^+$, $2^+$, $\mathrm{I}_{\pm}$, $\mathrm{II}_{\pm}$, $\mathrm{I}$, $\mathrm{II}$, and $\mathrm{III}$, and in each 0-handle at most one rectangle in $S_D$ can contain points in $\Sigma\cap K$.

\begin{remark}\label{rem:updown}
If we choose the direction ``down'' instead of ``up'', then we can get a similar definition of 0-moves and 1-moves, and by a similar process we can require that the only $\tau$-points are $1^-$-points and $2^-$-points.
\end{remark}


\subsection{Gluing of the $\tau$-disks}\label{subsec:ParaofD}
Now, for each $[-1,1]\times R_v$, we glue those $\tau$-disks in the four blocks together and simplify $\Sigma$ further in this process. Then $\Sigma$ also meets 0-handles in disjoint normal disks. This finishes the proof of Theorem~\ref{thm:normalForm}.

\vspace{5pt}

First we need to know the relations between the $\tau$-disks of the nine types. Also, we need certain parameters for unions of $\tau$-disks in a block, so that we can do the gluing easier. For simplicity, we denote $1^+$ and $2^+$ by $\tau_1$ and $\tau_2$, respectively. The following proposition can be checked case-by-case easily.

\begin{proposition}\label{prop:TypeR}
We regard $\mathrm{I}_+$, $\mathrm{I}_-$, $\mathrm{I}$, $\mathrm{II}_+$, $\mathrm{II}_-$, $\mathrm{II}$, $\mathrm{III}$, $\tau_1$, $\tau_2$ as $9$ vertices. If in a block, there are two disjoint $\tau$-disks having two different types, respectively, then we add an edge between the corresponding vertices. Then

(1) $\mathrm{I}_+$ is adjacent to any other vertex, and so are $\mathrm{I}_-$ and $\mathrm{II}_-$;

(2) the remaining edges are given in the graph in Figure~\ref{fig:graph3}.
\end{proposition}

\begin{figure}[h]
\includegraphics{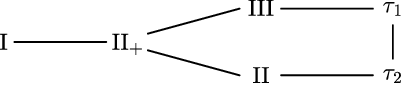}
\caption{Edges between the types other than $\mathrm{I}_+$, $\mathrm{I}_-$ and $\mathrm{II}_-$.}\label{fig:graph3}
\end{figure}

Let $B$ denote a block in $[-1,1]\times R_v$. Then $\Sigma\cap B$ is a union of disjoint $\tau$-disks with the nine types. We define seven parameters for $\Sigma\cap B$ as follows: we let $I_{v,+}$, $I_{v,0}$, and $I_{v,-}$ be those numbers of intersection points of $\Sigma\cap B$ with $[1/2,1]\times\{v\}$, $[-1/2,1/2]\times\{v\}$, and $[-1,-1/2]\times\{v\}$, respectively; we let $h_{1,+}$ and $h_{1,-}$ denote those numbers of intersection points of $\Sigma\cap B$ with the two edges of $B$ which meet the upper and the lower arcs from $K$, respectively, and the adjacent 1-handles; we let $h_2$ be the number of intersection points of $\Sigma\cap B$ with the edge of $B$ which lies in the adjacent 2-handle; also, we let $\kappa$ denote the number of points in $\Sigma\cap B\cap K$. In Table~\ref{tab:Para}, we list these parameters for each single $\tau$-disk.

Let $\mathfrak{T}$ be the vector space generated by the nine types over $\mathbb{R}$. Then $\Sigma\cap B$ can be uniquely represented by a vector $\Xi$ in $\mathfrak{T}$, up to isotopy of the $\tau$-disks. Let
\[\Xi=x_+\mathrm{I}_++x\mathrm{I}+x_-\mathrm{I}_-+ y_+\mathrm{II}_++y\mathrm{II}+y_-\mathrm{II}_-+ z\mathrm{III}+t_1\tau_1+t_2\tau_2,\]
where $x_+$, $x$, $x_-$, $y_+$, $y$, $y_-$, $z$, $t_1$, $t_2$ are nonnegative integers. The parameters are now linear functions of the coordinates if we identify $\mathfrak{T}$ with $\mathbb{R}^9$. From them there are six useful functions defined below, where $I_v=I_{v,+}+I_{v,0}+I_{v,-}$.
\begin{align*}
&a_+=\frac{I_v-h_{1,-}}{2},\, a_-=\frac{I_v-h_{1,+}}{2},\\
&h_+=\frac{h_{1,+}-h_2}{2},\, h_-=\frac{h_{1,-}-h_2}{2},\\
&\eta=I_{v,+}+I_{v,-}-\frac{h_{1,+}+h_{1,-}}{2},\\
&\sigma=\frac{I_{v,+}-I_{v,0}+I_{v,-}-h_2}{2}.
\end{align*}
In Table~\ref{tab:Dfunc}, we list the values of these derived functions at the generators.

\begin{table}[h]
\caption{Seven parameters for one $\tau$-disk.}\label{tab:Para}
\centerline{
\begin{tabular}{l|ccccccccc}
& $\mathrm{I}_+$ & $\mathrm{I}$ & $\mathrm{I}_-$ & $\mathrm{II}_+$ & $\mathrm{II}$ & $\mathrm{II}_-$ & $\mathrm{III}$ & $\tau_1$ & $\tau_2$\\
\hline
$I_{v,+}$ & $1$ & $0$ & $0$ & $1$ & $1$ & $0$ & $1$ & $0$ & $1$\\
$I_{v,0}$ & $0$ & $1$ & $0$ & $1$ & $0$ & $1$ & $1$ & $0$ & $0$\\
$I_{v,-}$ & $0$ & $0$ & $1$ & $0$ & $1$ & $1$ & $1$ & $1$ & $1$\\
$h_{1,+}$ & $1$ & $1$ & $1$ & $2$ & $2$ & $0$ & $1$ & $1$ & $0$\\
$h_{1,-}$ & $1$ & $1$ & $1$ & $0$ & $2$ & $2$ & $1$ & $1$ & $2$\\
$h_2$     & $1$ & $1$ & $1$ & $0$ & $0$ & $0$ & $1$ & $1$ & $0$\\
$\kappa$  & $0$ & $0$ & $0$ & $0$ & $0$ & $0$ & $0$ & $1$ & $1$
\end{tabular}}
\end{table}

\begin{table}[h]
\caption{Six derived linear functions on $\mathfrak{T}$.}\label{tab:Dfunc}
\centerline{
\begin{tabular}{l|ccccccccc}
& $\mathrm{I}_+$ & $\mathrm{I}$ & $\mathrm{I}_-$ & $\mathrm{II}_+$ & $\mathrm{II}$ & $\mathrm{II}_-$ & $\mathrm{III}$ & $\tau_1$ & $\tau_2$\\
\hline
$a_+$    & $0$ & $0$  & $0$ & $1$ & $0$ & $0$ & $1$ & $0$ & $0$\\
$a_-$    & $0$ & $0$  & $0$ & $0$ & $0$ & $1$ & $1$ & $0$ & $1$\\
$h_+$    & $0$ & $0$  & $0$ & $1$ & $1$ & $0$ & $0$ & $0$ & $0$\\
$h_-$    & $0$ & $0$  & $0$ & $0$ & $1$ & $1$ & $0$ & $0$ & $1$\\
$\eta$   & $0$ & $-1$ & $0$ & $0$ & $0$ & $0$ & $1$ & $0$ & $1$\\
$\sigma$ & $0$ & $-1$ & $0$ & $0$ & $1$ & $0$ & $0$ & $0$ & $1$
\end{tabular}}
\end{table}

\begin{proposition}\label{prop:PdetD}
Up to isotopy, the union of $\tau$-disks $\Sigma\cap B$ is determined by the seven parameters $I_{v,+}$, $I_{v,0}$, $I_{v,-}$, $h_{1,+}$, $h_{1,-}$, $h_2$, and $\kappa$.
\end{proposition}

\begin{proof}
Let $\Xi$ be the vector representing $\Sigma\cap B$ as above. Then, we need to find its coordinates. Note that $\kappa(\Xi)=t_1+t_2$, $\eta(\Xi)=-x+z+t_2$, $\sigma(\Xi)=-x+y+t_2$.

First, assume that $\kappa(\Xi)=0$. By Proposition~\ref{prop:TypeR}, we have three cases.

Case 1. $\eta(\Xi)<0$.

Then $x\neq 0$. So $y=z=0$ and $x=-\eta(\Xi)=-\sigma(\Xi)$. Then $y_+=a_+(\Xi)=h_+(\Xi)$ and $y_-=a_-(\Xi)=h_-(\Xi)$. Then by $I_{v,+}(\Xi)=x_++y_+$ and $I_{v,-}(\Xi)=x_-+y_-$ we can obtain $x_+$ and $x_-$. So $\Xi$ can be determined.

Case 2. $\eta(\Xi)>0$.

Then $z\neq 0$. So $x=y=0$ and $z=\eta(\Xi)$. Now $y_+=h_+(\Xi)$ and $y_-=h_-(\Xi)$. So by $I_{v,+}(\Xi)=x_++y_++z$ and $I_{v,-}(\Xi)=x_-+y_-+z$, we can obtain the numbers $x_+$ and $x_-$. Hence $\Xi$ can be determined.

Case 3. $\eta(\Xi)=0$.

Then $x=z$. So $x=z=0$ and $y=\sigma(\Xi)$. Now $y_+=a_+(\Xi)$ and $y_-=a_-(\Xi)$. So by $I_{v,+}(\Xi)=x_++y_++y$ and $I_{v,-}(\Xi)=x_-+y_-+y$, we can obtain the numbers $x_+$ and $x_-$. Hence $\Xi$ can be determined.

Then, assume that $\kappa(\Xi)>0$. So $t_1>0$ or $t_2>0$. Then, by Proposition~\ref{prop:TypeR}, we have $x=y_+=0$. So $y=h_+(\Xi)$ and $z=a_+(\Xi)$. There are also three cases.

Case 1. $a_+(\Xi)\neq 0$.

Then $y=t_2=0$ and $t_1=\kappa(\Xi)$. Now $y_-=h_-(\Xi)$. So by $I_{v,+}(\Xi)=x_++z$ and $I_{v,-}(\Xi)=x_-+y_-+z+t_1$, we can obtain $x_+$ and $x_-$, and determine $\Xi$.

Case 2. $h_+(\Xi)\neq 0$.

Then $z=t_1=0$ and $t_2=\kappa(\Xi)$. So, by $a_-(\Xi)=y_-+t_2$, $I_{v,+}(\Xi)=x_++y+t_2$, and $I_{v,-}(\Xi)=x_-+y_-+y+t_2$, we can get $y_-$, $x_+$, and $x_-$, and determine $\Xi$.

Case 3. $a_+(\Xi)=h_+(\Xi)=0$.

Then $\eta(\Xi)=\sigma(\Xi)=t_2$. Then by $a_-(\Xi)=h_-(\Xi)=y_-+t_2$, $I_{v,+}(\Xi)=x_++t_2$, and $I_{v,-}(\Xi)=x_-+y_-+t_1+t_2$, together with $\kappa(\Xi)=t_1+t_2$, we can get $y_-$, $x_+$, $x_-$, and $t_1$. So $\Xi$ can be determined.

Since $\Xi$ determines $\Sigma\cap B$ up to isotopy, we finish the proof.
\end{proof}

\begin{remark}\label{rem:PsMean}
The six cases in the proof correspond to those six edges of the graph in Figure~\ref{fig:graph3}. In each case, there exist at most five types of $\tau$-disks in a block, and the number of $\tau$-disks of a given type can be computed by the functions on $\mathfrak{T}$.

The functions $a_+$, $a_-$, $h_+$, and $h_-$ also have geometric meanings. Let $\Pi$ denote the rectangle that contains the upper (resp. lower) arc from $K$. Let $\Pi'$ denote the rectangle which meets $\Pi$ and the adjacent 2-handle. Then $a_-$ (resp. $a_+$) gives the number of arcs in $\Sigma\cap\Pi$ whose ends lie in the same edge, and $h_+$ (resp. $h_-$) gives the number of arcs in $\Sigma\cap\Pi'$ whose ends lie in the same edge.
\end{remark}

\vspace{3pt}

\noindent {\bf Step 7.} We glue the $\tau$-disks in 0-handles together and simplify $\Sigma$ further. By the step, we can finish the proof of Theorem~\ref{thm:normalForm}.

In each of the four blocks in $[-1,1]\times R_v$, the part of $\Sigma$ gives a vector in $\mathfrak{T}$. Let $\Xi_1$, $\Xi_2$, $\Xi_3$, and $\Xi_4$ be the corresponding four vectors, respectively, as indicated in the left picture of Figure~\ref{fig:Paste}. It is clear that the functions $I_{v,+}$, $I_{v,0}$, and $I_{v,-}$ have the same values at the four vectors, and for $a_+$ and $a_-$ we have $a_+(\Xi_1)=a_+(\Xi_2)$, $a_+(\Xi_3)=a_+(\Xi_4)$, $a_-(\Xi_1)=a_-(\Xi_3)$, and $a_-(\Xi_2)=a_-(\Xi_4)$.

\begin{figure}[h]
\includegraphics{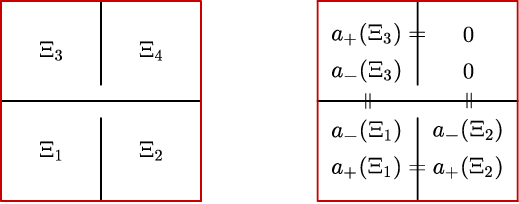}
\caption{The four vectors and their relations.}\label{fig:Paste}
\end{figure}

The four rectangles in $S_D$ form two squares which meet in $[-1,1]\times\{v\}$. Up to the reflections across these two squares, we can assume that $a_+(\Xi_1)\geq a_+(\Xi_3)$ and $a_-(\Xi_1)\geq a_-(\Xi_2)$. If $a_-(\Xi_2)>0$, then by the meaning of $a_-$ given in Remark~\ref{rem:PsMean}, there exist circles from $\Sigma$ in the square containing the upper arc from $K$. Since at most one rectangle in the square contains $\tau$-points, the circles do not meet $K$. By applying the $D^2$-surgery, $\mathcal{I}(\Sigma)$ can then be reduced. Similarly, if $a_+(\Xi_3)>0$, then $\mathcal{I}(\Sigma)$ can be reduced. So, by repeating previous steps, up to the reflections we can require that $a_-(\Xi_2)=a_+(\Xi_3)=0$. See the right picture of Figure~\ref{fig:Paste}.

Then $a_+(\Xi_4)=a_-(\Xi_4)=0$, and according to Table~\ref{tab:Dfunc}, we can assume that
\[\Xi_4=x_+\mathrm{I}_++x\mathrm{I}+x_-\mathrm{I}_-+y\mathrm{II}+t_1\tau_1.\]
If $t_1\neq 0$, then $\kappa(\Xi_2)=\kappa(\Xi_4)=t_1$ and $\kappa(\Xi_1)=\kappa(\Xi_3)=0$. By Proposition~\ref{prop:TypeR}, we have $x=y=0$. Then $I_{v,0}(\Xi_1)=I_{v,0}(\Xi_4)=0$. Then, according to Table~\ref{tab:Para}, we see that $\Xi_1$ is a linear combination of $\mathrm{I}_+$, $\mathrm{I}_-$, and $\mathrm{II}$. So $a_+(\Xi_1)=a_-(\Xi_1)=0$. So, up to the reflections, we can also require that $\kappa(\Xi_4)=0$.

In a block $B$, a $\mathrm{II}$-disk can be obtained from a $\mathrm{I}_+$-disk and a $\mathrm{I}_-$-disk by adding a band along the edge of $B$ that lies in the adjacent 2-handle. By replacing all the $\mathrm{II}$-disks in $[-1,1]\times R_v$ by pairs of $\mathrm{I}_+$-disks and $\mathrm{I}_-$-disks, we can get simpler cases. This may change the vectors $\Xi_1$, $\Xi_2$, $\Xi_3$, and $\Xi_4$, but their values of $a_+$ and $a_-$ do not change. Below we first deal with the simpler cases where there are no $\mathrm{II}$-disks, then we deal with the cases where there exist $\mathrm{II}$-disks.

\vspace{5pt}

\noindent {\bf The cases without $\mathrm{II}$-disks.} We determine all those possible coordinates of the vectors $\Xi_1$, $\Xi_2$, $\Xi_3$, and $\Xi_4$. This gives the four parts of $\Sigma$ in the four blocks up to isotopy. Then, we get all possible $\Sigma\cap[-1,1]\times R_v$, which are isotopic to unions of disjoint normal disks whose subclasses can be identified.

\vspace{5pt}

Because now $\Xi_4=x_+\mathrm{I}_++x\mathrm{I}+x_-\mathrm{I}_-$, we have $I_{v,+}(\Xi_i)=x_+$, $I_{v,0}(\Xi_i)=x$, and $I_{v,-}(\Xi_i)=x_-$ for $1\leq i\leq 4$. Since $a_-(\Xi_2)=\kappa(\Xi_2)=0$, $\Xi_2$ is a linear combination of $\mathrm{I}_+$, $\mathrm{I}$, $\mathrm{I}_-$, and $\mathrm{II}_+$. Let $y_+=a_+(\Xi_1)=a_+(\Xi_2)$. Then we have
\[\Xi_2=(x_+-y_+)\mathrm{I}_++(x-y_+)\mathrm{I}+x_-\mathrm{I}_-+y_+\mathrm{II}_+.\]
Because $a_+(\Xi_3)=0$, $\Xi_3$ is a linear combination of $\mathrm{I}_+$, $\mathrm{I}$, $\mathrm{I}_-$, $\mathrm{II}_-$, $\tau_1$, and $\tau_2$. Then according to the values of $\kappa(\Xi_1)$, $\eta(\Xi_1)$, and $a_+(\Xi_1)$, we have four cases.

\vspace{5pt}

\noindent {\bf Case 1.} $\kappa(\Xi_1)=0$ and $\eta(\Xi_1)\leq 0$.

Let $y_-=a_-(\Xi_1)=a_-(\Xi_3)$. Since $\kappa(\Xi_1)=\kappa(\Xi_3)=0$, we have
\[\Xi_3=x_+\mathrm{I}_++(x-y_-)\mathrm{I}+(x_--y_-)\mathrm{I}_-+y_-\mathrm{II}_-.\]
Note that $\eta=a_++a_--I_{v,0}$. By the proof of Proposition~\ref{prop:PdetD}, we have
\[\Xi_1=(x_+-y_+)\mathrm{I}_++(x-y_+-y_-)\mathrm{I}+(x_--y_-)\mathrm{I}_-+y_+\mathrm{II}_++y_-\mathrm{II}_-.\]
So, up to isotopy, $\Sigma\cap[-1,1]\times R_v$ consists of $x_+-y_+$ disks in $F_+$, $x_--y_-$ disks in $F_-$, $y_+$ disks in $C_+$, $y_-$ disks in $C_-$, and $x-y_+-y_-$ disks in $F_0$.

\vspace{5pt}

\noindent {\bf Case 2.} $\kappa(\Xi_1)=0$ and $\eta(\Xi_1)>0$.

We have $y_-$ and $\Xi_3$ as in Case 1. According to Table~\ref{tab:Dfunc}, now we also have
\[a_+(\Xi_1)=h_+(\Xi_1)+\eta(\Xi_1),\, a_-(\Xi_1)=h_-(\Xi_1)+\eta(\Xi_1).\]
Then, by the proof of Proposition~\ref{prop:PdetD} and $\eta=a_++a_--I_{v,0}$, we have
\[\Xi_1=(x_+-y_+)\mathrm{I}_++(x_--y_-)\mathrm{I}_-+(x-y_-)\mathrm{II}_++(x-y_+)\mathrm{II}_-+(y_++y_--x)\mathrm{III}.\]
So, up to isotopy, $\Sigma\cap[-1,1]\times R_v$ consists of $x_+-y_+$ disks in $F_+$, $x_--y_-$ disks in $F_-$, $y_++y_--x$ disks in $T_-$, $x-y_-$ disks in $C_+$, and $x-y_+$ disks in $C_-$.

\vspace{5pt}

\noindent {\bf Case 3.} $\kappa(\Xi_1)>0$ and $a_+(\Xi_1)=0$.

Now we have $\Xi_2=\Xi_4$. Since $\eta=a_++a_--I_{v,0}$, we have $\eta(\Xi_1)=\eta(\Xi_3)$. Then by the proof of Proposition~\ref{prop:PdetD} we have $\Xi_1=\Xi_3$. Let $y_-=a_-(\Xi_3)$ and $t=\kappa(\Xi_3)$. Then $\Xi_1=\Xi_3$ is given by the following combination
\[(x_+-y_-+x)\mathrm{I}_++(x_--x-t)\mathrm{I}_-+x\mathrm{II}_-+(t-y_-+x)\tau_1+(y_--x)\tau_2.\]
So, up to isotopy, $\Sigma\cap[-1,1]\times R_v$ is a union of $x$ disks in $C_-$, $x_--x-t$ disks in $F_-$, $x_+-y_-+x$ disks in $F_+$, $y_--x$ disks in $C_-^+$, and $t-y_-+x$ disks in $F_-^+$.

\vspace{5pt}

\noindent {\bf Case 4.} $\kappa(\Xi_1)>0$ and $a_+(\Xi_1)\neq 0$.

We have $y_-$, $t$, and $\Xi_3$ as in Case 3. Now by the proof of Proposition~\ref{prop:PdetD}, $\Xi_1$ is a linear combination of $\mathrm{I}_+$, $\mathrm{I}_-$, $\mathrm{II}_-$, $\mathrm{III}$, and $\tau_1$. Then there are no $\tau_2$-points. So
\[\Xi_3=x_+\mathrm{I}_++(x_--x-t)\mathrm{I}_-+x\mathrm{II}_-+t\tau_1.\]
According to Table~\ref{tab:Dfunc}, $a_-(\Xi_1)=h_-(\Xi_1)+a_+(\Xi_1)$. So $h_-(\Xi_1)=x-y_+$, and
\[\Xi_1=(x_+-y_+)\mathrm{I}_++(x_--x-t)\mathrm{I}_-+(x-y_+)\mathrm{II}_-+y_+\mathrm{III}+t\tau_1.\]
Hence, up to isotopy, $\Sigma\cap[-1,1]\times R_v$ consists of $x_+-y_+$ disks in $F_+$, $x_--x-t$ disks in $F_-$, $t$ disks in $F_-^+$, $y_+$ disks in $T_-$, and $x-y_+$ disks in $C_-$.

This finishes the discussion of the cases where there are no $\mathrm{II}$-disks.

\vspace{5pt}

\noindent {\bf The cases with $\mathrm{II}$-disks.} Starting from the normal disks described in Cases~1-4, we determine all the possible ways to add bands in the four blocks. This gives the remaining cases where there exist $\mathrm{II}$-disks, and the corresponding $\Sigma\cap[-1,1]\times R_v$ are unions of disjoint normal disks whose subclasses can be identified.

\vspace{5pt}

Note that when there exist $\mathrm{II}$-disks in a block $B$, both $h_+$ and $h_-$ have nonzero values. So, in $[-1,1]\times R_v$, if two blocks share a rectangle, then by the meaning of $h_{\pm}$ given in Remark~\ref{rem:PsMean}, only one block can contain $\mathrm{II}$-disks. Otherwise, $\Sigma$ cannot meet 1-handles in normal disks. According to Cases 1-4, we have Cases 5-8.

\vspace{5pt}

\noindent {\bf Case 5.} We add bands to the disks in $F_+$, $F_-$, $C_+$, $C_-$, and $F_0$ given in Case~1.

If there exist disks in $F_0$, then each block contains $\mathrm{I}$-disks. So, no bands can be added. If there are disks in $C_+$ or $C_-$, then the blocks that do not contain $\mathrm{I}$-disks share rectangles in $S_D$. So, we can only add bands in one block. Then in this case $\Sigma\cap[-1,1]\times R_v$ consists of normal disks in $F_+$, $F_-$, $C_+$, $C_-$, and $C_0$.

Now, assume that there are no disks in $F_0$, $C_+$, and $C_-$, and we can add bands in two blocks. Then the two blocks only meet in $[-1,1]\times\{v\}$. Then, after adding bands, the lowermost disk in $F_+$ and uppermost disk in $F_-$ become an annulus in $\Sigma$. So this annulus lies in some sphere $S_j\subseteq\Sigma$.

Consider the rectangle $\Pi$ in $[-1,1]\times R_v$ that contains $[-1,1]\times\{v\}$ and bisects each of the two blocks. Then, the annulus intersects $\Pi$ in a circle, which bounds a disk $D^2$ in $\Pi$; see the left picture of Figure~\ref{fig:AddTB}. Note that $D^2\cap S_j=\partial D^2$ and $\partial D^2$ splits $S_j$ into two disks. Then, since $D^2\cap K$ consists of two points, one of the two disks does not meet $K$. Together with $D^2$, it gives a sphere which meets $K$ in two points and splits $K$ into two arcs $E_1$ and $E_2$. Let $\mu=[-1/2,1/2]\times\{v\}$. Then we have two knots $\widehat{E}_i=E_i\cup\mu$ where $i=1,2$, and $K=\widehat{E}_1\#\widehat{E}_2$. Since the projection of $\widehat{E}_i$ to $S_0$ gives a diagram $Q_i$ of $\widehat{E}_i$ for $i=1,2$, from $Q_1$ and $Q_2$ we can obtain a new diagram of $K$ having fewer crossings than $D$. This is impossible.

So, as above, we can only add bands in one block, where the bands give normal disks in $C_0$. Also, note that up to the reflections across the two squares in $S_D$, we can assume that the block corresponding to $\Xi_1$ contains $\mathrm{II}$-disks.

\begin{figure}[h]
\includegraphics{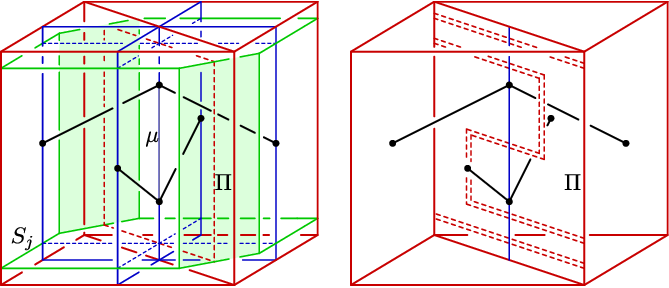}
\caption{Possible cases when we can add bands in two blocks.}\label{fig:AddTB}
\end{figure}

\vspace{5pt}

\noindent {\bf Case 6.} We add bands to the disks in $F_+$, $F_-$, $C_+$, $C_-$, and $T_-$ given in Case~2.

In Case~2, there exist disks in $T_-$. Then there are at most two blocks where we can add bands. The two blocks correspond to $\Xi_2$ and $\Xi_3$, respectively. If there are disks in $C_+$ and disks in $C_-$, then no bands can be added. Below we first consider the case where there are no disks in $C_+$ and $C_-$, then we consider the cases where there are no disks in exactly one of $C_+$ and $C_-$.

Case (a). There are no disks in $C_+$ and $C_-$.

Let $\Pi$ be the rectangle which bisects each of the two blocks, as in Case~5. Then those disks in $F_+$, $F_-$, and $T_-$ intersect $\Pi$ in some arcs which have certain shapes; see the right picture of Figure~\ref{fig:AddTB}. Now if we add bands between the disks, then in $\Pi$ we will have vertical arcs between the arcs coming from the disks. Assume that we have $t_+$, $t_-$, and $t_0$ disks in $F_+$, $F_-$, and $T_-$, respectively. In the above picture we add $b_-$ and $b_+$ bands along the left and the right edges of $\Pi$, respectively.

Since $t_0>0$, there exist integers $q_+,q_-\geq 0$ and $0\leq r_+,r_-<t_0$ such that
\[b_+=q_+t_0+r_+,\, b_-=q_-t_0+r_-.\]
Let $q=\min\{q_+, q_-\}$. Then we can add these bands in two steps. First we add $qt_0$ bands along each of the two edges. The intersections of the resulting disks with $\Pi$ have a pattern as shown in the left picture of Figure~\ref{fig:TwistA}, where the uppermost and lowermost arcs represent several horizontal arcs, and the middle part gives one arc, which represents $t_0$ parallel arcs. Then we add the remaining bands. According to whether $q_+=q_-$, $q_+>q_-$, or $q_+<q_-$, there are three cases.

Case (a1). $q_+=q_-$.

We also need to add $r_-$ bands along the left edge of $\Pi$, and $r_+$ bands along the right edge of $\Pi$. After adding the bands, we also have normal disks, as follows.

If $r_++r_-\leq t_0$, then we get $t_0-r_+-r_-$ disks in $T_-^{q,q}$, $r_+$ disks in $T_-^{q+1,q}$, and $r_-$ disks in $T_-^{q,q+1}$; and when $r_++r_->t_0$, we get $r_++r_--t_0$ disks in $T_-^{q+1,q+1}$, $t_0-r_-$ disks in $T_-^{q+1,q}$, and $t_0-r_+$ disks in $T_-^{q,q+1}$. Moreover, in each case, there are also $t_+-b_+$ disks in $F_+$ and $t_--b_-$ disks in $F_-$.

Case (a2). $q_+>q_-$.

We also need to add $r_-$ bands and $(q_+-q)t_0+r_+$ bands along the edges of $\Pi$, respectively. We first add $t_0$ bands along the right edge. This gives a new pattern of the intersections as shown in the right picture of Figure~\ref{fig:TwistA}, and we also need to add $r_+'=(q_+-q-1)t_0+r_+$ bands along the right edge of $\Pi$. After adding those remaining bands, we also have normal disks, as follows.

If $r_-\geq r_+'$, then we get $t_0-r_-$ disks in $T_-^{q+1,q}$, $r_--r_+'$ disks in $T_-^{q+1,q+1}$, and $r_+'$ disks in $T_-^{q+2,q+1}$; and if $r_-<r_+'$, then we get $t_0-r_-$ disks in $T_-^{q+1,q}$, $r_-$ disks in $T_-^{q+2,q+1}$, and $r_+'-r_-$ disks in $C_0$. Moreover, in each case we also have $t_+-b_+$ disks in $F_+$ and $t_--\max\{b_-,b_+-t_0\}$ disks in $F_-$.

Case (a3). $q_+<q_-$.

The discussion is similar to Case~(a2). In this case we have normal disks in $F_+$, $F_-$, $T_-^{q,q+1}$, $T_-^{q+1,q+1}$, and $T_-^{q+1,q+2}$, or, in $F_+$, $F_-$, $T_-^{q,q+1}$, $T_-^{q+1,q+2}$, and $C_0$.

This finishes the discussion of Case~(a).

\begin{figure}[h]
\includegraphics{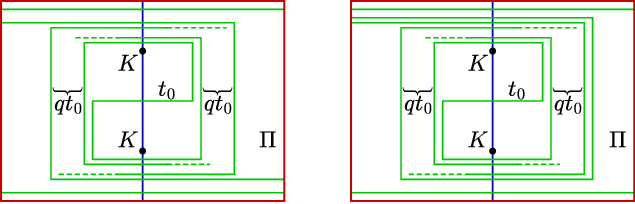}
\caption{Two intersection patterns of the disks with $\Pi$.}\label{fig:TwistA}
\end{figure}

Case (b). There are no disks in $C_-$, but there exist disks in $C_+$.

Now, the block corresponding to $\Xi_3$ cannot contain $\mathrm{II}$-disks. We can add bands only in the block corresponding to $\Xi_2$. Consider the intersection pattern of all the disks with the rectangle $\Pi$ as in Case~(a). As above, assume that we have $t_0$ disks in $T_-$, and we add $b_+$ bands along the right edge of $\Pi$. Then, as in Case~(a), after adding the bands, we also have normal disks. When $t_0\geq b_+$, we have disks in $F_+$, $F_-$, $C_+$, $T_-$, and $T_-^{1,0}$; otherwise, we have disks in $F_+$, $F_-$, $C_+$, $T_-^{1,0}$, and $C_0$.

Case (c). There are no disks in $C_+$, but there exist disks in $C_-$.

The discussion is similar to Case~(b). Now, we get normal disks in $F_+$, $F_-$, $C_-$, $T_-$, and $T_-^{0,1}$, or normal disks in $F_+$, $F_-$, $C_-$, $T_-^{0,1}$, and $C_0$.

\vspace{5pt}

\noindent {\bf Case 7.} We add bands to the disks in $F_+$, $F_-$, $C_-$, $C_-^+$, and $F_-^+$ given in Case~3.

In Case~3, there are disks in $C_-^+$ or $F_-^+$. By Proposition~\ref{prop:TypeR} and the meaning of $h_-$ given in Remark~\ref{rem:PsMean}, if there exist disks in $C_-^+$ and disks in $F_-^+$, then no bands can be added. So, there are two cases, as follows.

Case (a). There exist disks in $C_-^+$, and there are no disks in $F_-^+$.

We can add bands only in one of the two blocks that contain $\tau_2$-points, and we can only get new disks in $C_0$. Hence we get normal disks in $F_+$, $F_-$, $C_-$, $C_-^+$, and $C_0$. Note that, up to reflections across the two squares in $S_D$, we can assume that the block corresponding to $\Xi_1$ contains $\mathrm{II}$-disks.

Case (b). There exist disks in $F_-^+$, and there are no disks in $C_-^+$.

If there are disks in $C_-$, then no bands can be added. So, we assume that there are no disks in $C_-$. Now we can add bands only in one of the two blocks which do not contain $\tau_1$-points. Assume that we have $t$ disks in $F_-^+$ and add $b$ bands. Then we get normal disks in $F_+$, $F_-$, $F_-^+$, and $C_0^+$ if $t\geq b$; and we have normal disks in $F_+$, $F_-$, $C_0^+$, and $C_0$ if $t<b$. Also, up to reflections across the two squares in $S_D$, we can assume that the block corresponding to $\Xi_2$ contains $\mathrm{II}$-disks.

\vspace{5pt}

\noindent {\bf Case 8.} We add bands to the disks in $F_+$, $F_-$, $C_-$, $T_-$, and $F_-^+$ given in Case~4.

In Case 4, there are disks in $T_-$ and disks in $F_-^+$. If there are disks in $C_-$, then no bands can be added. So, we assume that there are no disks in $C_-$. Then, there is only one block where we can add bands. The block corresponds to $\Xi_2$.

Assume that there are $t_0$ disks in $T_-$ and $t$ disks in $F_-^+$, and we add $b$ bands. If $b<t_0$, then we have normal disks in $F_+$, $F_-$, $T_-$, $F_-^+$, and $T_-^{1,0}$; if $t_0\leq b<t_0+t$, then we have normal disks in $F_+$, $F_-$, $F_-^+$, $T_-^{1,0}$, and $C_0^+$; and when $b\geq t_0+t$, we have normal disks in $F_+$, $F_-$, $T_-^{1,0}$, $C_0^+$, and $C_0$.

This finishes the discussion of the cases where there exist $\mathrm{II}$-disks.

\vspace{5pt}

So, up to isotopy, we can also require that $\Sigma$ intersects the 0-handles in disjoint normal disks, and we finish the proof of Theorem~\ref{thm:normalForm}.

\begin{remark}\label{rem:NDtype}
In the proof, we see that there are at most five kinds of normal disks that can appear simultaneously in a 0-handle. In Table~\ref{tab:subclass} we summarize the cases, where $z$, $t_1$, and $t_2$ are the numbers of disks in $T_-$, $F_-^+$, and $C_-^+$, respectively.

\begin{table}[h]
\caption{The subclasses that can appear simultaneously.}\label{tab:subclass}
\centerline{
\begin{tabular}{l|l||l|l}
\hline
Case~1 & $F_+$, $F_-$, $C_+$, $C_-$, $F_0$ & Case~6 & $F_+$, $F_-$, $C_+$, $T_-$, $T_-^{1,0}$\\ \hline
Case~2 & $F_+$, $F_-$, $C_+$, $C_-$, $T_-$ ($z\neq 0$) & & $F_+$, $F_-$, $C_+$, $T_-^{1,0}$, $C_0$\\ \hline
Case~3 & $F_+$, $F_-$, $C_-$, $C_-^+$, $F_-^+$ ($t_1+t_2\neq 0$) & & $F_+$, $F_-$, $C_-$, $T_-$, $T_-^{0,1}$\\ \hline
Case~4 & $F_+$, $F_-$, $C_-$, $T_-$, $F_-^+$ ($zt_1\neq 0$) & & $F_+$, $F_-$, $C_-$, $T_-^{0,1}$, $C_0$\\ \hline
Case~5 & $F_+$, $F_-$, $C_+$, $C_-$, $C_0$ & Case~7 & $F_+$, $F_-$, $C_-$, $C_-^+$, $C_0$\\ \hline
Case~6 & $F_+$, $F_-$, $T_-^{q+1,q}$, $T_-^{q,q+1}$, $T_-^{q,q}$ & & $F_+$, $F_-$, $F_-^+$, $C_0^+$\\ \hline
       & $F_+$, $F_-$, $T_-^{q+1,q+1}$, $T_-^{q+1,q}$, $T_-^{q,q+1}$ & & $F_+$, $F_-$, $C_0^+$, $C_0$\\ \hline
       & $F_+$, $F_-$, $T_-^{q+2,q+1}$, $T_-^{q+1,q+1}$, $T_-^{q+1,q}$ & Case~8 & $F_+$, $F_-$, $T_-$, $F_-^+$, $T_-^{1,0}$\\ \hline
       & $F_+$, $F_-$, $T_-^{q+2,q+1}$, $T_-^{q+1,q}$, $C_0$ & & $F_+$, $F_-$, $F_-^+$, $T_-^{1,0}$, $C_0^+$\\ \hline
       & $F_+$, $F_-$, $T_-^{q+1,q+2}$, $T_-^{q+1,q+1}$, $T_-^{q,q+1}$ & & $F_+$, $F_-$, $T_-^{1,0}$, $C_0^+$, $C_0$\\ \hline
       & $F_+$, $F_-$, $T_-^{q+1,q+2}$, $T_-^{q,q+1}$, $C_0$ & & \\ \hline
\end{tabular}}
\end{table}

Note that Cases~1-4 correspond to four triangles in the graph in Figure~\ref{fig:graph1}, and Cases~5-8 can similarly provide more triangles if we also consider the normal disks with bands. Also, there are no subclasses $T_+$, $T_+^{r,s}$, $F_+^-$, $C_+^-$, and $C_0^-$ in the proof. For $T_+$ and $T_+^{r,s}$, this is because we have used the reflections across the squares in $S_D$. Disks in these subclasses may actually appear in 0-handles. For $F_+^-$, $C_+^-$, and $C_0^-$, this is because we have chosen the direction ``up'' in Section~\ref{subsec:Moves}. If we choose the direction ``down'', then we have $F_+^-$, $C_+^-$, and $C_0^-$ instead of $F_-^+$, $C_-^+$, and $C_0^+$, respectively. See also Remark~\ref{rem:updown}.
\end{remark}

By Remark~\ref{rem:updown}, we have similar 0-moves and 1-moves if we choose the direction ``down''. It is possible that $\mathcal{I}(\Sigma)$ can be further reduced after applying such moves, as in the paragraphs before Propositions~\ref{prop:0move} and \ref{prop:1move}. If the case does not happen, we can get parallel propositions, with the signs $+$ and $-$ interchanged.

\begin{definition}\label{def:simNF}
For the normal surface $\Sigma$ obtained in the proof of Theorem~\ref{thm:normalForm}, we can require that $\mathcal{I}(\Sigma)$ cannot be reduced by applying sequences of the possible 0-moves and 1-moves. We call such $\Sigma$ having the {\it simple normal form}.
\end{definition}


\section{Parallel pieces, nonparallel pieces, and $X$-pieces}\label{sec:FDH}
Let $K$, $D$, $\mathcal{H}_D$, and $\Sigma$ be as in previous sections. Here $\Sigma$ has the simple normal form. Then, by cutting $S^3$ along $\Sigma$, we have manifold pairs $(M_i,A_i)$, $1\leq i\leq n$, as described in Section~\ref{sec:MSS}. Because $\partial M_i$ meets the handles of $\mathcal{H}_D$ in normal disks, $M_i$ meets each handle of $\mathcal{H}_D$ in some 3-balls. So we have the following.

\begin{definition}\label{def:indHS}
For a given $(M_i,A_i)$ and a $j$-handle $H^j$ of $\mathcal{H}_D$, where $1\leq i\leq n$, $0\leq j\leq 3$, define any component of their intersection to be a $j$-handle in $(M_i,A_i)$. Then all such handles together give a handle structure for $(M_i,A_i)$. We call it the {\it induced handle structure} for $(M_i,A_i)$ and denote it by $\mathcal{H}_{D,i}$.
\end{definition}

Clearly, for a fixed $0\leq j\leq 3$, two $j$-handles of $\mathcal{H}_{D,i}$ are disjoint. Note that two handles of $\mathcal{H}_D$ meet in at most one disk and each such disk is divided into several disks by $\Sigma$. Then, any 1-handle of $\mathcal{H}_{D,i}$ meets 0-handles of $\mathcal{H}_{D,i}$ in two disks, and any 2-handle of $\mathcal{H}_{D,i}$ meets the union of all 1-handles and 0-handles of $\mathcal{H}_{D,i}$ in an annulus. Also, the following proposition can be checked easily.

\begin{proposition}\label{prop:goodhandle2}
If two handles of $\mathcal{H}_{D,i}$ meet, then they meet in one disk.
\end{proposition}

In this section, we first divide 1-handles and 0-handles of $\mathcal{H}_{D,i}$ in a certain way into smaller pieces, and we define the parallel pieces and nonparallel pieces. Then, we identify each component of the union of nonparallel pieces in a handle of $\mathcal{H}_{D,i}$, and we define the model of the component, and $X$-pieces. Finally, we simplify the models and modify the $X$-pieces and $(M_i,A_i)$ correspondingly.

\begin{figure}[h]
\includegraphics{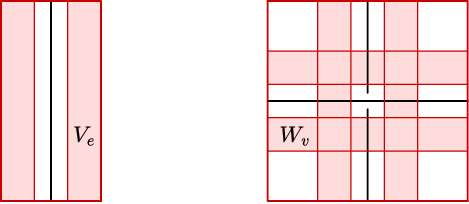}
\caption{Cut $R_e$ and $R_v$ along $G_e$ and $G_v$, respectively.}\label{fig:divide}
\end{figure}

To divide those handles of $\mathcal{H}_{D,i}$, we first cut the 1-handles and 0-handles of $\mathcal{H}_D$ into some blocks as follows. For a 1-handle $([-1,1]\times R_e,A_e)$, we identify $R_e$ with $[-1/2,1/2]\times[-1,1]$. Then, by cutting $[-1,1]\times R_e$ along $[-1,1]\times G_e$, where
\[G_e=\{-1/6,1/6\}\times[-1,1]\subset R_e,\]
we get $3$ blocks. In a similar way, for a 0-handle $([-1,1]\times R_v,A_v)$, we identify $R_v$ with $[-1,1]\times[-1,1]$. Then, by cutting $[-1,1]\times R_v$ along $[-1,1]\times G_v$, where \[G_v=(\{\pm1/2,\pm1/6\}\times[-1,1])\cup([-1,1]\times\{\pm1/2,\pm1/6\})\subset R_v,\]
we have $25$ blocks. In Figure~\ref{fig:divide}, we show the straight arcs in $G_e$ and $G_v$. We also view the pictures as the projections of $[-1,1]\times G_e$ and $[-1,1]\times G_v$.

Then, we let $J=(-1/2,-1/6)\cup(1/6,1/2)$, and we define the regions
\[V_e=J\times[-1,1]\subset R_e,\, W_v=(J\times[-1,1])\cup([-1,1]\times J)\subset R_v.\]
See the shaded regions in the pictures of Figure~\ref{fig:divide}. Recall that $\Sigma$ meets 1-handles and 0-handles in unions of normal disks, which have diagrams; see Definition~\ref{def:diagram}. Up to isotopy, we can assume that the parts of the diagrams which come from the vertical disks lie in $V_e$ and $W_v$, and they give certain patterns, as described below; see Figure~\ref{fig:case0} and Figures~\ref{fig:case1}-\ref{fig:case678} for some examples.

Now $G_e$ (resp. $G_v$) cuts $R_e$ (resp. $R_v$) into some rectangles. Let $U$ denote such a rectangle. It corresponds to a block $[-1,1]\times U$. Then, let $H$ denote a handle of $\mathcal{H}_{D,i}$, where $H$ and $[-1,1]\times U$ lie in the same handle of $\mathcal{H}_D$.

\begin{definition}\label{def:piece}
We call any component of $H\cap[-1,1]\times U$ a {\it piece}.

Let $B$ denote a piece, and let $\Sigma_0=\Sigma\cup\{\pm1\}\times S_0$. We call $B$ {\it parallel} if

(1) $B\cap K=\emptyset$ and $(B,B\cap\Sigma_0)$ is homeomorphic to $(D^2\times[0,1],D^2\times\{0,1\})$,

(2) the vertical part in $\Sigma\cap B$ projects to two parallel arcs in the rectangle $U$ or it projects to two (parallel) figures ``$\textsf{Y}$'' in the rectangle $U$.

Otherwise, we call $B$ {\it nonparallel}.
\end{definition}

Note that if $B$ does not meet those vertical disks in $\Sigma$, then condition~(2) holds automatically. By the definition, 1-handles and 0-handles of $\mathcal{H}_{D,i}$ are divided into many parallel and nonparallel pieces.

\vspace{8pt}

Below we identify all those possible components of the union of the nonparallel pieces in $H$. Each component will be a 3-ball, which has some peculiar shape. So, according to the shape, we can define a model of the component, and we call such a component with model $X$ an $X$-piece. There are nine cases: one for the possible $H$ lying in 1-handles of $\mathcal{H}_D$ and eight for the possible $H$ lying in 0-handles of $\mathcal{H}_D$. The eight cases correspond to those ones in Table~\ref{tab:subclass}.

\vspace{5pt}

\noindent {\bf Case 0.} The handle $H$ lies in some $([-1,1]\times R_e,A_e)$.

Up to the reflection across $[-1,1]\times A_e$, we can require that all possible vertical disks in $\Sigma\cap H$ project to parallel arcs in the rectangle $[1/6,1/2]\times[-1,1]\subset R_e$.

\begin{figure}[h]
\includegraphics{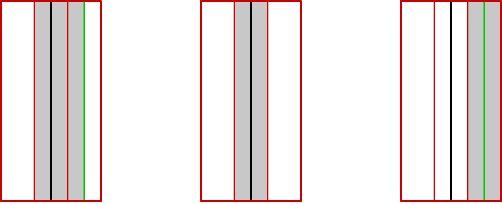}
\caption{The three models $Z_1$, $Z_2$, and $Z_3$ appearing in Case~0.}\label{fig:case0}
\end{figure}

If $H\cap\Sigma_0$ has only one component, then $A_e\subset H$, and there are two nonparallel pieces in $H$. In the left picture of Figure~\ref{fig:case0}, the projection of the two pieces gives the shaded region, where the green line indicates a vertical disk. The union of the two pieces is a 3-ball. We call such a 3-ball having model $Z_1$.

If $H\cap\Sigma_0$ has exactly two components, then such components are parallel disks in $[-1,1]\times R_e$. If the disks are curved, then by Definition~\ref{def:piece}, the pieces in $H$ are all parallel. In the other cases, $H$ contains nonparallel pieces only if $A_e\subset H$. The middle picture of Figure~\ref{fig:case0} gives the projection of the only nonparallel piece. The piece is a 3-ball. We call such a 3-ball having model $Z_2$.

Otherwise $H\cap\Sigma_0$ has exactly three components, and there are five pieces in $H$ where only one is nonparallel. The right picture of Figure~\ref{fig:case0} shows the projection of the piece. The piece is a 3-ball. We call such a 3-ball having model $Z_3$.

In Figure~\ref{fig:case0F}, we give some examples of the vertical sections of $[-1,1]\times R_e$ and the pieces, where the green arcs indicate the normal disks coming from $\Sigma$, and the shaded regions indicate the $Z_1$-piece, $Z_2$-piece, and $Z_3$-piece, respectively.

\begin{figure}[h]
\includegraphics{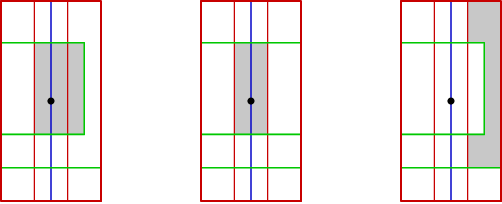}
\caption{The vertical sections of $[-1,1]\times R_e$ and the $X$-pieces.}\label{fig:case0F}
\end{figure}

Now we consider those cases when the handle $H$ lies in some $([-1,1]\times R_v,A_v)$. Let $\Psi=\Sigma\cap[-1,1]\times R_v$. It is a union of normal disks. According to the proof of Theorem~\ref{thm:normalForm}, up to the reflections across the two squares in $S_D\cap[-1,1]\times R_v$, we have eight cases for all possible $\Psi$: four without bands and four with bands. Then we also have eight cases below corresponding to those ones in Table~\ref{tab:subclass}.

Note that when $H\cap\Sigma_0$ consists of a pair of parallel normal disks, or it consists of some disks in $F_0$, $F_{\pm}$, or $\{\pm1\}\times R_v$, a piece $B$ in $H$ is nonparallel if and only if $B\cap K\neq\emptyset$. Then, by viewing $\{\pm1\}\times R_v$ as two disks in $F_{\pm}$, respectively, in most cases below, we only need to focus on those cases where disks in $H\cap\Sigma_0$ belong to different subclasses. Also, for simplicity, we only provide the projection pictures of the nonparallel pieces and $X$-pieces, as in Figure~\ref{fig:case0}.

\vspace{5pt}

\noindent {\bf The cases without bands.} Now, $\Psi$ is a union of some disks in $F_{\pm}$, $F_0$, $C_{\pm}$, $T_-$, $F_-^+$, and $C_-^+$. Assume that $\Psi$ contains exactly $x$, $y_+$, $y_-$, $z$, $t_1$, and $t_2$ disks in $F_0$, $C_+$, $C_-$, $T_-$, $F_-^+$, and $C_-^+$, respectively. According to Table~\ref{tab:subclass}, we have four cases.

\vspace{5pt}

\noindent {\bf Case 1.} $z=t_1=t_2=0$.

According to whether $x$, $y_+$, or $y_-$ is zero, there are five subcases.

Case 1(0). $x\neq 0$.

Let $D^2\subseteq\Psi$ be a disk in $F_0$. Then, it splits $[-1,1]\times R_v$ into two 3-balls, which are essentially the same as 1-handles of $\mathcal{H}_D$. The interiors of the 3-balls meet $\Psi$ in some flat and curved disks, where vertical disks project to straight arcs in $W_v$. So, as in Case~0, if $H$ contains nonparallel pieces, then the union of nonparallel pieces in $H$ is a 3-ball. According to whether $H\cap\Sigma_0$ has one, two, or three components, we also call such a 3-ball having model $Z_1$, $Z_2$, or $Z_3$, respectively.

Case 1(1). $x=0$, $y_+\neq 0$, and $y_-\neq 0$.

If $H\cap\Sigma_0$ has at most two components, then we meet some situation which has been discussed in Case~1(0). Otherwise, $H\cap\Sigma_0$ has exactly four components, and there are exactly one disk in $C_+$ and one disk in $C_-$. We can require that the two vertical disks in the two curved disks project to two straight arcs in $W_v$, as shown in the left picture of Figure~\ref{fig:case1}. Then one can check that there are $55$ pieces in $H$, where $9$ pieces are nonparallel. See the shaded region in the picture. The union of the $9$ pieces is a 3-ball. We call such a 3-ball having model $X_4$.

Case 1(2). $x=0$, $y_+=0$, and $y_-\neq 0$.

As in Case~1(1), we only need to consider the case when $H\cap\Sigma_0$ has more than two components. Then $H\cap\Sigma_0$ has exactly three components, and there is exactly one disk in $C_-$. Then we can require that the vertical disk projects to the straight arc in $W_v$ as shown in the middle picture of Figure~\ref{fig:case1}. Now there are $40$ pieces in $H$, and among them, $9$ pieces are nonparallel. See the picture. The union of the $9$ pieces is a 3-ball. We call such a 3-ball having model $X_3$.

Case 1(3). $x=0$, $y_+\neq 0$, and $y_-=0$.

By applying an isometry of $[-1,1]\times R_v$ which interchanges the two arcs of $A_v$, we see that this case is essentially the same as Case~1(2). When $H\cap\Sigma_0$ has three components, we can have a picture similar to the middle one in Figure~\ref{fig:case1}. All the nonparallel pieces in $H$ give a 3-ball. We also call this 3-ball having model $X_3$.

Case 1(4). $x=y_+=y_-=0$.

Now, $H\cap\Sigma_0$ has exactly two components, which are parallel in $[-1,1]\times R_v$. If $H$ contains nonparallel pieces, then $A_v\subset H$. The nonparallel pieces are the pieces that meet $A_v$. The right picture of Figure~\ref{fig:case1} gives their projections. The union of these pieces is a 3-ball. We call such a 3-ball having model $X_2$.

\begin{figure}[h]
\includegraphics{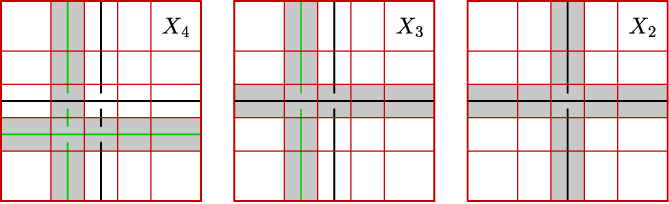}
\caption{The three models $X_4$, $X_3$, and $X_2$ appearing in Case~1.}\label{fig:case1}
\end{figure}

\noindent {\bf Case 2.} $x=t_1=t_2=0$ and $z\neq 0$.

Let $D^2\subseteq\Psi$ be a disk in $T_-$. Then, it splits $[-1,1]\times R_v$ into two 3-balls, which are essentially the same, up to Euclidean isometries of $[-1,1]\times R_v$. Hence we can require that $H$ lies in the 3-ball which contains the upper arc of $A_v$. For the other 3-ball, we have parallel results and the same models.

Then, according to whether $y_+$ is zero, there are two subcases.

Case 2(1). $y_+=0$.

We can assume that the diagram of each twisted disk in $\Psi$ lies in $W_v$ and gives a pattern as in the left picture of Figure~\ref{fig:case2}. If there are many such disks, then we have ``parallel'' figures ``$\textsf{Y}$'', whose branch points lie in the same rectangle. Denote this rectangle by $U$. If $H\cap K=\emptyset$, then either $H$ lies between two parallel normal disks, or it lies between $\{-1\}\times R_v$ and a disk in $F_-$. Then, by Definition~\ref{def:piece}, the pieces in $H$ are all parallel. Note that if $H$ lies between two twisted disks, then $H$ meets $[-1,1]\times U$ in a piece $B$ which is ``twisted'', but $B$ is parallel.

So we assume that $H\cap K\neq\emptyset$. Then the case is essentially the same as the one shown in Figure~\ref{fig:example2}. We have $37$ pieces in $H$, where $11$ pieces are nonparallel. Note that now $H$ meets $[-1,1]\times U$ in a piece $B$ which satisfies Definition~\ref{def:piece}(1), but it does not satisfy Definition~\ref{def:piece}(2). So $B$ is nonparallel. We also give the projection of the union of nonparallel pieces in the left picture of Figure~\ref{fig:case2}. This union gives a 3-ball. We call such a 3-ball having model $Y_2$.

Case 2(2). $y_+\neq 0$.

Now $\Psi$ contains disks in $C_+$, where the vertical disks project to straight arcs in $W_v$. If $H\cap\Sigma_0$ has exactly one component, then it is a disk in $C_+$. Then, we meet a situation which has been discussed in Case~1(0). When $H\cap\Sigma_0$ contains exactly two components, $H$ lies between two parallel disks. So all pieces in $H$ are parallel. Otherwise, $H\cap\Sigma_0$ has exactly three components, which contain a disk in $C_+$ and a disk in $T_-$. We can assume that the vertical parts in $H\cap\Sigma_0$ project to two arcs as shown in the right picture of Figure~\ref{fig:case2}. Then, we have $52$ pieces in $H$, where $5$ pieces are nonparallel, as shown in the picture. The union of these $5$ pieces gives a 3-ball. We call such a 3-ball having model $YZ_3$.

\begin{figure}[h]
\includegraphics{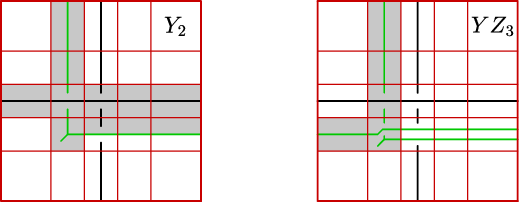}
\caption{The two models $Y_2$ and $YZ_3$ appearing in Case~2.}\label{fig:case2}
\end{figure}

\noindent {\bf Case 3.} $x=y_+=z=0$ and $t_1+t_2\neq 0$.

Now $\Psi\cap K\neq\emptyset$. We can require that the intersection points coming from disks in $F_-^+$ project into the rectangle which is not a square, and the vertical parts in $\Psi$ project to straight arcs in $W_v$; see the pictures below.

According to whether $y_-$, $t_1$, or $t_2$ is zero, there are three subcases.

Case 3(1). $y_-=0$, $t_1\neq 0$, and $t_2=0$.

Since all disks in $\Psi$ are flat, the nonparallel pieces are the ones that meet $K$. If $H$ contains nonparallel pieces, then we meet one of the three situations below.

(a) $H\cap K$ is a union of two arcs. There are $9$ nonparallel pieces, and the union of these pieces is a 3-ball. We call such a 3-ball having model $X_2'$.

(b) $H\cap K$ is one arc, and $H\cap\Sigma_0\cap K$ consists of two points. There is only one nonparallel piece, which is a 3-ball. We call such a 3-ball having model $O_2''$.

(c) $H\cap K$ is one arc, and $H\cap\Sigma_0\cap K$ is a point. There is only one nonparallel piece, which is a 3-ball. We call such a 3-ball having model $O_2'$.

The shaded regions in the left, middle, and right pictures of Figure~\ref{fig:case3a} show the projections of the $X_2'$-piece, $O_2''$-piece, and $O_2'$-piece, respectively.

Case 3(2). $y_-=0$, $t_1=0$, and $t_2\neq 0$.

As in Case~3(1), when $H$ contains nonparallel pieces, we have three situations.

(a) $H\cap\Sigma_0$ has three components. There are $6$ nonparallel pieces. The union of these pieces is a 3-ball. We call such a 3-ball having model $Y_3'$.

(b) $H\cap\Sigma_0$ has exactly two components. There is exactly one nonparallel piece, which is a 3-ball. We call such a 3-ball having model $\bar{O}_2''$.

(c) $H\cap\Sigma_0$ has only one component. There are $12$ nonparallel pieces, and these pieces give a 3-ball. We call such a 3-ball having model $Y_1'$.

See pictures of Figure~\ref{fig:case3b} for the $Y_3'$-piece, $\bar{O}_2''$-piece, and $Y_1'$-piece, respectively.

Case 3(3). At most one of $y_-$, $t_1$, and $t_2$ is zero.

By above cases, we only need to consider the following three situations.

(a) $y_-=0$ and $H\cap\Sigma_0$ has three components. We have a new model $Y_3''$.

(b) $t_2=0$ and $H\cap\Sigma_0$ has three components. We have a new model $X_3'$.

(c) $y_-\neq 0$, $t_2\neq 0$, and $H\cap\Sigma_0$ contains one disk in $C_-$ and one disk in $C_-^+$. In this situation, there are $4$ nonparallel pieces, and we have a new model $Z_2'$.

See pictures of Figure~\ref{fig:case3c} for the $Y_3''$-piece, $X_3'$-piece, and $Z_2'$-piece, respectively.

\begin{figure}[h]
\includegraphics{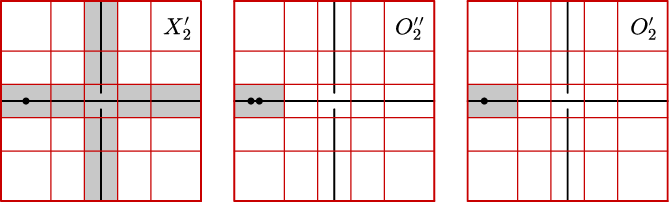}
\caption{The three models $X_2'$, $O_2''$, $O_2'$ appearing in Case~3(1).}\label{fig:case3a}
\end{figure}

\begin{figure}[h]
\includegraphics{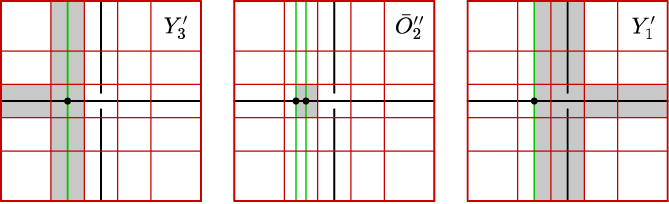}
\caption{The three models $Y_3'$, $\bar{O}_2''$, $Y_1'$ appearing in Case~3(2).}\label{fig:case3b}
\end{figure}

\begin{figure}[h]
\includegraphics{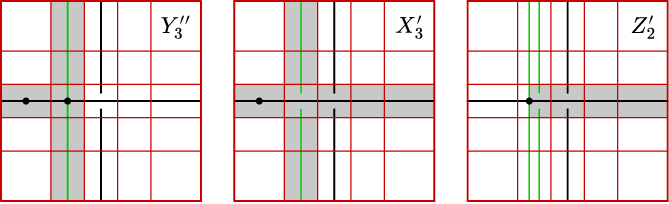}
\caption{The three models $Y_3''$, $X_3'$, $Z_2'$ appearing in Case~3(3).}\label{fig:case3c}
\end{figure}

\noindent {\bf Case 4.} $x=y_+=t_2=0$ and $zt_1\neq 0$.

As in Case~2, $\Psi$ contains a disk in $T_-$ which splits $[-1,1]\times R_v$ into two 3-balls. We can require that $H$ lies in the 3-ball which contains the upper arc of $A_v$. Then we only need to consider the case when $H\cap\Sigma_0$ consists of one disk in $T_-$ and one disk in $F_-^+$. We can require that $H\cap\Sigma_0$ has the diagram shown in the left picture of Figure~\ref{fig:case4}. Then there are $11$ nonparallel pieces, as shown by the shaded region in the picture. Their union is a 3-ball. We call such a 3-ball having model $Y_2'$.

This finishes the discussion of the cases where there are no bands.

\begin{figure}[h]
\includegraphics{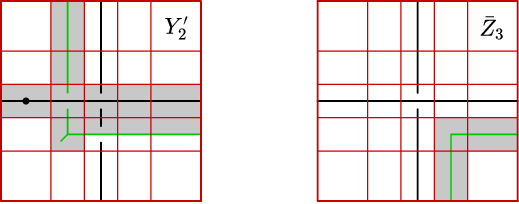}
\caption{The models $Y_2'$ in Case~4 and $\bar{Z}_3$ coming from bands.}\label{fig:case4}
\end{figure}

\noindent {\bf The cases with bands.} Now, $\Psi$ can be obtained from normal disks in Cases~1-4 by adding bands. The bands are vertical disks, which project to straight arcs near the corners of $R_v$. We can isotope $\Psi$ so that all vertical (resp. horizontal) parts of $\Psi$ are still vertical (resp. horizontal), and any band projects to a union of straight arcs in $W_v$. Up to isometries of $R_v$, we can assume that each band corresponds to an arc as shown in the right picture of Figure~\ref{fig:case4}. Then, the arc in $R_v$ meets three rectangles. Let $L$ be the union of the three rectangles. Let $U$ be a rectangle which is adjacent to $L$. If $U$ contains the corner of $R_v$, then the pieces in $[-1,1]\times U$ are all parallel. If $U\cap p(A_v)\neq\emptyset$, then $\Psi\cap[-1,1]\times U$ has not been changed when we add bands. So, we only need to focus on the nonparallel pieces in $[-1,1]\times L$.

\vspace{5pt}

We use $\bar{H}$ to denote a component of $H\cap[-1,1]\times L$. Recall that each band lies in a $\mathrm{II}$-disk. So we only need to consider the cases when $\bar{H}$ meets such $\mathrm{II}$-disks. If $\bar{H}$ lies between two parallel $\mathrm{II}$-disks, then by Definition~\ref{def:piece}, the pieces in $\bar{H}$ are all parallel. Otherwise, $\bar{H}$ can only meet one $\mathrm{II}$-disk. If $\bar{H}$ also meets the block which corresponds to the $U$ containing the corner, then $\bar{H}$ is a union of three nonparallel pieces. Consider a block $[-1,1]\times U$ meeting $\bar{H}$, where $U\cap p(A_v)\neq\emptyset$. The $\mathrm{II}$-disk meets exactly two components of $\Psi\cap[-1,1]\times U$, which must be horizontal disks, and do not meet $A_v$. Then, the pieces in $H\cap[-1,1]\times U$ which are adjacent to $\bar{H}$ must be parallel. So $\bar{H}$ is a component of the union of nonparallel
pieces in $H$. In the right picture of Figure~\ref{fig:case4}, the shaded region gives its projection. Note that $\bar{H}$ is a 3-ball which looks like a $Z_3$-piece. We call such a 3-ball having model $\bar{Z}_3$.

We also have the cases when $\bar{H}$ meets exactly one $\mathrm{II}$-disk and it only meets the $[-1,1]\times U$ with $U\cap p(A_v)\neq\emptyset$. Based on Cases 1-4, we have Cases 5-8.

\vspace{5pt}

\noindent {\bf Case 5.} $\Psi$ is obtained by adding bands to the normal disks in Case~1.

According to the proof of Theorem~\ref{thm:normalForm} and Table~\ref{tab:subclass}, $\Psi$ contains disks in $C_0$ and $x=0$. Then, according to whether $y_+$ or $y_-$ is zero, there are four subcases.

Case 5(1). $y_+\neq 0$ and $y_-\neq 0$.

Only one component of $[-1,1]\times R_v\setminus S_D$ can contain bands. Then the possible $\bar{H}$ lies in the handle $H$ where $H\cap\Sigma_0$ consists of a disk in $C_0$, a disk in $C_+$, and a disk in $C_-$. We can assume that the vertical parts in $H\cap\Sigma_0$ project to three arcs as shown in the left picture of Figure~\ref{fig:case5}. Then $\bar{H}$ is a union of three pieces. Only one of them is nonparallel. So there are $7$ nonparallel pieces in $H$, as shown in the picture. Their union gives a 3-ball. We call such a 3-ball having model $XZ_3$. One can compare it with $X_4$; see the left picture of Figure~\ref{fig:case1}.

Case 5(2). $y_+=0$ and $y_-\neq 0$.

There are exactly two components of $[-1,1]\times R_v\setminus S_D$ that can contain bands. Now $\bar{H}$ lies in the handle $H$ where $H\cap\Sigma_0$ is a union of a disk in $C_0$ and a disk in $C_-$. Up to reflections across the squares in $S_D\cap[-1,1]\times R_v$, we can assume that the two vertical disks in $H\cap\Sigma_0$ project to the arcs shown in the middle picture of Figure~\ref{fig:case5}. Then $\bar{H}$ consists of three pieces where two pieces are nonparallel. Note that the piece whose projection lies in the square satisfies Definition~\ref{def:piece}(1). But it does not satisfy Definition~\ref{def:piece}(2). So the piece in $\bar{H}$ is nonparallel. Then, we have $9$ nonparallel pieces in $H$. Their union gives a 3-ball. We call such a 3-ball having model $XY_2$. One can compare it with $X_3$; see the middle picture of Figure~\ref{fig:case1}.

Case 5(3). $y_+\neq 0$ and $y_-=0$.

By applying an isometry of $[-1,1]\times R_v$ which interchanges the two arcs of $A_v$, we see that this case is essentially the same as Case~5(2). We can obtain a picture similar to the middle one in Figure~\ref{fig:case5}. For the possible $\bar{H}\subset H$, those nonparallel pieces in $H$ give a 3-ball. We also call this 3-ball having model $XY_2$.

Case 5(4). $y_+=y_-=0$.

Each component of $[-1,1]\times R_v\setminus S_D$ can contain bands. The possible $\bar{H}$ lies in the handle $H$ where $H\cap\Sigma_0$ is one disk in $C_0$. Up to reflections across the squares in $S_D$ we can require that the vertical disk in $H\cap\Sigma_0$ projects to the arc shown in the right picture of Figure~\ref{fig:case5}. Then $\bar{H}$ consists of three nonparallel pieces, and $H$ contains exactly $12$ nonparallel pieces, as shown in the picture. The union of these pieces is a 3-ball. We call such a 3-ball having model $X_1$. One can also compare it with $X_2$; see the right picture of Figure~\ref{fig:case1}.

\begin{figure}[h]
\includegraphics{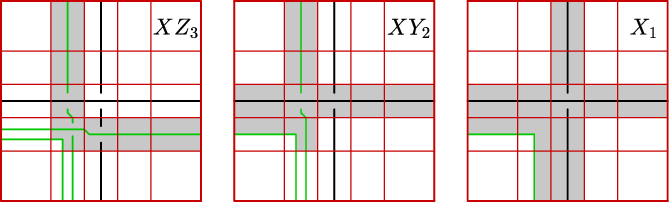}
\caption{The three models $XZ_3$, $XY_2$, $X_1$ appearing in Case~5.}\label{fig:case5}
\end{figure}

\noindent {\bf Case 6.} $\Psi$ is obtained by adding bands to the normal disks in Case~2.

Assume that $\Psi'$ is a union of some normal disks in $F_+$, $F_-$, $C_+$, $C_-$, and $T_-$ as in Case~2, and $\Psi$ can be obtained from $\Psi'$ by adding some bands. By the proof of Theorem~\ref{thm:normalForm}, there are exactly two components of $[-1,1]\times R_v\setminus S_D$ where we can add bands. In such a component we can add bands inductively, and the first band lies in an innermost $\mathrm{II}$-disk. Then the possible $\bar{H}$ meets such a $\mathrm{II}$-disk.

The twisted disks in $\Psi'$ cut $[-1,1]\times R_v$ into some 3-balls. There is exactly one 3-ball containing the upper (resp. lower) arc of $A_v$. Those 3-balls disjoint from $A_v$ lie between parallel disks in $T_-$. Since the first band lies between a disk in $T_-$ and a disk in $F_{\pm}$, $\bar{H}$ lies in one of the two 3-balls which meet $A_v$. We can assume that $\bar{H}$ lies in the 3-ball containing the upper arc of $A_v$. For the other 3-ball, there are parallel results and the same models, as in Case~2.

Now in the 3-ball containing the upper arc of $A_v$, $\bar{H}$ meets a $\mathrm{II}$-disk lying in $\Psi$, so $y_+=0$. Then, by cutting $[-1,1]\times R_v$ along $\Psi'$, we have a 3-ball $H'$ containing the upper arc of $A_v$ so that $H'\cap\Sigma_0$ consists of a disk in $T_-$ and a disk in $F_-$. So, by Case~2(1), the nonparallel pieces in $H'$ give a $Y_2$-piece. Then, when we add the first band, we can get a modified $Y_2$-piece as given in the left picture of Figure~\ref{fig:case678}. Note that $\bar{H}$ consists of three pieces, where two pieces are nonparallel. Two pieces in the original $Y_2$-piece have been modified, where one piece becomes parallel. For any rectangle $U$ with $U\cap p(A_v)\neq\emptyset$, the pieces in $[-1,1]\times U$ do not change when we add bands. So the modified $Y_2$-piece also does not change when we add bands. Then it is a component of the union of nonparallel pieces in $H$. The component is a 3-ball. We call such a 3-ball having model $YZ_1$.

\begin{figure}[h]
\includegraphics{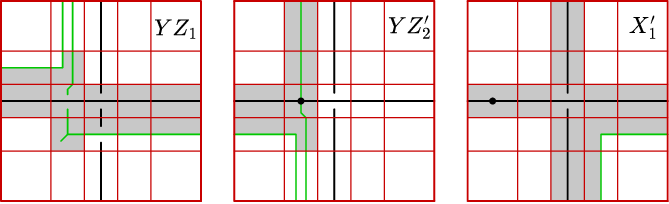}
\caption{The models $YZ_1$, $YZ_2'$, $X_1'$ appearing in Cases~6 and 7.}\label{fig:case678}
\end{figure}

\noindent {\bf Case 7.} $\Psi$ is obtained by adding bands to the normal disks in Case~3.

As above, we assume that $\Psi'$ is a union of some normal disks as in Case~3, and $\Psi$ can be obtained from $\Psi'$ by adding some bands. We assume that $\Psi'$ contains $t_1$ disks in $F_-^+$ and $t_2$ disks in $C_-^+$. Then, according to the proof of Theorem~\ref{thm:normalForm}, one can add bands only if $t_1=0$ or $t_2=0$. So there are two subcases.

Case 7(1). $t_1=0$ and $t_2\neq 0$.

There are exactly two components of $[-1,1]\times R_v\setminus S_D$ that can contain bands. Now $\bar{H}$ lies in the handle $H$ where $H\cap\Sigma_0$ is a union of a disk in $C_0$ and a disk in $C_-^+$. Up to reflections across the squares in $S_D\cap[-1,1]\times R_v$, we can assume that the two vertical disks in $H\cap\Sigma_0$ project to the arcs shown in the middle picture of Figure~\ref{fig:case678}. Then there are exactly $6$ nonparallel pieces in $H$, and their union gives a 3-ball. We call such a 3-ball having model $YZ_2'$. One can compare it with $Y_3'$.

Case 7(2). $t_1\neq 0$ and $t_2=0$.

According to the proof of Theorem~\ref{thm:normalForm} and Table~\ref{tab:subclass}, we have $y_-=0$. There are exactly two components of $[-1,1]\times R_v\setminus S_D$ that can contain bands. The possible $\bar{H}$ lies in the handle $H$ where $H\cap\Sigma_0$ is a disk in $C_0^+$. Up to reflections across the squares in $S_D$, we can require that the vertical disk in $H\cap\Sigma_0$ gives the arc shown in the right picture of Figure~\ref{fig:case678}. Then $H$ has exactly $12$ nonparallel pieces, which give a 3-ball. We call such a 3-ball having model $X_1'$. See also $X_2'$.

\vspace{5pt}

\noindent {\bf Case 8.} $\Psi$ is obtained by adding bands to the normal disks in Case~4.

As above, we assume that $\Psi'$ is a union of some normal disks as in Case~4, and $\Psi$ can be obtained from $\Psi'$ by adding some bands. Then according to the proof of Theorem~\ref{thm:normalForm}, $y_-=0$, and exactly one component of $[-1,1]\times R_v\setminus S_D$ can contain bands. The possible $\bar{H}$ lies in the handle $H$ where $H\cap\Sigma_0$ is a disk in $T_-^{1,0}$. So we meet a situation which has been discussed in Case~6. We can get a picture similar to the left one in Figure~\ref{fig:case678}. The nonparallel pieces in $H$ give a $YZ_1$-piece.

This finishes the discussion of the cases where there exist bands.

\vspace{5pt}

So we have found all possible components of the union of the nonparallel pieces in $H$. The components are 3-balls, which have the following models:
\begin{align*}
X_4,\, &X_3, X_2, X_1, XY_2, XZ_3,\quad Y_2, YZ_3, YZ_1,\quad Z_3, Z_2, Z_1, \bar{Z}_3,\\
&X_3', X_2', X_1',\quad Y_3', Y_3'', Y_2', Y_1', YZ_2',\quad Z_2', O_2', O_2'', \bar{O}_2''.
\end{align*}

\begin{remark}\label{rem:sym}
Note that in the proof of Theorem~\ref{thm:normalForm} and the above discussions, we have applied the reflections across the squares in $S_D\cap[-1,1]\times R_v$. So for each of the pictures in Figures~\ref{fig:case1}-\ref{fig:case678}, by applying reflections across the two arcs in $p(A_v)$, there are similar pictures. Similarly, in Cases~1(3), 2, 5(3), and 6, we have applied isometries of $[-1,1]\times R_v$ which interchange the two arcs of $A_v$. So, in these cases there are also similar pictures corresponding to the isometries. So, for each model, we usually have several pictures which differ by certain symmetries.
\end{remark}

\vspace{3pt}

Below we simplify the models further and modify $(M_i,A_i)$ correspondingly. We locally push $\partial M_i$ ``up'' or ``down'' so that the points in $\partial M_i\cap K$ move along $K$ as in Section~\ref{subsec:Moves}. Then $(M_i,A_i)$ becomes some $(M_i',A_i')$ with $M_i'\subset M_i$ while certain nonparallel pieces in $M_i$ become parallel in $M_i'$, and the $X$-pieces become simpler. Now, we use $B_X$ to denote an $X$-piece in $M_i$. Note that $B_X\cap\Sigma_0$ consists of some horizontal and vertical disks. The modification has three steps.

\vspace{5pt}

\noindent {\bf Step 1.} We simplify $B_X$ when $X$ is $X_1'$, $X_2'$, or $X_3'$.

The bottom horizontal disk in $B_X\cap\Sigma_0$ meets $K$ in a point. Near this point we can isotope $\partial M_i$ so that the disk moves ``up'' and the point moves along $K$. Then some nonparallel pieces in $B_X$ become parallel. Corresponding to $X_1'$, $X_2'$, and $X_3'$, we can then get the new models $X_1'Y$, $X_2'Y$, and $X_3'Y$ shown in the three pictures of Figure~\ref{fig:step1}, respectively. As before, those shaded regions indicate the projections of the nonparallel pieces. The new $X$-pieces are also 3-balls.

\begin{figure}[h]
\includegraphics{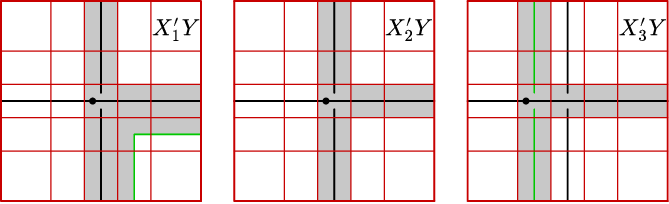}
\caption{The models $X_1'Y$, $X_2'Y$, $X_3'Y$ coming from $X_1'$, $X_2'$, $X_3'$.}\label{fig:step1}
\end{figure}

\noindent {\bf Step 2.} We simplify $B_X$ when $X$ is $Y_1'$, $Y_2'$, or $Y_3''$.

As in Step~1, we can modify $Y_2'$-pieces and $Y_3''$-pieces. Corresponding to $Y_2'$ and $Y_3''$, we can obtain the new models $Y_2'Z$ and $Y_3''Z$ shown in the left two pictures of Figure~\ref{fig:step2}, respectively. Their corresponding new $X$-pieces are also 3-balls. If $X$ is $Y_1'$, then we can modify $B_X$ as follows. Now the vertical disk in $B_X\cap\Sigma_0$ meets $K$ in a point. We first isotope $\partial M_i$ so that the top horizontal disk in $B_X\cap\Sigma_0$ moves ``down'' and contains the point, then we push this disk ``down'' further so that the point moves along $K$. Then, a nonparallel piece in $B_X$ becomes parallel, while $B_X$ splits into a $Z_1$-piece and an $O_2'$-piece; see the right picture of Figure~\ref{fig:step2}.

\begin{figure}[h]
\includegraphics{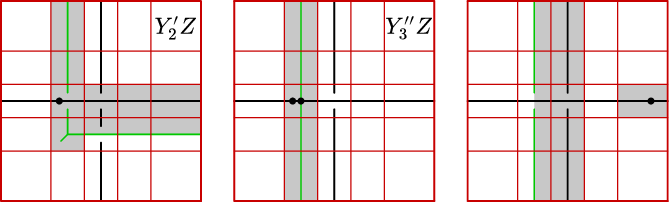}
\caption{The two models $Y_2'Z$, $Y_3''Z$ and the decomposition of $Y_1'$.}\label{fig:step2}
\end{figure}

\noindent {\bf Step 3.} We simplify $B_X$ when $X$ is $Z_2'$ or $O_2'$.

Similar to the case of $Y_1'$ in Step~2, we can modify any $Z_2'$-piece $B_X$ by pushing ``down'' the top horizontal disk in $B_X\cap\Sigma_0$, so that the point $B_X\cap\Sigma_0\cap K$ moves along $K$. Then three nonparallel pieces in $B_X$ become parallel, while $B_X$ becomes an $O_2'$-piece. Note that the ``pushing'' corresponds to a 0-move.

Now consider an $O_2'$-piece $B_X$. The top horizontal disk in $B_X\cap\Sigma_0$ meets $K$ in a point. Then corresponding to a 1-move with the direction ``down'', we can move the point across the adjacent 1-handle into another 0-handle. In the 1-handle, the $X$-piece in $M_i$ that meets the $O_2'$-piece has model $Z_2$. Then in the other 0-handle, the $X$-piece in $M_i$ that meets the $Z_2$-piece must have a certain model. The model cannot be $O_2'$. Otherwise, we will get a trivial connected summand of $K$, as in the discussion in Section~\ref{subsec:Moves}. If the $X$-piece has model $Z_2$, then it meets a $Z_2$-piece in another 1-handle. We can assume that such a sequence of $Z_2$-pieces stops at some $X$-piece in a 0-handle, and the $X$-piece has a model other than $O_2'$ and $Z_2$. Then, by Step~1 and Step~2, the model of this $X$-piece must be one of the following:
\[X_1, X_2, X_3,\quad XY_2, Y_2, Y_3',\quad X_1'Y, X_2'Y, X_3'Y.\]
According to the 0-moves and 1-moves, we can move $B_X\cap\Sigma_0\cap K$ along $K$. Then all the pieces in the $O_2'$-piece and the $Z_2$-pieces become parallel. The model of the final $X$-piece will change. According to the above list, we have three cases.

Case (a). The final $X$-piece has model $X_1$, $X_2$, or $X_3$.

We first assume that it has model $X_1$ as given in the right picture of Figure~\ref{fig:case5}. The point in $\Sigma_0\cap K$ may come from the top side of $R_v$ or the right side of $R_v$. In the former case, the $X_1$-piece will become the $X_1'Y$-piece shown in the left picture of Figure~\ref{fig:step3Ca}. It differs from the one in Step~1 by isometries which interchange the two arcs of $A_v$, as in Remark~\ref{rem:sym}. The latter case cannot happen. Otherwise $\mathcal{I}(\Sigma)$ can be reduced by the moves as in the discussion in Section~\ref{subsec:Moves}, which contradicts the assumption that $\Sigma$ has the simple normal form; see Definition~\ref{def:simNF}.

Then we assume that the $X$-piece has model $X_2$ as given in the right picture of Figure~\ref{fig:case1}. As above, we have essentially two situations where the point in $\Sigma_0\cap K$ comes from the top side or the right side of $R_v$. For the former case, the $X_2$-piece will become the $X_2'Y$-piece shown in the middle picture of Figure~\ref{fig:step3Ca}, which differs from the one in Step~1 by an isometry interchanging the two arcs of $A_v$, as above. For the latter case, the $X_2$-piece will split into a $Z_2$-piece and an $O_2'$-piece. So, we can repeat the previous processes for this new $O_2'$-piece.

Then, we assume that the $X$-piece has model $X_3$ as given in the middle picture of Figure~\ref{fig:case1}. The point in $\Sigma_0\cap K$ may come from the left side or the right side of $R_v$. In the former case, the $X$-piece will split into a $Z_3$-piece and an $O_2'$-piece. For the latter case, we can get a $Z_3$-piece and an $O_2'$-piece which are adjacent. In each case, we can move the point further and repeat the previous processes.

Finally, note that in Case~1(3) we have another picture for $X_3$. So we also have two cases for this picture. In such cases, the $X$-piece will become an $X_3'Y$-piece or a $Y_3'$-piece. As above, the $X_3'Y$-piece differs from the one in Step~1 by an isometry interchanging the two arcs of $A_v$. See the right picture of Figure~\ref{fig:step3Ca}.

Now, up to symmetry, we have listed all the possibilities in Case~(a).

\begin{figure}[h]
\includegraphics{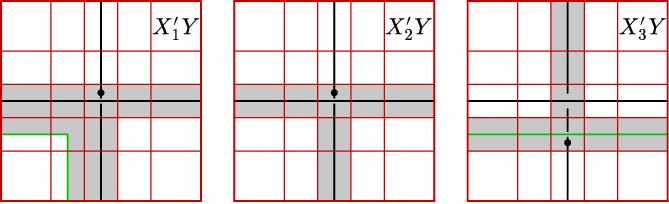}
\caption{The other pictures for the models $X_1'Y$, $X_2'Y$, and $X_3'Y$.}\label{fig:step3Ca}
\end{figure}

Case (b). The final $X$-piece has model $XY_2$, $Y_2$, or $Y_3'$.

First, we assume that the $X$-piece has model $XY_2$. Then, similar to the case of $X_1$ in Case (a), since $\Sigma$ has the simple normal form, this $XY_2$-piece cannot be the one shown in the middle picture of Figure~\ref{fig:case5}. Then, it suffices to consider the one given in Case~5(3). This one will become a $YZ_2'$-piece.

When the $X$-piece has model $Y_2$, the argument is similar. It suffices to consider the $Y_2$-piece which contains the lower arc of $A_v$. Then, the $Y_2$-piece will become a $Y_2'Z$-piece, which differs from the one obtained in Step~2 by an isometry. The case of $Y_3'$ is similar. The $Y_3'$-piece shown in the left picture of Figure~\ref{fig:case3b} will become a $Y_3''Z$-piece, which differs from the one obtained in Step~2 by an isometry.

Now, up to symmetry, we have listed all the possibilities in Case~(b).

Case (c). The final $X$-piece has model $X_1'Y$, $X_2'Y$, or $X_3'Y$.

Note that we need to consider two pictures for each of the three models. One is given in Figure~\ref{fig:step1}. The other is given in Figure~\ref{fig:step3Ca}. For $X_1'Y$, the latter picture is impossible. Otherwise we will obtain a contradiction as in Case~(a). For $X_2'Y$, the latter picture is also impossible. If the point in $\Sigma_0\cap K$ comes from the left side or the right side of $R_v$, then as in Step~5 in Section~\ref{sec:SF}, the diagram $D$ is not minimal. If this point comes from the bottom side of $R_v$, then as in Step~6 in Section~\ref{sec:SF}, the knot $K$ is not connected. Below we consider the remaining four cases.

We first assume that the $X$-piece has model $X_1'Y$ as given in the left picture of Figure~\ref{fig:step1}. The point in $\Sigma_0\cap K$ comes from the top side of $R_v$. Then we can have a new model $X_1''Z$ as given in the left picture of Figure~\ref{fig:step3}. This $X_1''Z$-piece meets $\Sigma$ in a disk, which contains two points in $K$. Let $S_j\subseteq\Sigma$ be the sphere that contains the disk. Then $S_j$ splits $K$ into two arcs $E_1$ and $E_2$. Let $\mu\subset S_j$ be the arc shown in the picture. It consists of one vertical arc and two horizontal arcs. Then we get two knots $\widehat{E}_i=E_i\cup\mu$, $i=1,2$, and $K=\widehat{E}_1\#\widehat{E}_2$. By projecting $\widehat{E}_i$ to $S_0$ we have a diagram $Q_i$ of $\widehat{E}_i$. Now let $c(Q_i)$ denote the number of crossings in $Q_i$. We note that outside $R_v$ there are at least two points in $Q_1\cap Q_2$. So, we have
\[c(K)=c(D)\geq c(Q_1)+c(Q_2)\geq c(\widehat{E}_1)+c(\widehat{E}_2)\geq c(\widehat{E}_1\#\widehat{E}_2).\]
Then, $c(K)=c(\widehat{E}_1)+c(\widehat{E}_2)$, $c(\widehat{E}_i)=c(Q_i)$ for $i=1,2$, and outside $R_v$ there exist exactly two points in $Q_1\cap Q_2$. Moreover, one of $\widehat{E}_1$ and $\widehat{E}_2$ contains a trefoil knot as a connected summand; see the picture. Then, by replacing $K$ with $\widehat{E}_i$, which is nontrivial, we can assume that there are no $X_1''Z$-pieces.

\begin{figure}[h]
\includegraphics{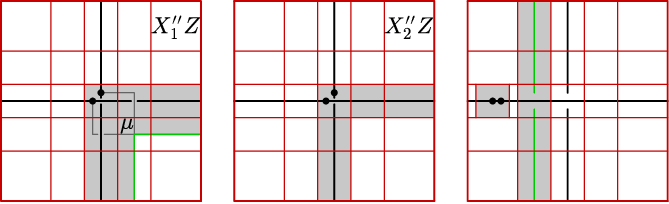}
\caption{The models $X_1''Z$, $X_2''Z$ and the decomposition of $X_3'Y$.}\label{fig:step3}
\end{figure}

When the $X$-piece has model $X_2'Y$ as shown in the middle picture of Figure~\ref{fig:step1}, there are essentially two situations where the point in $\Sigma_0\cap K$ comes from the top side or the right side of $R_v$. In the former case, we can have a new model $X_2''Z$ as given in the middle picture of Figure~\ref{fig:step3}. In the latter case, by reversing the move in Step~1, we can first turn the $X_2'Y$-piece into the original $X_2'$-piece. Then we see that this $X_2'$-piece splits into a $Z_2$-piece and an $O_2''$-piece.

Now we assume that the $X$-piece has model $X_3'Y$. As above, for the $X_3'Y$-piece shown in the right picture of Figure~\ref{fig:step1}, we can obtain a $Z_3$-piece and an $O_2''$-piece which are adjacent. Then, for our purpose, we can shrink this $O_2''$-piece a little bit so that it does not intersect the $Z_3$-piece. See the right picture of Figure~\ref{fig:step3}. Note that any $O_2''$-piece is essentially an $I$-bundle over a disk, which meets $K$ in a fiber. We also regard the shrunken piece as an $O_2''$-piece. For the $X_3'Y$-piece given in the right picture of Figure~\ref{fig:step3Ca}, we will obtain a $Y_3''Z$-piece.

Then, up to symmetry, we have listed all the possibilities in Case~(c).

Finally, the moves of the point in $\Sigma_0\cap K$ cannot affect $X_2''Z$-pieces. Otherwise, we will get a contradiction as in Step~5 or Step~6 in Section~\ref{sec:SF}. Since we can apply only finitely many times of the 0-moves and 1-moves, as in Step~6 in Section~\ref{sec:SF}, by Step~3, we can remove all the possible $O_2'$-pieces.

\begin{definition}\label{def:unexp}
We call a nonprime knot $K\subset S^3$ {\it unexpected} if for any nontrivial decomposition $K=\widehat{E}_1\#\widehat{E}_2$, the inequality $c(K)<c(\widehat{E}_1)+c(\widehat{E}_2)$ holds.
\end{definition}

Clearly we only need to show Theorem~\ref{thm:main} for those possible unexpected knots. We simply use this condition to avoid the $X_1''Z$-pieces, so that the discussions can be simpler. We summarize the results in this section as follows.

\begin{proposition}\label{pro:Xmodels}
Assume that $K$, $D$, $\mathcal{H}_D$, $\Sigma$, and $(M_i,A_i)$ be as before, where $\Sigma$ has the simple normal form, and $K$ is unexpected. Then we can obtain a new pair $(M_i',A_i')$ from $(M_i,A_i)$ by removing an open collar of $(\partial M_i,\partial A_i)$, so that the part of $M_i'$ in each 1-handle and 0-handle of $\mathcal{H}_D$ is a union of some parallel pieces and $X$-pieces, where each $X$-piece has one of the following models:
\begin{align*}
X_4,\, &X_3, X_2, X_1, XY_2, XZ_3,\quad Y_2, YZ_3, YZ_1,\quad Z_3, Z_2, Z_1, \bar{Z}_3,\\
&X_3'Y, X_2'Y, X_1'Y, X_2''Z,\quad Y_3', Y_3''Z, Y_2'Z, YZ_2',\quad O_2'', \bar{O}_2''.
\end{align*}
Moreover, all the possible $X$-pieces that lie in the same 0-handle, with model other than $O_2''$, $\bar{O}_2''$, $\bar{Z}_3$, can be given by using Table~\ref{tab:modelsX} and Table~\ref{tab:effectX}. In Table~\ref{tab:modelsX}, the eight cases correspond to those eight ones before the modification, and each model name represents one $X$-piece. In one row, those $X$-pieces outside the round brackets can appear simultaneously, and the $X$-piece followed by round brackets can be replaced by the ones in the bracket. All those possible variations of the $X$-pieces during the three steps of the modification are listed in Table~\ref{tab:effectX}.
\end{proposition}

\begin{table}[h]
\caption{The $X$-pieces that can appear simultaneously.}\label{tab:modelsX}
\centerline{
\begin{tabular}{l|l||l|l}
\hline
Case~1 & $Z_2$ ($Z_1,Z_3$), $Z_2$ ($Z_1,Z_3$)    & Case~4 & $Y_2'$, $Y_2$ ($YZ_3,Z_1$), $O_2'$\\ \hline
       & $X_4$, $Z_1$, $Z_1$                     & Case~5 & $XZ_3$, $Z_1$, $Z_1$\\ \hline
       & $X_3$, $Z_1$                            &        & $XY_2$, $Z_1$\\ \hline
       & $X_2$                                   &        & $X_1$\\ \hline
Case~2 & $Y_2$ ($YZ_3,Z_1$), $Y_2$ ($YZ_3,Z_1$)  & Case~6 & $YZ_1$, $Y_2$ ($YZ_3,Z_1$) ($YZ_1$)\\ \hline
Case~3 & $X_3'$, $Z_1$, $O_2'$                   & Case~7 & $YZ_2'$, $Y_1'$ ($Z_2',Z_1$)\\ \hline
       & $X_2'$, $O_2'$                          &        & $X_1'$, $O_2'$\\ \hline
       & $Y_3'$ ($Y_3'',O_2'$), $Y_1'$ ($Z_2',Z_1$) & Case~8 & $Y_2'$, $YZ_1$, $O_2'$\\ \hline
\end{tabular}}
\end{table}

\begin{table}[h]
\caption{All possible variations of the $X$-pieces.}\label{tab:effectX}
\centerline{
\begin{tabular}{l||l||l}
\hline
$X_1'\rightarrow X_1'Y$     & $Z_2'\rightarrow O_2'$ & $XY_2\rightarrow YZ_2'$\\ \hline
$X_2'\rightarrow X_2'Y$     & $O_2'\rightarrow$ Parallel & $Y_2\rightarrow Y_2'Z$\\ \hline
$X_3'\rightarrow X_3'Y$     & $Z_2\rightarrow$ Parallel  & $Y_3'\rightarrow Y_3''Z$\\ \hline
$Y_1'\rightarrow Z_1,O_2'$  & $X_1\rightarrow X_1'Y$ & $X_1'Y\rightarrow X_1''Z$\\ \hline
$Y_2'\rightarrow Y_2'Z$     & $X_2\rightarrow X_2'Y$ ($Z_2,O_2'$) & $X_2'Y\rightarrow X_2''Z$ ($Z_2$)\\ \hline
$Y_3''\rightarrow Y_3''Z$   & $X_3\rightarrow X_3'Y$ ($Y_3'$) ($Z_3,O_2'$) & $X_3'Y\rightarrow Y_3''Z$ ($Z_3$)\\ \hline
\end{tabular}}
\end{table}

\begin{proof}
To modify $(M_i,A_i)$, we only need to suitably push $\partial M_i$ inward. So we can require that $M_i'$ lies in the interior of $M_i$, and $\partial M_i'$ is parallel to $\partial M_i$. To find the list of the models and which $X$-pieces may appear simultaneously in a 0-handle of $\mathcal{H}_D$, we only need to check the discussions in Cases~1-8 and Steps~1-3. When $K$ is unexpected, the variation ``$X_1'Y\rightarrow X_1''Z$'' in Table~\ref{tab:effectX} cannot happen.
\end{proof}

\begin{remark}\label{rem:Z3}
We have not shown $O_2''$, $\bar{O}_2''$, $\bar{Z}_3$ in the tables. We will see that not as the other models in the list, they are irrelevant to the estimation of $c(K_i)$.
\end{remark}


\section{Modeled handle structure for manifold pairs}\label{sec:GP}
Let $K$, $D$, $\mathcal{H}_D$, $\Sigma$, $(M_i,A_i)$, $(M_i',A_i')$, $1\leq i\leq n$, be as in Section~\ref{sec:FDH}. Here $\Sigma$ has the simple normal form. Since $M_i'\subset M_i$, the $n$ new pairs are pairwise disjoint. So now a sphere $S_j\subseteq\Sigma$ corresponds to two spheres in some $\partial M_i'$ and $\partial M_k'$. Let $\bar{\Sigma}$ be the union of all $\partial M_i'$. Note that $\partial M_i'$ also lies in $\mathcal{C}$, and $\mathcal{I}(\partial M_i')=\mathcal{I}(\partial M_i)$.

In this section we deal with the parallel pieces in $M_i'$. We construct generalized parallelity 2-handles in $M_i'$. Then, by suitably modifying $(M_i',A_i')$ further, we give a new handle structure for the pair, where handles also have certain models.

\vspace{8pt}

First, we shrink each $O_2''$-piece as in the right picture of Figure~\ref{fig:step3}, and we view that original $O_2''$-piece as the union of the new one and two smaller parallel pieces. Then, in each handle of $\mathcal{H}_D$, the $X$-pieces in $M_i'$ are pairwise disjoint. An $X$-piece meets another $X$-piece only if they meet in a disk. Similar to the case of $(M_i,A_i)$, we have an induced handle structure $\mathcal{H}_{D,i}'$ for $(M_i',A_i')$, and the 2-handles of $\mathcal{H}_{D,i}'$ are essentially those ones of $\mathcal{H}_{D,i}$. Now the union of all the 2-handles of $\mathcal{H}_{D,i}'$ and all the parallel pieces in $M_i'$ naturally gives an $I$-bundle over a compact surface. It meets each $\bar{O}_2''$-piece and new $O_2''$-piece in an annulus, and it meets the $X$-piece in $M_i'$ with any other model in disjoint disks which are $I$-bundles over arcs. Let $\Gamma$ be a component of the $I$-bundle of the surface, and let $\Gamma_0$ be the base of $\Gamma$.

\begin{proposition}\label{pro:bundle}
The surface $\Gamma_0$ is a 2-sphere with holes, $\Gamma$ is a trivial $I$-bundle over $\Gamma_0$, and $\Gamma$ contains at least one 2-handle of $\mathcal{H}_{D,i}'$.
\end{proposition}

\begin{proof}
If $\partial\Gamma_0=\emptyset$, then $\partial\Gamma$ lies in $\bar{\Sigma}_0=\bar{\Sigma}\cup\{\pm1\}\times S_0$. Since $\Gamma\cap K=\emptyset$, we must have $\partial\Gamma\subseteq\{\pm1\}\times S_0$, which is impossible. Hence $\partial\Gamma_0\neq\emptyset$. If $\Gamma_0$ is nonorientable, then there is an orientation-reversing simple closed curve $\gamma$ in $\Gamma_0$. So the $I$-bundle over $\gamma$ gives a properly embedded M\"obius band in $M_i'$. Then since $\partial M_i'$ is a union of spheres, the boundary of the band must bound a disk in $\partial M_i'$, and this gives an embedded projective plane in $S^3$, which is impossible. Hence $\Gamma_0$ is orientable, and $\Gamma$ is a trivial $I$-bundle over $\Gamma_0$. Then $\Gamma_0$ is a 2-sphere with holes.

Now we let $\gamma$ be a component of $\partial\Gamma_0$. The $I$-bundle over $\gamma$ gives an annulus. If it meets some $\bar{O}_2''$-piece or new $O_2''$-piece, then $\Gamma$ intersects $[-1,1]\times U$, where $U$ is some rectangle in $R_v$ which contains a corner; see Section~\ref{sec:FDH}. Then, since all pieces in $[-1,1]\times U$ are parallel and meet some 2-handles of $\mathcal{H}_{D,i}'$, $\Gamma$ contains a 2-handle of $\mathcal{H}_{D,i}'$. If the annulus meets an $X$-piece with any other model, then it also meets an $X$-piece which lies in a 1-handle of $\mathcal{H}_{D,i}'$. Then, the $X$-piece has model $Z_1$, $Z_2$, or $Z_3$, and one can check that $\Gamma$ also contains a 2-handle of $\mathcal{H}_{D,i}'$.
\end{proof}

Below we identify $\Gamma$ with $\Gamma_0\times[0,1]$ and let $\Gamma_-=\Gamma_0\times\{0\}$ and $\Gamma_+=\Gamma_0\times\{1\}$. Then $\Gamma_-$ and $\Gamma_+$ are 2-spheres with holes. We let $\gamma_1,\ldots,\gamma_l$ be the components of $\partial\Gamma_0$, and let $\gamma_{s,-}$ and $\gamma_{s,+}$ be the two components in $\partial\Gamma_-$ and $\partial\Gamma_+$ corresponding to $\gamma_s$, respectively, where $1\leq s\leq l$.

Clearly $\Gamma_0$ is a disk if and only if $l=1$. In this case, $\Gamma$ looks like a 2-handle. In general, we may have $l>1$. Then, we need to extend $\Gamma$ in a certain way so that it looks like a 2-handle, as in \cite{La}. Here the idea is that we can always choose $l-1$ components of $\partial\Gamma_0$, so that for each chosen $\gamma_s$, we can find a suitable 3-ball in $M_i'$ which meets $\Gamma$ in the annulus $\gamma_s\times[0,1]$, and all these 3-balls are pairwise disjoint. Moreover, we can identify the 3-ball with $D^2\times[0,1]$ so that $\partial D^2\times[0,1]$ coincides with $\gamma_s\times[0,1]$, and either $K$ does not meet the $I$-bundle over $D^2$, or it meets the $I$-bundle in exactly one fiber. See Figure~\ref{fig:GPIB} for a sketch.

So we can regard the union of $\Gamma$ and the 3-balls as an $I$-bundle over a disk. We call this union a {\it generalized parallelity 2-handle} in $M_i'$. Below we show how to get these 3-balls. In the process, we need to consider some disk $D^2$ in $\bar{\Sigma}_0$ which meets the 2-handles of $\mathcal{H}_D$ in normal disks or disks having the form $\{\pm1\}\times R_f$. We use $I_2(D^2)$ to denote the number of those disks in $D^2$. Also, note that $D$ is a 4-valent graph in $S_0$, so the number of the regions $R_f$ in $S_0$ is equal to $c(K)+2$.

According to which sphere in $\bar{\Sigma}_0$ contains $\Gamma_{\pm}$, there are four cases.

\begin{figure}[h]
\includegraphics{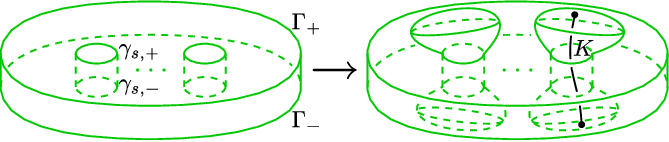}
\caption{Extend $\Gamma$ to a generalized parallelity 2-handle.}\label{fig:GPIB}
\end{figure}

\noindent {\bf Case 1.} Neither $\Gamma_+$ nor $\Gamma_-$ lies in $\bar{\Sigma}$.

We can require that $\Gamma_+\subset\partial B_+$. By Proposition~\ref{pro:bundle}, $\Gamma$ contains one 2-handle of $\mathcal{H}_{D,i}'$. Then, we have $\Gamma_-\subset\partial B_-$, and $\Gamma$ contains exactly one 2-handle of $\mathcal{H}_{D,i}'$. So in this case $\Gamma$ consists of one $[-1,1]\times R_f$ and some parallel pieces in the adjacent 0-handles and 1-handles of $\mathcal{H}_D$. It is indeed an $I$-bundle over a disk.

\vspace{5pt}

\noindent {\bf Case 2.} Exactly one of $\Gamma_+$ and $\Gamma_-$ lies in $\bar{\Sigma}$.

We can assume that $\Gamma_+\subset\bar{\Sigma}$ and $\Gamma_-\subset\partial B_-$. Then, $\Gamma_+$ lies in a sphere $\bar{S}_j\subseteq\bar{\Sigma}$. When $l>1$, we use $D^2_{s,+}$ to denote the disk in $\bar{S}_j$ bounded by $\gamma_{s,+}$ that is disjoint from the interior of $\Gamma_+$, and we let $D^2_{s,-}$ be the disk in $\partial B_-$ bounded by $\gamma_{s,-}$ that is disjoint from the interior of $\Gamma_-$, where $1\leq s\leq l$. Since $\Gamma$ contains a 2-handle of $\mathcal{H}_{D,i}'$ by Proposition~\ref{pro:bundle}, at most one $D^2_{s,-}$ satisfies $I_2(D^2_{s,-})\geq c(K)/2+1$. Then, we pick any $D^2_{s,-}$ that does not satisfy the inequality. Let $A_{s,0}=\gamma_s\times[0,1]$. Then there is a horizontal disk $D^2$ parallel to the disk $A_{s,0}\cup D^2_{s,-}$, with $\partial D^2\subset \bar{S}_j\setminus\Gamma_+$, as shown in Figure~\ref{fig:paraD}. So we can apply the $D^2$-surgery.

\begin{figure}[h]
\includegraphics{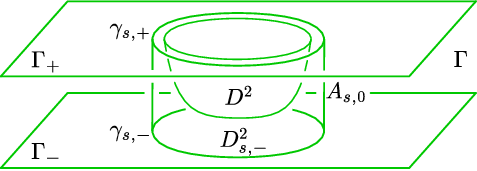}
\caption{A disk $D^2$ parallel to the disk $A_{s,0}\cup D^2_{s,-}$.}\label{fig:paraD}
\end{figure}

The sphere $\bar{S}_j$ will be replaced by some sphere $\bar{S}_j'$. If $D^2_{s,+}\cap K\neq\emptyset$, then it has two points. So the other disk in $\bar{S}_j$ bounded by $\gamma_{s,+}$, which contains $\Gamma_+$, will then be removed. Note that it intersects those 2-handles of $\mathcal{H}_D$ in at least $I_2(D^2_{s,-})+1$ normal disks, because $\Gamma$ contains a 2-handle of $\mathcal{H}_{D,i}'$ and $D^2_{s,-}$ does not satisfy the inequality. So $\mathcal{I}(\partial M_i')$ will be reduced. Then $\mathcal{I}(\Sigma)$ can also be reduced. Otherwise $D^2_{s,+}\cap K=\emptyset$. Now the disk in $D^2_{s,+}$ bounded by $\partial D^2$ will be removed. Note that the 2-sphere $D^2_{s,+}\cup A_{s,0}\cup D^2_{s,-}$ bounds a 3-ball that does not meet the interior of $\Gamma$, and this 3-ball does not meet $K$. The annulus $A_{s,0}$ must meet some $X$-piece in a 1-handle of $\mathcal{H}_D$, whose model can only be $Z_3$. Then this $Z_3$-piece meets $D^2_{s,+}$ in two disks and meets $D^2_{s,-}$ in one disk. So, we must have $I_2(D^2_{s,+})>I_2(D^2_{s,-})$, and $\mathcal{I}(\partial M_i')$ will be reduced. Then $\mathcal{I}(\Sigma)$ can be reduced as above.

So we can require that $\mathcal{I}(\Sigma)$ cannot be reduced and $l=1$. By the argument we also see that $\Gamma_{\pm}$ intersects the 2-handles of $\mathcal{H}_D$ in at most $c(K)/2+1$ disks.

\vspace{5pt}

\noindent {\bf Case 3.} Both $\Gamma_+$ and $\Gamma_-$ lie in the same sphere $\bar{S}_j\subseteq\bar{\Sigma}$.

Let $D^2_{s,+}$ be the disk in $\bar{S}_j$ bounded by $\gamma_{s,+}$ as in Case~2 where $1\leq s\leq l$. Then we can assume that $\Gamma_-$ lies in $D^2_{1,+}$. Similarly, let $D^2_{s,-}$ be the disk in $\bar{S}_j$ bounded by $\gamma_{s,-}$, and let $A_{s,0}$ be the $I$-bundle over $\gamma_s$. Note that $\Gamma$ lies in a 3-ball bounded by $\bar{S}_j$, and the annuli $A_{s,0}$ are disjoint from each other. So we see that $\gamma_{1,-}$ is the outermost circle in $D^2_{1,+}$, and $D^2_{s,-}$ lies in $D^2_{1,+}$ for $2\leq s\leq l$. So, if $l>1$, then we have pairwise disjoint spheres $D^2_{s,+}\cup A_{s,0}\cup D^2_{s,-}$ for $2\leq s\leq l$. Let $\Theta_s$ denote the sphere $D^2_{s,+}\cup A_{s,0}\cup D^2_{s,-}$. For a given $\Theta_s$ we have three cases.

Case 3(1). Neither $D^2_{s,+}$ nor $D^2_{s,-}$ meets $K$.

Then, $\bar{S}_j\setminus(D^2_{s,+}\cup D^2_{s,-})$ intersects $K$, and $\Theta_s$ bounds a 3-ball disjoint from $K$. Because $\Gamma\setminus A_{s,0}$, $\bar{S}_j\setminus(D^2_{s,+}\cup D^2_{s,-})$, and $K$ lie on the same side of $\Theta_s$, this 3-ball only intersects $\Gamma$ in the annulus $A_{s,0}$. Clearly, this 3-ball lies in $M_i'$. Then, we can add it to $\Gamma$, as shown in Figure~\ref{fig:GPIB}. We can also require that $I_2(D^2_{s,+})=I_2(D^2_{s,-})$. Otherwise, $\mathcal{I}(\partial M_i')$ and $\mathcal{I}(\Sigma)$ can be reduced, as in Case~2.

Case 3(2). Exactly one of $D^2_{s,+}$ and $D^2_{s,-}$ meets $K$.

We can assume that $D^2_{s,+}\cap K\neq\emptyset$ and $D^2_{s,-}\cap K=\emptyset$. Then $D^2_{s,+}\cap K$ contains two points. Similar to Case~2, we can find a disk $D^2$ parallel to $A_{s,0}\cup D^2_{s,-}$ where $\partial D^2\subset D^2_{s,+}$. See Figure~\ref{fig:paraD} for a sketch, where both $\Gamma_+$ and $\Gamma_-$ lie in $\bar{S}_j$ now and $D^2_{s,-}$ can be complicated. Then, since we can choose $D^2$ very close to $A_{s,0}\cup D^2_{s,-}$, we can also define $I_2(D^2)$, and we have $I_2(D^2)=I_2(D^2_{s,-})$.

Now, by applying the $D^2$-surgery, the disk in $\bar{S}_j$ bounded by $\gamma_{s,+}$ that contains $\Gamma_+$, $\Gamma_-$, and $D^2_{s,-}$ will then be removed. Since $\Gamma$ contains one 2-handle of $\mathcal{H}_{D,i}'$ by Proposition~\ref{pro:bundle}, this disk contains more normal disks in the 2-handles of $\mathcal{H}_D$ than $D^2_{s,-}$. So $\mathcal{I}(\partial M_i')$ will be reduced. Then $\mathcal{I}(\Sigma)$ can be reduced, as before.

So we can require that $\mathcal{I}(\Sigma)$ cannot be reduced and this case does not happen.

\begin{figure}[h]
\includegraphics{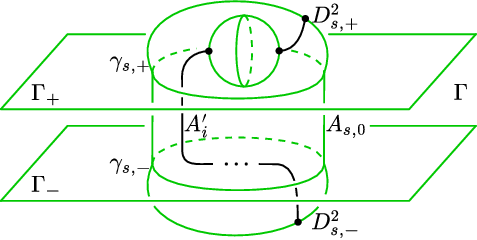}
\caption{The manifold pair with boundary lying in $\Theta_s\cup\bar{\Sigma}$.}\label{fig:SPinM}
\end{figure}

Case 3(3). Both $D^2_{s,+}$ and $D^2_{s,-}$ meet $K$.

Since $\bar{S}_j\cap K$ contains exactly two points, each of $D^2_{s,+}\cap K$ and $D^2_{s,-}\cap K$ only contains one point, and $\bar{S}_j\setminus(D^2_{s,+}\cup D^2_{s,-})$ does not meet $K$. As in Case~3(1), the sphere $\Theta_s$ bounds a 3-ball in $S^3$ that only intersects $\Gamma$ in $A_{s,0}$. But now the 3-ball meets $K$, and it may contain some spheres in $\bar{\Sigma}$. See Figure~\ref{fig:SPinM} for a sketch of this situation. The intersection of the 3-ball with $M_i'$ gives a manifold pair.

We can assume that $I_2(D^2_{s,+})\geq I_2(D^2_{s,-})$. Then, as in Case~3(2), we can pick a disk $D^2$ very close to $A_{s,0}\cup D^2_{s,-}$, with $\partial D^2\subset D^2_{s,+}$. Since $D^2$ meets $K$ in exactly one point, we can apply the $D^2$-surgery. Similar to Case~3(2), we can assume that $\mathcal{I}(\Sigma)$ cannot be reduced. This implies that the disk bounded by $\partial D^2$ which lies in $D^2_{s,+}$ will be removed, and we have $I_2(D^2_{s,+})=I_2(D^2_{s,-})$. Hence, the manifold pair given above is divided into a simple pair and a trivial pair by $D^2$. The trivial pair gives a 3-ball in $M_i'$ which only intersects $\Gamma$ in the annulus $A_{s,0}$. Then we can add this 3-ball to $\Gamma$, as shown in Figure~\ref{fig:GPIB}.

In summary, we can require that $\mathcal{I}(\Sigma)$ cannot be reduced. Then Case~3(2) does not happen, and in Case~3(1) and Case~3(3) we can add 3-balls in $M_i'$ to $\Gamma$. There are $l-1$ such 3-balls, which are pairwise disjoint, by the constructions. The union of $\Gamma$ and these 3-balls can be viewed as an $I$-bundle over a disk.

\begin{remark}\label{rem:addtri}
In Case~3(3), if $\Theta_s$ already bounds a trivial pair, then one can add it to $\Gamma$ directly, and has no need to cut it by $D^2$. Note that the $D^2$-surgery does not affect the correspondence between pairs and knots, by Remark~\ref{rem:corr}. So removing a simple pair as in Case~3(3) from $(M_i', A_i')$ does not affect the type of $\widehat{A}_i'$. Also, the $I$-bundle over a disk obtained in Case~3 can meet $K$ in at most one arc.
\end{remark}

\noindent {\bf Case 4.} $\Gamma_+$ and $\Gamma_-$ lie in different spheres in $\bar{\Sigma}$.

We assume that $\Gamma_+\subset\bar{S}_j$ and $\Gamma_-\subset\bar{S}_k$, where $j\neq k$. For $1\leq s\leq l$, let $D^2_{s,+}$ be the disk in $\bar{S}_j$ bounded by $\gamma_{s,+}$, as in Case~2. Similarly, let $D^2_{s,-}$ be the disk in $\bar{S}_k$ bounded by $\gamma_{s,-}$, and let $A_{s,0}$ denote the $I$-bundle over $\gamma_s$. Since $\bar{S}_k\cap K$ contains exactly two points, there is at most one $D^2_{s,-}$ which can meet $K$ twice. So, we can require that $D^2_{s,-}$ meets $K$ at most once for $2\leq s\leq l$ if $l>1$. Then, as in Case~3, we have pairwise disjoint spheres $\Theta_s$ for $2\leq s\leq l$, and according to whether $D^2_{s,+}$ or $D^2_{s,-}$ meets $K$, we have three cases for a given $\Theta_s$.

Case 4(1). Neither $D^2_{s,+}$ nor $D^2_{s,-}$ meets $K$.

Then, as in Case~3, $\Theta_s$ bounds a 3-ball disjoint from $K$. This 3-ball only meets $\Gamma$ in the annulus $A_{s,0}$, and it lies in $M_i'$. So, we can add this 3-ball to $\Gamma$, as shown in Figure~\ref{fig:GPIB}. Moreover, we can require that $I_2(D^2_{s,+})=I_2(D^2_{s,-})$. Otherwise $\mathcal{I}(\partial M_i')$ and $\mathcal{I}(\Sigma)$ can be reduced as in Case~2.

Case 4(2). Exactly one of $D^2_{s,+}$ and $D^2_{s,-}$ meets $K$.

Then, since $D^2_{s,-}$ meets $K$ at most once, we must have $D^2_{s,+}\cap K\neq\emptyset$. Then we have $D^2_{s,-}\cap K=\emptyset$, and $D^2_{s,+}\cap K$ contains two points. Now consider the 2-sphere obtained from $\bar{S}_j$ by replacing $D^2_{s,+}$ with $A_{s,0}\cup D^2_{s,-}$. The 2-sphere does not meet $K$, but both sides of it meet $K$. Hence we get a contradiction. See Figure~\ref{fig:paraD} for a sketch of this situation, where $\Gamma_+\subset\bar{S}_j$ and $\Gamma_-\subset\bar{S}_k$ now.

Case 4(3). Both $D^2_{s,+}$ and $D^2_{s,-}$ meet $K$.

Then, because $D^2_{s,-}$ intersects $K$ at most once, each of $D^2_{s,-}\cap K$ and $D^2_{s,+}\cap K$ contains exactly one point. So, as in Case~3, $\Theta_s$ bounds a 3-ball which only meets $\Gamma$ in $A_{s,0}$, and the intersection of the 3-ball with $M_i'$ provides a manifold pair. See Figure~\ref{fig:SPinM} for a sketch, where $\Gamma_+\subset\bar{S}_j$ and $\Gamma_-\subset\bar{S}_k$ now.

As in Case~3, we can assume that $I_2(D^2_{s,+})\geq I_2(D^2_{s,-})$ and pick the disk $D^2$. If the pair is simple, then by applying the $D^2$-surgery, we have $I_2(D^2_{s,+})=I_2(D^2_{s,-})$. Otherwise $\mathcal{I}(\Sigma)$ can be reduced. In the simple pair, $D^2$ then cuts off a trivial pair that meets $D^2_{s,-}$, and we can add it to $\Gamma$ as before. When the pair is nonsimple, a little more work is needed, since this $D^2$-surgery may not reduce $\mathcal{I}(\partial M_i')$.

Let $\bar{D}^2_{s,+}$ be the disk in $\bar{S}_j$ bounded by $\gamma_{s,+}$ that contains $\Gamma_+$. This $D^2$-surgery does not reduce $I_2(\partial M_i')$ if and only if $I_2(\bar{D}^2_{s,+})\leq I_2(D^2_{s,-})$. Similarly, let $\bar{D}^2_{s,-}$ be the disk in $\bar{S}_k$ bounded by $\gamma_{s,-}$ that contains $\Gamma_-$. By arguments as above, we can also assume that $I_2(\bar{D}^2_{s,-})\leq I_2(D^2_{s,+})$. Then we also need to consider $D^2_{1,+}$, $D^2_{1,-}$, and the corresponding sphere $\Theta_1$. There is some $r\neq s$ such that each of $D^2_{r,+}$ and $D^2_{r,-}$ also meets $K$ in one point, and there are $\bar{D}^2_{r,+}$ and $\bar{D}^2_{r,-}$ as well. Then, since $\Gamma$ contains a 2-handle of $\mathcal{H}_{D,i}'$ by Proposition~\ref{pro:bundle}, we have
\[I_2(D^2_{r,+})<I_2(\bar{D}^2_{s,+})\leq I_2(D^2_{s,-})<I_2(\bar{D}^2_{r,-}),\]
and similarly, $I_2(D^2_{r,-})<I_2(\bar{D}^2_{r,+})$. Hence, we can get a simple pair corresponding to $\Theta_r$. Otherwise $\mathcal{I}(\partial M_i')$ and $\mathcal{I}(\Sigma)$ can be reduced as before. So now we can find a required 3-ball for each of the spheres $\Theta_t$ with $t\neq s$.

In summary, Case~4(2) cannot happen, and we can require that $\mathcal{I}(\Sigma)$ cannot be reduced. Then in Case~4(1) and Case~4(3) we can add 3-balls in $M_i'$ to $\Gamma$. We can find $l-1$ such 3-balls, which are pairwise disjoint. The union of $\Gamma$ and the 3-balls can then be viewed as an $I$-bundle over a disk.

\begin{remark}\label{rem:notwoarc}
In Case~4(3), if $\Theta_s$ or $\Theta_r$ already bounds a trivial pair, then one can add it to $\Gamma$ directly as in Remark~\ref{rem:addtri}. In general one needs to remove some simple pair from $(M_i', A_i')$. This does not affect the type of $\widehat{A}_i'$. Also, this $I$-bundle over a disk given in Case~4 meets $K$ in at most one arc. Otherwise, by choosing two arcs in $\bar{S}_j$ and $\bar{S}_k$ connecting the points in $K$, we will get a trivial connected summand of $K$; see Figure~\ref{fig:TMpoints} for a sketch. This is a contradiction.
\end{remark}

\begin{figure}[h]
\includegraphics{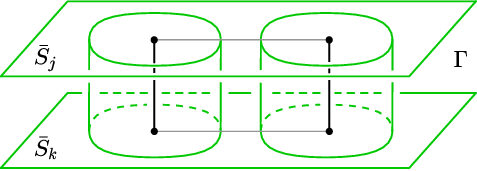}
\caption{Obtain a trivial knot from two arcs in $K$.}\label{fig:TMpoints}
\end{figure}

So by assuming that $\mathcal{I}(\Sigma)$ cannot be reduced, we can extend $\Gamma$ to a generalized parallelity 2-handle $\tilde{\Gamma}$ in $M_i'$, as in Cases~1-4. We see that $\tilde{\Gamma}$ can be identified with an $I$-bundle over a disk $D^2$, where either $K\cap\tilde{\Gamma}=\emptyset$ or $K\cap\tilde{\Gamma}$ is a fiber. Note that the $I$-bundle over $\partial D^2$ cannot meet $\bar{O}_2''$-pieces or those new $O_2''$-pieces. Otherwise, we must have $K\cap\tilde{\Gamma}\neq\emptyset$, and we can obtain a trivial connected summand of $K$ as shown in Figure~\ref{fig:TMpoints}, which is a contradiction.

\vspace{8pt}

Below we modify $(M_i',A_i')$ further so that it becomes a new pair $(M_i'',A_i'')$, and we construct a modeled handle structure for $(M_i'',A_i'')$. There are two steps.

\vspace{5pt}

\noindent {\bf Step 1.} We modify $(M_i',A_i')$ and construct the 2-handles and 3-handles.

We assume that $\mathcal{I}(\Sigma)$ cannot be reduced. Recall that the union of all 2-handles of $\mathcal{H}_{D,i}'$ and all parallel pieces in $M_i'$ is an $I$-bundle over a compact surface. It has finitely many components, and one can extend each component $\Gamma$ to a generalized parallelity 2-handle $\tilde{\Gamma}$ in $M_i'$. In the process, we may need to remove a simple pair from $(M_i',A_i')$. We let $\bar{\Gamma}$ denote the union of $\tilde{\Gamma}$ and the possible simple pair. Then for any components $\Gamma\neq\Gamma'$ satisfying $\bar{\Gamma}\cap\bar{\Gamma}'\neq\emptyset$, according to whether $\bar{\Gamma}\cap\Gamma'=\emptyset$, we see that either $\bar{\Gamma}\subset\bar{\Gamma}'\setminus\Gamma'$ or $\Gamma'\subset\bar{\Gamma}\setminus\Gamma$. Then, among the components, we can inductively pick $\Gamma_1,\ldots,\Gamma_t$, so that $\bar{\Gamma}_r\cap\bar{\Gamma}_s=\emptyset$ for $1\leq r<s\leq t$, and the union of $\bar{\Gamma}_1,\ldots,\bar{\Gamma}_t$ contains all the components of the $I$-bundle. Hence, this union contains all 2-handles of $\mathcal{H}_{D,i}'$ and all parallel pieces in $M_i'$. Then by removing the possible simple pairs that correspond to $\bar{\Gamma}_1,\ldots,\bar{\Gamma}_t$ from $(M_i',A_i')$, we get $M_i''$.

By the constructions, $M_i''$ contains the disjoint generalized parallelity 2-handles $\tilde{\Gamma}_1,\ldots,\tilde{\Gamma}_t$, and $M_i''\subseteq M_i'$. The closure of $M_i''\setminus(\tilde{\Gamma}_1\cup\cdots\cup\tilde{\Gamma}_t)$ consists of $X$-pieces and possibly $B_{\pm}$. We note that there are no $\bar{O}_2''$-pieces or the new $O_2''$-pieces, since such a piece must lie in $\bar{\Gamma}_1\cup\cdots\cup\bar{\Gamma}_t$. Otherwise we can obtain a trivial connected summand of $K$ as in Figure~\ref{fig:TMpoints}. If there exist $Y_3''Z$-pieces, then the pieces meet $K$ in some arcs, and we isotope the arcs into $\tilde{\Gamma}_1\cup\cdots\cup\tilde{\Gamma}_t$, keeping $M_i''$ invariant. As before, each $\tilde{\Gamma}_s$ still meets $K$ in at most one arc; see Remark~\ref{rem:notwoarc}. The $Y_3''Z$-pieces now become $Z_3$-pieces. Then, after applying the isotopy, we let $A_i''=K\cap M_i''$. So we get $(M_i'',A_i'')$, where $\widehat{A}_i''$ gives the same knot as $\widehat{A}_i'$ and $\widehat{A}_i$. Then for $1\leq s\leq t$, we let $(\tilde{\Gamma}_s,\tilde{\Gamma}_s\cap A_i'')$ be a 2-handle. The possible $B_+$ and $B_-$ are 3-handles.

\vspace{5pt}

\noindent {\bf Step 2.} We construct the 1-handles and 0-handles in $(M_i'',A_i'')$.

We remove the above 2-handles and 3-handles from $(M_i'',A_i'')$, and then denote the closure of the resulting pair by $(H_i,\lambda_i)$. Then by Step~1, $(H_i,\lambda_i)$ is a union of some $X$-pieces, where each $X$-piece has one of the following models:
\begin{align*}
X_4, X_3, X_2,\, &X_1, XY_2, XZ_3,\quad Y_2, YZ_3, YZ_1,\quad Z_3, Z_2, Z_1, \bar{Z}_3,\\
&X_3'Y, X_2'Y, X_1'Y, X_2''Z,\quad Y_3', Y_2'Z, YZ_2'.
\end{align*}
Since $M_i''$ is connected, $H_i$ is also connected. Then $H_i$ is a handlebody in $S^3$, and the above 2-handles meet $H_i$ in disjoint annuli. Then, since $\partial M_i''$ bounds pairwise disjoint 3-balls in $S^3$, we see that $\partial H_i$ is a Heegaard surface for $S^3$. So, we call $H_i$ an {\it unknotted} handlebody in $S^3$. Below we define the handles in it.

First, we consider the union of all the $X$-pieces in $H_i$ with models $Z_1$ and $YZ_1$. Each component of this union is either a 3-ball or a solid torus. In the latter case, the solid torus is $H_i$, and then $K$ is the unknot, which is a contradiction. So, each component is a 3-ball. We view it as an extension of a $Z_1$-piece or $YZ_1$-piece, and we call it a {\it generalized 1-handle with model $Z_1$}.

Then we consider the union of all the $X$-pieces in $H_i$ with model $Z_2$. As above, each component of this union is either a 3-ball or a solid torus, and the latter case cannot happen. Otherwise, $H_i$ is an unknotted solid torus, and $K$ is trivial. Then each component is a 3-ball. We view this 3-ball as an extension of a $Z_2$-piece, and we call it a {\it generalized 1-handle with model $Z_2$}.

Finally we consider the union of all the $X$-pieces in $H_i$ with models $Z_3$ and $\bar{Z}_3$. As above, each component of this union is either a 3-ball or a solid torus. Now, in the latter case, $H_i$ is an unknotted solid torus which does not meet $K$, and in $M_i''$ there are three 2-handles which meet $H_i$ in disjoint annuli. Then each of the three 2-handles must meet $K$. So $(M_i'',A_i'')$ is a simple pair, and we get a contradiction. So, as above, each component is a 3-ball. We view it as an extension of a $Z_3$-piece or $\bar{Z}_3$-piece, and we call it a {\it generalized 1-handle with model $Z_3$}.

Now, we define the generalized 1-handles with models $Z_1$, $Z_2$, and $Z_3$ (together with their intersections with $K$) to be the 1-handles in $(M_i'',A_i'')$. Then, we define the $X$-piece in $H_i$ (together with its intersection with $K$) that does not lie in such a 1-handle to be a 0-handle in $(M_i'',A_i'')$. It also has a certain model.

\begin{definition}\label{def:MHS}
All 3-handles, 2-handles, 1-handles, and 0-handles constructed in Step~1 and Step~2 together provide a handle structure for $(M_i'',A_i'')$. We call it the {\it modeled handle structure} for $(M_i'',A_i'')$ and denote it by $\mathcal{H}_{X,i}$.
\end{definition}

Note that each 0-handle of $\mathcal{H}_{X,i}$ lies in some 0-handle of $\mathcal{H}_D$, and a 1-handle of $\mathcal{H}_{X,i}$ may meet a 0-handle of $\mathcal{H}_D$ in some $X$-pieces. Clearly, for a fixed $0\leq j\leq 3$, two $j$-handles of $\mathcal{H}_{X,i}$ are disjoint. Each 1-handle of $\mathcal{H}_{X,i}$ meets 0-handles of $\mathcal{H}_{X,i}$ in two disks, and each 2-handle of $\mathcal{H}_{X,i}$ meets $H_i$ in an annulus. Now two handles of $\mathcal{H}_{X,i}$ can intersect in more than one disks. In the next section, we will use $\mathcal{H}_{X,i}$ to estimate $c(K_i)$. We summarize the results in this section as follows.

\begin{theorem}\label{thm:decompose}
Let $K\subset S^3$ be an unexpected knot. Let $D$ be a diagram of $K$ that has $c(K)$ crossings. Let $\Sigma$ be the union of spheres in the maximal sphere system of $K$ so that $\Sigma$ has the simple normal form and the minimal complexity, with respect to $\mathcal{H}_D$. Let $(M_i,A_i)$, $1\leq i\leq n$, be the manifold pairs obtained by cutting $S^3$ along $\Sigma$. Then, for each $(M_i,A_i)$, there exists a new pair $(M_i'',A_i'')$ so that

(1) $M_i''$ lies in the interior of $M_i$, and $\widehat{A}_i''$ gives the same knot as $\widehat{A}_i$;

(2) $(M_i'',A_i'')$ admits a modeled handle structure $\mathcal{H}_{X,i}$, where each 3-handle can only be $B_+$ or $B_-$, each 2-handle is a generalized parallelity 2-handle which meets $A_i''$ in at most one arc, each 1-handle is a generalized 1-handle with model $Z_1$, $Z_2$, or $Z_3$, and each 0-handle is an $X$-piece that has one of the following models:
\begin{align*}
&X_4, X_3, X_2, X_1, XY_2, XZ_3,\quad Y_2, YZ_3,\\
&X_3'Y, X_2'Y, X_1'Y, X_2''Z,\quad Y_3', Y_2'Z, YZ_2';
\end{align*}

(3) the union of 1-handles and 0-handles of $\mathcal{H}_{X,i}$ is an unknotted handlebody.\\
Moreover, in a 0-handle of $\mathcal{H}_D$, all those possible $X$-pieces coming from 1-handles and 0-handles of all the $\mathcal{H}_{X,i}$ can be obtained by using Table~\ref{tab:models} and Table~\ref{tab:effect}, where the eight cases in Table~\ref{tab:models} correspond to the ones in Table~\ref{tab:modelsX}, all possible variations of the $X$-pieces are listed in Table~\ref{tab:effect}, as in the case of Table~\ref{tab:effectX}, and all the possible $\bar{Z}_3$-pieces are omitted, as explained in Remark~\ref{rem:Z3}.
\end{theorem}

\begin{table}[h]
\caption{The possible $X$-pieces in a 0-handle of $\mathcal{H}_D$.}\label{tab:models}
\centerline{
\begin{tabular}{l|l||l|l}
\hline
Case~1 & $Z_2$ ($Z_1,Z_3$), $Z_2$ ($Z_1,Z_3$)    & Case~4 & $Y_2'Z$, $Y_2$ ($YZ_3,Z_1$)\\ \hline
       & $X_4$, $Z_1$, $Z_1$                     & Case~5 & $XZ_3$, $Z_1$, $Z_1$\\ \hline
       & $X_3$, $Z_1$                            &        & $XY_2$, $Z_1$\\ \hline
       & $X_2$                                   &        & $X_1$\\ \hline
Case~2 & $Y_2$ ($YZ_3,Z_1$), $Y_2$ ($YZ_3,Z_1$)  & Case~6 & $YZ_1$, $Y_2$ ($YZ_3,Z_1$) ($YZ_1$)\\ \hline
Case~3 & $X_3'Y$, $Z_1$                          & Case~7 & $YZ_2'$, $Z_1$\\ \hline
       & $X_2'Y$                                 &        & $X_1'Y$\\ \hline
       & $Y_3'$ ($Z_3$), $Z_1$                   & Case~8 & $Y_2'Z$, $YZ_1$\\ \hline
\end{tabular}}
\end{table}

\begin{table}[h]
\caption{Possible variations of the above $X$-pieces.}\label{tab:effect}
\centerline{
\begin{tabular}{l||l||l}
\hline
$X_1\rightarrow X_1'Y$                  & $Z_2\rightarrow$ Null   & $XY_2\rightarrow YZ_2'$\\ \hline
$X_2\rightarrow X_2'Y$ ($Z_2$)          & $Y_2\rightarrow Y_2'Z$  & $X_2'Y\rightarrow X_2''Z$ ($Z_2$)\\ \hline
$X_3\rightarrow X_3'Y$ ($Y_3'$) ($Z_3$) & $Y_3'\rightarrow Z_3$   & $X_3'Y\rightarrow Z_3$\\ \hline
\end{tabular}}
\end{table}

\begin{proof}
For each $(M_i,A_i)$, we can get $(M_i',A_i')$ by Proposition~\ref{pro:Xmodels}. Since now $\mathcal{I}(\Sigma)$ cannot be reduced, we can then get $(M_i'',A_i'')$ and $\mathcal{H}_{X,i}$ by the above construction of generalized parallelity 2-handles and Steps~1-2. Then we only need to check the tables. By Steps~1-3 in Section~\ref{sec:FDH}, $X_1'$, $X_2'$, $X_3'$, $Y_2'$, and $Y_3''$ in Table~\ref{tab:modelsX} are replaced by $X_1'Y$, $X_2'Y$, $X_3'Y$, $Y_2'Z$, and $Y_3''Z$, respectively, while $Y_1'$ is replaced by $Z_1$, and $Z_2'$ and $O_2'$ are removed. Also, in Table~\ref{tab:effectX}, ``$X_1'Y\rightarrow X_1''Z$'' cannot happen, since $K$ is unexpected. Then by above Step~1, there are no $\bar{O}_2''$-pieces and $O_2''$-pieces in the 0-handles or 1-handles of $\mathcal{H}_{X,i}$, and all the possible $Y_3''Z$-pieces become $Z_3$-pieces. So, by Table~\ref{tab:modelsX} and Table~\ref{tab:effectX}, we have Table~\ref{tab:models} and Table~\ref{tab:effect}.
\end{proof}

\begin{remark}\label{rem:notallX}
The condition ``unexpected'' is only used to avoid $X_1''Z$-pieces. Since $(M_i'',A_i'')$ only contains part of the $X$-pieces in $(M_i',A_i')$ whose models are not $O_2''$ and $\bar{O}_2''$, actually in Table~\ref{tab:models} some $X$-pieces may not appear.
\end{remark}

\begin{remark}
The union of all $\partial M_i''$ consists of $2n$ spheres. So, we can have $2n-1$ manifold pairs by cutting $(S^3,K)$ along these spheres. Since $(M_i'',A_i'')$ is obtained from $(M_i',A_i')$ by removing some possible simple pairs and applying an isotopy, we see that among those $2n-1$ pairs, the pairs other than $(M_i'',A_i'')$, $1\leq i\leq n$, must be simple. So $(S^3,K)$ is decomposed into $n-1$ simple pairs and all the $(M_i'',A_i'')$. Also, the unknotted $H_i$, $1\leq i\leq n$, are pairwise unlinked in $S^3$.
\end{remark}


\section{Estimation of the number of crossings}\label{sec:EN}
In this section we finish the proof of Theorem~\ref{thm:main}. We can require that $K$ is an unexpected knot, and $D$ is a diagram of $K$ that has $c(K)$ crossings. By Section~\ref{sec:SF}, we can obtain a maximal sphere system of $K$ so that the corresponding $\Sigma$ has the simple normal form and the minimal $\mathcal{I}(\Sigma)$, with respect to $\mathcal{H}_D$. So, we have pairs $(M_i,A_i)$, $1\leq i\leq n$, corresponding to those connected summands $K_i$, $1\leq i\leq n$, of $K$. Then, by Section~\ref{sec:FDH} and Section~\ref{sec:GP}, we have $(M_i'',A_i'')$ and $\mathcal{H}_{X,i}$, $1\leq i\leq n$, and we will use their properties given in Theorem~\ref{thm:decompose} to estimate $c(K_i)$.

Assume that $\partial M_i''$ consists of the spheres $S_{i,j}$, where $1\leq j\leq m_i$. So $S_{i,j}$ meets $K$ in exactly two points. We pick an arc $\mu_{i,j}\subset S_{i,j}$ connecting the two points and let $\mu_i$ be the union of $\mu_{i,j}$, $1\leq j\leq m_i$. Then by Theorem~\ref{thm:decompose}, $\widehat{A}_i''=A_i''\cup\mu_i$ gives the knot $K_i$. To prove Theorem~\ref{thm:main}, we need to isotope $A_i''$ and choose $\mu_i$ suitably so that the diagram given by $p(A_i'')\cup p(\mu_i)$ is as simple as possible. Then, we give the upper bound on $c(K_1)+\cdots+c(K_n)$. There are four steps.

\vspace{5pt}

\noindent {\bf Step 1.} We can choose $\mu_i\subset H_i$ avoiding the 1-handles of $\mathcal{H}_{X,i}$ with model $Z_1$.

Recall that $H_i$ is the union of all the 1-handles and 0-handles of $\mathcal{H}_{X,i}$ and each 2-handle of $\mathcal{H}_{X,i}$ meets $H_i$ in an annulus. A 2-handle of $\mathcal{H}_{X,i}$ may intersect $A_i''$ in an arc. We first isotope the possible arc into any 1-handle of $\mathcal{H}_{X,i}$ adjacent to the 2-handle, keeping $M_i''$ invariant. Note that this 1-handle of $\mathcal{H}_{X,i}$ cannot have model $Z_1$. Otherwise, by picking the $\mu_{i,j}$, we see that $K$ is not connected. So, the model of this 1-handle can only be $Z_2$ or $Z_3$.

After applying the isotopy, the 2-handles of $\mathcal{H}_{X,i}$ are disjoint from $A_i''$. Because the intersection of each $S_{i,j}$ with all these 2-handles is a union of pairwise disjoint disks, we can pick the $\mu_{i,j}$ so that it is disjoint from the 2-handles of $\mathcal{H}_{X,i}$. So, we can let $\mu_i\subset H_i$. Actually, $\mu_i$ can also avoid the 1-handles with model $Z_1$.

Consider a 2-handle of $\mathcal{H}_{X,i}$ that meets some 1-handles with model $Z_1$. We use $B$ to denote the union of this 2-handle and all the 1-handles with model $Z_1$ which are adjacent to it. Then, $B$ is a 3-ball, and $B\cap\partial M_i''$ is a 2-sphere with holes that lies in some $S_{i,j}$. Now, each boundary circle of $B\cap\partial M_i''$ bounds a disk in $\partial B$ and a disk in $S_{i,j}$ which do not contain $B\cap\partial M_i''$. The two disks give a 2-sphere. Note that the disk in $\partial B$ meets $K$ in exactly two points. So, the disk in $S_{i,j}$ either does not meet $K$ or meets $K$ also in exactly two points, and there is only one such disk $D^2\subset S_{i,j}$ that meets $K$. See Figure~\ref{fig:AvoidZ1} for a sketch of the situation. Then, we can pick $\mu_{i,j}\subset D^2$. So we can require that $\mu_i$ is disjoint from $B$.

\begin{figure}[h]
\includegraphics{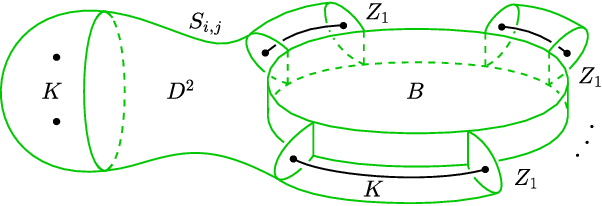}
\caption{A 2-handle and the adjacent 1-handles with model $Z_1$.}\label{fig:AvoidZ1}
\end{figure}

In general, each 1-handle with model $Z_1$ must meet some 2-handle of $\mathcal{H}_{X,i}$, and we can have many 3-balls as the above $B$. These 3-balls meet a given $S_{i,j}$ in some 2-spheres with holes. Then the complement of their union in $S_{i,j}$ consists of many connected components. By the above argument, the points in $K\cap S_{i,j}$ must lie in the same component. So we can pick $\mu_{i,j}$ in the component as before, and then $\mu_i$ is disjoint from all the 1-handles with model $Z_1$.

\vspace{5pt}

\noindent {\bf Step 2.} The part of $\mu_i$ in each 1-handle of $\mathcal{H}_{X,i}$ can have a standard form.

Now among the 1-handles of $\mathcal{H}_{X,i}$, only those 1-handles with models $Z_2$ and $Z_3$ can meet $\mu_i$. Note that such a 1-handle looks like a $Z_2$-piece or $Z_3$-piece. Actually it contains a $Z_2$-piece or $Z_3$-piece that lies in some $[-1,1]\times R_e$. If the 1-handle of $\mathcal{H}_{X,i}$ contains some arcs of $A_i''$ coming from the 2-handles of $\mathcal{H}_{X,i}$, as described in Step~1, then by identifying $R_e$ with $[-1/2,1/2]\times[-1,1]$, we can require that such arcs lie in $[-1,1]\times[-1/2,1/2]\times\{0\}$. Then for a 1-handle with model $Z_2$, there is at most one such arc. Otherwise, we have two such arcs, and we can pick two arcs of $\mu_i$ lying in $[-1,1]\times[-1/2,1/2]\times\{0\}$ and connecting their endpoints. Then, $K$ is not connected, and we get a contradiction. Similarly, for a 1-handle with model $Z_3$, there are at most two such arcs. Otherwise, we have three such arcs, and they give a trivial connected summand of $K$ if we pick three arcs of $\mu_i$ connecting their endpoints and lying in $[-1,1]\times[-1/2,1/2]\times\{0\}$. So, we also get a contradiction. Further, if a 1-handle of $\mathcal{H}_{X,i}$ with model $Z_3$ contains two arcs of $A_i''$ coming from the 2-handles of $\mathcal{H}_{X,i}$, then we can always replace such two arcs by another arc in $[-1,1]\times[-1/2,1/2]\times\{0\}$, as shown in Figure~\ref{fig:subarc}. If we pick some $\mu_{i,j}$ connecting the two arcs, then we get the required arc up to isotopy.

Hence we can require that any 1-handle of $\mathcal{H}_{X,i}$ contains at most one arc of $A_i''$ coming from the 2-handles of $\mathcal{H}_{X,i}$, and any such arc is horizontal. If the 1-handle has model $Z_2$, then we can further require that the projection of such an arc to $S_0$ is disjoint from $D$. If the 1-handle has model $Z_3$, then we can further require that the arc is as the ones shown in Figure~\ref{fig:subarc}, up to symmetry.

\begin{figure}[h]
\includegraphics{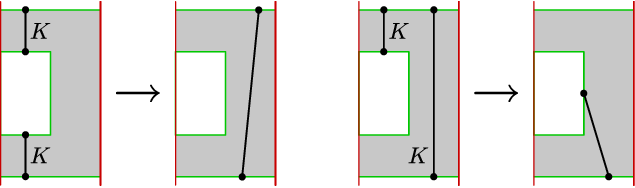}
\caption{Replace two arcs by one in $[-1,1]\times[-1/2,1/2]\times\{0\}$.}\label{fig:subarc}
\end{figure}

Below we simplify the part of $\mu_i$ in a 1-handle with model $Z_2$ or $Z_3$. As before, we assume that the 1-handle intersects $[-1,1]\times R_e$ in a $Z_2$-piece or $Z_3$-piece. The rectangle $[-1,1]\times[-1/2,1/2]\times\{0\}$ then intersects the piece in a disk. We denote this disk by $\Omega$. Then we can identify the 1-handle of $\mathcal{H}_{X,i}$ with $\Omega\times[-1,1]$, where $\Omega\times\{0\}$ corresponds to $\Omega$. Now we let $\alpha$ denote a component of $\partial\Omega\cap\partial M_i''$, which is an arc, and we consider the part of $\mu_i$ in the disk $\alpha\times[-1,1]$.

If $\alpha\times[-1,1]$ meets some $\mu_{i,j}$, then $\mu_{i,j}$ does not meet $\partial\alpha\times[-1,1]$, and we can require that it meets $\alpha\times\{-1,1\}$ transversely. Then the intersection of $\alpha\times[-1,1]$ and $\mu_{i,j}$ consists of some disjoint arcs. Let $\beta$ be such an arc. Then we can assume that $\beta$ has the form $\{s\}\times[-1,0]$ or $\{s\}\times[0,1]$ if it has one endpoint in $K$, and $\beta$ has the form $\{s\}\times[-1,1]$ if its endpoints lie in different sides of $\alpha\times[-1,1]$, up to isotopy, where $s$ is a point in $\alpha$. If its endpoints lie in the same side of $\alpha\times[-1,1]$, then $\beta$ cuts off a disk that is disjoint from $\partial\alpha\times\{-1,1\}$. We can move the part of $\mu_{i,j}$ in the disk into the adjacent 0-handle of $\mathcal{H}_{X,i}$ if the disk does not meet $K$. So in this case we can assume that the disk meets $K$, and $\beta$ consists of three straight arcs as shown in the left picture of Figure~\ref{fig:arcsur}. The picture also shows another two forms of $\beta$. Note that there can be parallel arcs with the same form.

\begin{figure}[h]
\includegraphics{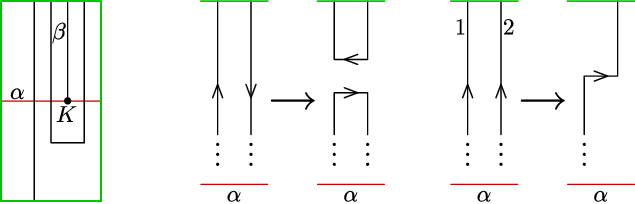}
\caption{The possible $\beta$ and the surgeries on $\mu_{i,j}$.}\label{fig:arcsur}
\end{figure}

If $\mu_{i,j}$ intersects $\alpha\times\{1\}$ in more than one points, then we apply surgeries on it as follows. Assume that in $\alpha\times\{1\}$ there are two adjacent intersection points, and each point is contained in a straight arc lying in some $\beta$ as shown in the right two pictures of Figure~\ref{fig:arcsur}, where we have given $\mu_{i,j}$ an orientation. These two oriented arcs induce orientations on $[-1,1]$. If their induced orientations are different, then we replace these two arcs by the arcs as shown in the middle picture of Figure~\ref{fig:arcsur}. Then $\mu_{i,j}$ becomes a union of an arc and a circle where the circle can be removed. Hence the number of intersection points in $\alpha\times\{1\}$ can be reduced. If the induced orientations of the two arcs are the same, then according to their order along $\mu_{i,j}$, we can replace the two arcs, together with the arc lying between them, by one arc as shown in the right picture of Figure~\ref{fig:arcsur}. Now $\mu_{i,j}$ becomes a new arc, and there are fewer intersection points in $\alpha\times\{1\}$. For $\alpha\times\{-1\}$ we have parallel results.

So we can assume that $\mu_{i,j}$ meets each of $\alpha\times\{1\}$ and $\alpha\times\{-1\}$ in at most one point. Hence $\mu_{i,j}$ meets $\alpha\times[-1,1]$ in at most one arc, which has one of the forms $\{s\}\times[-1,0]$, $\{s\}\times[0,1]$, or $\{s\}\times[-1,1]$, for some point $s$ in $\alpha$. After simplifying the part of $\mu_{i,j}$ in $\alpha\times[-1,1]$, we can repeat the process for the other components of $\partial\Omega\cap\partial M_i''$. This can affect $\mu_{i,j}\cap(\alpha\times[-1,1])$ only if the intersection is the arc having the form $\{s\}\times[-1,1]$, and this arc will be removed in this process. Hence, in each component of $\Omega\times[-1,1]\cap\partial M_i''$, there is at most one arc coming from $\mu_i$, and such an arc has one of the forms. This also holds for the other 1-handles with models $Z_2$ and $Z_3$, by repeating the process.

\begin{figure}[h]
\includegraphics{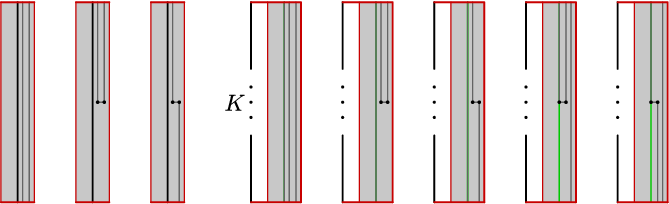}
\caption{The possible part of $\mu_i$ in a 1-handle of $\mathcal{H}_{X,i}$.}\label{fig:1-Model}
\end{figure}

In summary, we can require that the part of $\mu_i$ in each 1-handle with model $Z_2$ or $Z_3$ has a {\it standard form}, as shown in Figure~\ref{fig:1-Model}, where some of the arcs may not appear. For each 1-handle with model $Z_2$, there are at most two arcs coming from $\mu_i$, and the first three pictures show the possible two arcs. For each 1-handle with model $Z_3$, there are at most three arcs coming from $\mu_i$, as shown in the remaining five pictures. In the pictures, the arcs which do not meet $K$ may not appear. This provides the cases when there are fewer arcs. Note that in a ``curved disk'' the arc lies in the vertical part. Also, the pictures only show the projections, so in general a picture represents more than one cases, which differ by symmetries.

\vspace{5pt}

\noindent {\bf Step 3.} We simplify the part of $\mu_i$ in each 0-handle of $\mathcal{H}_{X,i}$.

Now let $B$ be a 0-handle of $\mathcal{H}_{X,i}$, and let $\Upsilon$ be a component of $B\cap\partial M_i''$. Then $\Upsilon$ is a disk, and it meets those adjacent 1-handles with models $Z_2$ and $Z_3$ in some disjoint arcs. Let $\alpha\subset\partial\Upsilon$ be the union of all such arcs. If $\Upsilon$ meets some $\mu_{i,j}$, then by Step~1 and Step~2, $\mu_{i,j}$ does not meet $\partial\Upsilon\setminus\alpha$, and it intersects each component of $\alpha$ transversely in at most one point. So, $\mu_{i,j}\cap\Upsilon$ is a union of disjoint arcs, and by suitably identifying $\Upsilon$ with some convex polygon, up to isotopy we can assume that these arcs are straight. Then as in Step~2, we can apply surgeries on $\mu_{i,j}$ if it meets $\Upsilon$ in more than one arcs. Here we choose adjacent intersection points in $\partial\Upsilon$ which lie in different arcs, and the surgery reduces the number of arcs. So, we can require that $\mu_{i,j}$ meets $\Upsilon$ in at most one arc. Then, as in Step~2, by repeating the process, we can require that $\mu_i$ meets each component of $B\cap\partial M_i''$ in at most one arc, and this holds for each 0-handle of $\mathcal{H}_{X,i}$. The result in Step~2 still holds.

Then for any 0-handle $B$ of $\mathcal{H}_{X,i}$, we can list all the possible cases of $\mu_i\cap B$ up to isotopy. The projections of the arcs in $\mu_i\cap B$ and $K\cap B$ then give crossings in the corresponding $R_v$. In each case, we can obtain the minimum of the number of crossings. However, we will not provide all these cases, since the 0-handles of $\mathcal{H}_{X,i}$ have $15$ models, and the list will then be too long. Instead, in Table~\ref{tab:0-Model}, we provide an upper bound on the minimum numbers of crossings among all those cases for a fixed model. We also provide an example which realizes the upper bound. In each picture listed in Table~\ref{tab:0-Model}, the model name and the upper bound are marked at the corners of $R_v$. One can compare the pictures with the ones in Section~\ref{sec:FDH}.

To verify the upper bounds (or list all cases), the following facts are useful.

(1) For a 0-handle $B$ with model $X$, any component of $B\cap\partial M_i''$ meets $\mu_i$ in at most one arc, and the number of components is marked in the model name $X$.

(2) For a possible arc as in (1), each endpoint of this arc either lies in $K$ or lies in an adjacent 1-handle of $\mathcal{H}_{X,i}$, where the part of $\mu_i$ has some standard form.

(3) The projection of $B$ to $S_0$ is a disk which contains the images of the arcs in $\mu_i\cap B$ and $K\cap B$. We can identify it with a convex polygon and then ``straighten the image arcs'' as much as possible. Then, in most cases, the number of crossings can be estimated by checking whether the boundaries of the image arcs are linked to each other in the boundary of the disk.

\begin{figure}[h]
\includegraphics{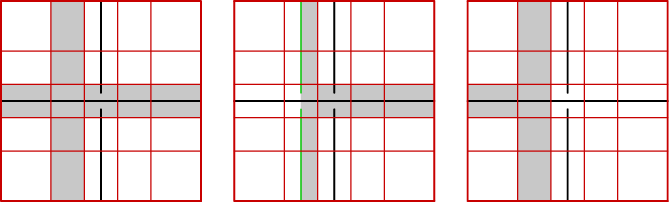}
\caption{The three disk components for the $X_3$-piece.}\label{fig:DinX3}
\end{figure}

As an example, consider a 0-handle $B$ with model $X_3$. Then $B\cap\partial M_i''$ contains three components, where the top one and the bottom one are horizontal disks and the middle one consists of two horizontal disks and one vertical disk. Assume that $B$ has the picture in Figure~\ref{fig:case1}, then the three components can be shown as in the left, right, and middle pictures in Figure~\ref{fig:DinX3}, respectively. In each of them, there is at most one arc, whose endpoints come from the arcs in adjacent 1-handles. Then we can assume that the possible arcs in the middle disk and the bottom disk have disjoint projections, because in the boundary of the projection of $B$, the arcs have unlinked pairs of endpoints. Now, if the number of crossings is minimal, then each pair of the arcs in $B$ gives at most one crossing. Since together with $K\cap B$, there are at most four arcs in $B$, the minimum number of crossings is at most $5$.

The cases of other models given in Table~\ref{tab:0-Model} can be analyzed in a similar way, or can be obtained just by enumeration. We leave the details to the readers.

\begin{table}[p]
\caption{Upper bounds on the minimal numbers of crossings.}\label{tab:0-Model}
\centerline{\includegraphics{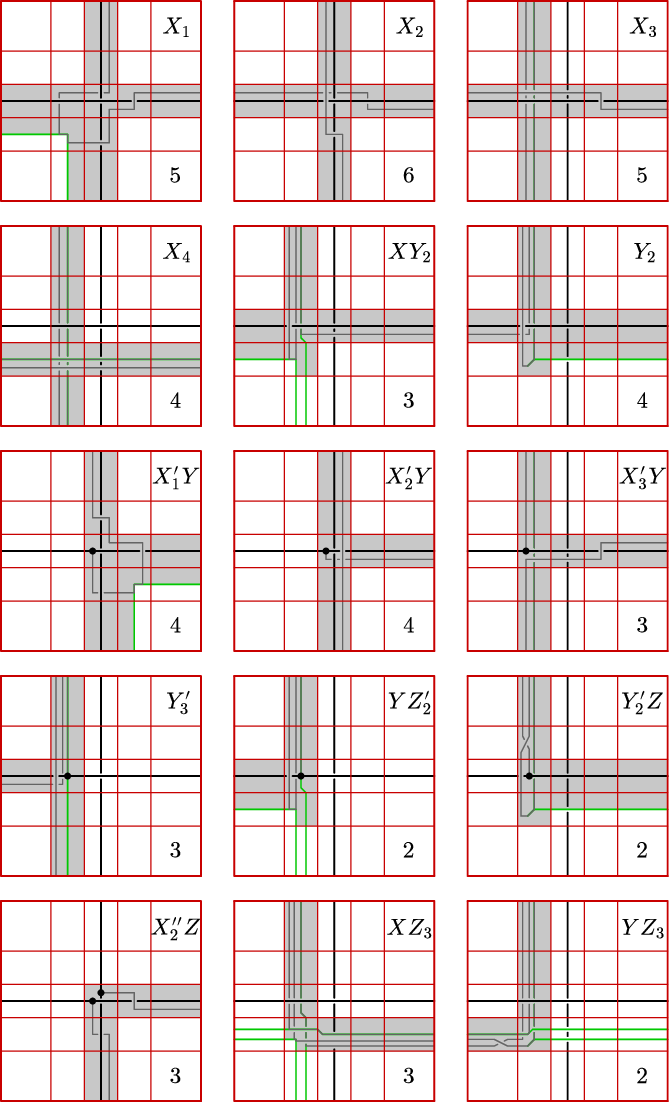}}
\end{table}

\noindent {\bf Step 4.} We show the required upper bound on $c(K_1)+\cdots+c(K_n)$. By the step, we can finish the proof of Theorem~\ref{thm:main}.

Now, for each $1\leq i\leq n$, we have modified $A_i''\subset M_i''$ and found some simple $\mu_i$ so that $p(A_i'')\cup p(\mu_i)$ gives a specific diagram of $K_i$ in $S_0$. We denote the diagram by $D_i$. Here, note that if we replaced two arcs by one as shown in Figure~\ref{fig:subarc}, then we actually first picked some $\mu_{i,j}$ and then applied an isotopy.

By Step~1 and Step~2, in any $R_f$ or $R_e$, the part of $D_i$ has no crossings. Hence, each crossing of $D_i$ lies in some $R_v$. For $1\leq i\leq n$ and a crossing $v$ in $D$, let $c(D_i)$ and $c_v(D_i)$ be the numbers of crossings in $D_i$ and $D_i\cap R_v$, respectively. Then, we have $c(D_i)=\sum_v c_v(D_i)$. So we have the following inequality
\[\sum_i c(K_i)\leq\sum_i c(D_i)=\sum_i\sum_v c_v(D_i)=\sum_v\sum_i c_v(D_i).\]
Below we first show that $\sum_i c_v(D_i)\leq 16$ holds for each crossing $v$ in $D$.

Let $D_v$ be the union of $D_i\cap R_v$, $1\leq i\leq n$. We only need to prove that $D_v$ has at most $16$ crossings if we pick all those $\mu_i$ suitably. Note that the crossings of $D_v$ come from the arcs in $\mu_i\cap B$ and $K\cap B$, where $1\leq i\leq n$, and $B$ is an $X$-piece in the 0-handle $[-1,1]\times R_v$ of $\mathcal{H}_D$. By Theorem~\ref{thm:decompose}, the $X$-pieces in the 0-handle of $\mathcal{H}_D$ are disjoint from each other, and we know that which kinds of $X$-pieces could appear simultaneously. Each $X$-piece either lies in a 1-handle of some $\mathcal{H}_{X,i}$ or the $X$-piece itself is a 0-handle of some $\mathcal{H}_{X,i}$. So we have two cases as follows.

Case (a). The $X$-piece $B$ lies in a 1-handle of $\mathcal{H}_{X,i}$.

Now $B$ has one of the five models $YZ_1$, $Z_1$, $Z_2$, $Z_3$, and $\bar{Z}_3$. If it has model $\bar{Z}_3$, then $\mu_i\cap B$ contains at most three arcs, and $K\cap B=\emptyset$. The projection of $\mu_i\cap B$ lies near a corner of $R_v$. Then, because the part of $\mu_i$ in any 1-handle of $\mathcal{H}_{X,i}$ has some standard form, the image arcs in $R_v$ meet no crossings of $D_v$. Hence, we can ignore $\bar{Z}_3$-pieces; see also Remark~\ref{rem:Z3}. If $B$ has model $YZ_1$ or $Z_1$, then $\mu_i\cap B=\emptyset$ and $K\cap B$ is one arc. In the case of $Z_2$ or $Z_3$, $\mu_i\cap B$ and $K\cap B$ together provide at most three arcs, which have parallel projections.

Case (b). The $X$-piece $B$ is a 0-handle of $\mathcal{H}_{X,i}$.

Now $B$ has one of the $15$ models listed in Table~\ref{tab:0-Model}. In each case, this table gives an upper bound on the minimum numbers of crossings given by the projections of $\mu_i\cap B$ and $K\cap B$. Note that $[-1,1]\times R_v$ contains at most two such $X$-pieces, by Theorem~\ref{thm:decompose}. If there are two such $X$-pieces, denoted by $B_1$ and $B_2$, then each $B_l$ must have model $Y_2$, $YZ_3$, or $Y_2'Z$. By the construction of the $X$-pieces, there is a twisted disk in $T_{\pm}$ that lies between $B_1$ and $B_2$. It gives a figure ``$\textsf{Y}$'' in $R_v$. Then we can require that those arcs in each $B_l$ have images disjoint from the figure ``$\textsf{Y}$'' except the possible arc $K\cap B_l$. We can further require that between those arcs in $B_1$ and those arcs in $B_2$, only $K\cap B_1$ and $K\cap B_2$ can give a crossing.

Then, according to Table~\ref{tab:models} and Table~\ref{tab:effect}, one can count the number of crossings between all possible arcs $\mu_i\cap B$ and $K\cap B$ described in Case~(a) and Case~(b). If an $X$-piece varies as in Table~\ref{tab:effect}, then the number of crossings will become smaller. Then it suffices to consider Table~\ref{tab:models}. If there are two $X$-pieces as in Case~(b), then there are at most $11$ crossings. The crossings are given by Case~2 in Table~\ref{tab:models} when there are two $YZ_3$-pieces and two $Z_1$-pieces. If there exists only one $X$-piece as in Case~(b), then there are also at most $11$ crossings. Now the crossings are given by Case~1 in Table~\ref{tab:models} when there are one $X_4$-piece and two $Z_1$-pieces. So only if there are two $Z_3$-pieces and two $Z_1$-pieces, we can obtain the upper bound $16$.

Since $\sum_i c_v(D_i)=16$ cannot hold for all $v$, we finish the proof of Theorem~\ref{thm:main}.


\bibliographystyle{amsalpha}

\end{document}